\documentclass[a4paper, oneside]{amsart}
\usepackage[utf8]{inputenc}

\title[A Weiss--Williams theorem for spaces of embeddings and the homotopy type of spaces of long knots]{A Weiss--Williams theorem for spaces of embeddings\\ and the homotopy type of spaces of long knots}
\author{Samuel Mu\~noz-Ech\'aniz}
\email{\url{sm2600@cam.ac.uk}}
\address{Centre for Mathematical Sciences,
Wilberforce Road,
Cambridge CB3 0WB, UK}

\usepackage{a4wide}
\usepackage{url}
\usepackage{lmodern,microtype}
\usepackage{amsmath,amstext,amssymb,mathrsfs,amscd,amsthm,indentfirst}
\usepackage{leftidx}
\usepackage{amsfonts}
\usepackage{enumerate}
\usepackage{float}
\usepackage{bbm,dsfont}
\usepackage{diagbox, multicol, multirow}

\usepackage[dvipsnames,svgnames,x11names,hyperref]{xcolor}

\usepackage[colorlinks,citecolor=PineGreen,linkcolor=Mahogany,urlcolor=blue,filecolor=Mahogany]{hyperref}

\usepackage[utf8]{inputenc}
\usepackage[mathscr]{eucal}

\usepackage{subcaption}
\usepackage{mathtools}
\usepackage{graphicx}
\usepackage{graphics}
\graphicspath{ {Figures/} }
\usepackage{tikz-cd}
\usepackage{pifont}
\usepackage[inline]{enumitem}
\usepackage{scalerel}
\tikzcdset{scale cd/.style={every label/.append style={scale=#1},
    cells={nodes={scale=#1}}}}

\usepackage[leqno]{amsmath}

\makeatletter
\newtheorem*{rep@theorem}{\rep@title}
\newcommand{\newreptheorem}[2]{%
\newenvironment{rep#1}[1]{%
 \def\rep@title{#2 \ref{##1}}%
 \begin{rep@theorem}}%
 {\end{rep@theorem}}}
\makeatother

\newtheorem{theorem}{Theorem}[section]
\newreptheorem{theorem}{Theorem}

\newreptheorem{lemma}{Lemma}
\newtheorem{proposition}[theorem]{Proposition}
\newreptheorem{proposition}{Proposition}

\DeclareFontFamily{OT1}{ptm}{}
\DeclareFontShape{OT1}{ptm}{m}{n} { <-> ptmr}{}
\DeclareFontShape{OT1}{ptm}{m}{it}{ <-> ptmri}{}
\DeclareFontShape{OT1}{ptm}{m}{sl}{ <->ptmro}{}
\DeclareFontShape{OT1}{ptm}{m}{sc}{ <-> ptmrc}{}
\DeclareFontShape{OT1}{ptm}{b}{n} { <-> ptmb}{}
\DeclareFontShape{OT1}{ptm}{b}{it}{ <-> ptmbi}{}     
\DeclareFontShape{OT1}{ptm}{bx}{n} {<->ssub * ptm/b/n}{}
\DeclareFontShape{OT1}{ptm}{bx}{it}{<->ssub * ptm/b/it}{}

\DeclareSymbolFont{bold}{OT1}{ptm}{b}{n}
\DeclareMathAlphabet{\mathbf}{OT1}{ptm}{b}{n}  
\DeclareMathAlphabet{\mathrm}{OT1}{ptm}{m}{n}

\DeclareFontFamily{OT1}{psy}{}      
\DeclareFontShape{OT1}{psy}{m}{n}{ <-> s * [0.9] psyr}{}
\DeclareFontFamily{OMS}{ptm}{}     
\DeclareFontShape{OMS}{ptm}{m}{n}{ <8> <9> <10> gen * cmsy }{}
\DeclareFontFamily{OMS}{cmtt}{}     
\DeclareFontShape{OMS}{cmtt}{m}{n}{ <8> <9> <10> gen * cmsy }{}

\SetSymbolFont{operators}{normal}{OT1}{ptm}{m}{n}   
\SetSymbolFont{operators}{bold}{OT1}{ptm}{b}{n}     
\DeclareSymbolFont{emsy}{OT1}{ptm}{m}{it}
\DeclareSymbolFont{emsr}{OT1}{ptm}{m}{n}
\DeclareSymbolFont{emcmr}{OT1}{cmr}{m}{n}   
\DeclareSymbolFont{emsymb}{OT1}{psy}{m}{n}  
\DeclareMathSymbol a{\mathalpha}{emsy}{"61}
\DeclareMathSymbol b{\mathalpha}{emsy}{"62}
\DeclareMathSymbol c{\mathalpha}{emsy}{"63}
\DeclareMathSymbol d{\mathalpha}{emsy}{"64}
\DeclareMathSymbol e{\mathalpha}{emsy}{"65}
\DeclareMathSymbol f{\mathalpha}{emsy}{"66}
\DeclareMathSymbol g{\mathalpha}{emsy}{"67}
\DeclareMathSymbol h{\mathalpha}{emsy}{"68}
\DeclareMathSymbol i{\mathalpha}{emsy}{"69}
\DeclareMathSymbol j{\mathalpha}{emsy}{"6A}
\DeclareMathSymbol k{\mathalpha}{emsy}{"6B}
\DeclareMathSymbol l{\mathalpha}{emsy}{"6C}
\DeclareMathSymbol m{\mathalpha}{emsy}{"6D}
\DeclareMathSymbol n{\mathalpha}{emsy}{"6E}
\DeclareMathSymbol o{\mathalpha}{emsy}{"6F}
\DeclareMathSymbol p{\mathalpha}{emsy}{"70}
\DeclareMathSymbol q{\mathalpha}{emsy}{"71}
\DeclareMathSymbol r{\mathalpha}{emsy}{"72}
\DeclareMathSymbol s{\mathalpha}{emsy}{"73}
\DeclareMathSymbol t{\mathalpha}{emsy}{"74}
\DeclareMathSymbol u{\mathalpha}{emsy}{"75}
\DeclareMathSymbol v{\mathalpha}{emsy}{"76}
\DeclareMathSymbol w{\mathalpha}{emsy}{"77}
\DeclareMathSymbol x{\mathalpha}{emsy}{"78}
\DeclareMathSymbol y{\mathalpha}{emsy}{"79}
\DeclareMathSymbol z{\mathalpha}{emsy}{"7A}
\DeclareMathSymbol A{\mathalpha}{emsy}{"41}
\DeclareMathSymbol B{\mathalpha}{emsy}{"42}
\DeclareMathSymbol C{\mathalpha}{emsy}{"43}
\DeclareMathSymbol D{\mathalpha}{emsy}{"44}
\DeclareMathSymbol E{\mathalpha}{emsy}{"45}
\DeclareMathSymbol F{\mathalpha}{emsy}{"46}
\DeclareMathSymbol G{\mathalpha}{emsy}{"47}
\DeclareMathSymbol H{\mathalpha}{emsy}{"48}
\DeclareMathSymbol I{\mathalpha}{emsy}{"49}
\DeclareMathSymbol J{\mathalpha}{emsy}{"4A}
\DeclareMathSymbol K{\mathalpha}{emsy}{"4B}
\DeclareMathSymbol L{\mathalpha}{emsy}{"4C}
\DeclareMathSymbol M{\mathalpha}{emsy}{"4D}
\DeclareMathSymbol N{\mathalpha}{emsy}{"4E}
\DeclareMathSymbol O{\mathalpha}{emsy}{"4F}
\DeclareMathSymbol P{\mathalpha}{emsy}{"50}
\DeclareMathSymbol Q{\mathalpha}{emsy}{"51}
\DeclareMathSymbol R{\mathalpha}{emsy}{"52}
\DeclareMathSymbol S{\mathalpha}{emsy}{"53}
\DeclareMathSymbol T{\mathalpha}{emsy}{"54}
\DeclareMathSymbol U{\mathalpha}{emsy}{"55}
\DeclareMathSymbol V{\mathalpha}{emsy}{"56}
\DeclareMathSymbol W{\mathalpha}{emsy}{"57}
\DeclareMathSymbol X{\mathalpha}{emsy}{"58}
\DeclareMathSymbol Y{\mathalpha}{emsy}{"59}
\DeclareMathSymbol Z{\mathalpha}{emsy}{"5A}
\DeclareMathSymbol{\bullet}{\mathalpha}{emsymb}{"B7}
\def\Bullet{\leavevmode\unkern{$\m@th\bullet$}\kern.32em\ignorespaces}
\DeclareMathSymbol +{\mathbin}{emcmr}{`+}
\DeclareMathSymbol ={\mathrel}{emcmr}{`=}  
\DeclareMathSymbol{\Gamma}{\mathalpha}{emcmr}{"00}
\DeclareMathSymbol{\Delta}{\mathalpha}{emcmr}{"01}
\DeclareMathSymbol{\Lambda}{\mathalpha}{emcmr}{"03}
\DeclareMathSymbol{\Xi}{\mathalpha}{emcmr}{"04}
\DeclareMathSymbol{\Pi}{\mathalpha}{emcmr}{"05}
\DeclareMathSymbol{\Sigma}{\mathalpha}{emcmr}{"06}
\DeclareMathSymbol{\Upsilon}{\mathalpha}{emcmr}{"07}
\DeclareMathSymbol{\Phi}{\mathalpha}{emcmr}{"08}
\DeclareMathSymbol{\Psi}{\mathalpha}{emcmr}{"09}
\DeclareMathSymbol{\Omega}{\mathalpha}{emcmr}{"0A}
\DeclareMathSizes{7.6}{8}{6}{5}
\DeclareMathAccent{\dot}{\mathalpha}{operators}{"C7} 
\DeclareMathOperator{\hocolim}{\operatorname{hocolim}}

\usepackage[bbgreekl]{mathbbol}
\DeclareSymbolFontAlphabet{\mathbb}{AMSb}
\DeclareSymbolFontAlphabet{\mathbbl}{bbold}

\usepackage{booktabs}
\usepackage{verbatim}

\newcommand{\diff}{\mathrm{Diff}}
\newcommand{\bdiff}{\widetilde{\mathrm{Diff}}}

\newcommand{\hcob}{\hspace{2pt}\ensuremath \raisebox{-1pt}{$\overset{h}{\leadsto}$}\hspace{2pt}}

\newcommand{\emb}{\operatorname{Emb}}
\newcommand{\bemb}{\widetilde{\operatorname{Emb}}}
\newcommand{\hofib}{\operatorname{hofib}}
\newcommand{\cemb}{C\mathrm{Emb}}
\newcommand{\vsimeq}{\rotatebox[origin=c]{-90}{$\simeq$}}

\newcommand{\vsim}{\rotatebox[origin=c]{-90}{$\sim$}}

\newcommand{\vapprox}{\rotatebox[origin=c]{-90}{$\approx$}}

\newcommand{\Id}{\mathrm{Id}}

\newcommand{\map}{\operatorname{Map}}
\newcommand{\Sing}{\operatorname{Sing}}

\newcommand{\sSet}{\mathsf{sSet}}

\newcommand{\Top}{\mathsf{Top}}

\newcommand{\rmTop}{\mathrm{Top}}

\newtheorem{thm}{Theorem}[section]
\newtheorem*{thm*}{Theorem}
\newtheorem*{lem*}{Lemma}

\newtheorem{cor}[thm]{Corollary}
\newtheorem{lem}[thm]{Lemma}
\newtheorem*{cl}{Claim}
\newtheorem{claim}{Claim}
\newtheorem*{question*}{\textbf{Question}}
\newtheorem{thmx}{Theorem}
\newtheorem{corx}[thmx]{Corollary}
\newtheorem{notn}[thm]{Notation}
\newtheorem{conv}[thm]{Convention}

\newcommand*\circled[1]{\tikz[baseline=(char.base)]{%
            \node[shape=circle,draw,inner sep=1pt, scale = 0.95] (char) {#1};}}

\newtheorem{prop}[thm]{Proposition}

\newtheorem{defn}[thm]{Definition}

\theoremstyle{remark}
\newtheorem{exm}[thm]{Example}
\newtheorem{rem}[thm]{Remark}

\newtheorem{warn}[thm]{Warning}
\newtheorem{step}{Step}

\theoremstyle{definition}

\newcommand{\norm}[1]{\left\lVert #1 \right\rVert }

\newcommand{\R}{{\scaleobj{0.97}{\mathbb{R}}}}
\newcommand{\C}{\mathbb{C}}

\newcommand{\Z}{{\scaleobj{0.97}{\mathbb{Z}}}}
\newcommand{\Q}{\mathbb{Q}}

\newcommand{\geqs}{\scaleobj{0.85}{\geq}}
\newcommand{\leqs}{\scaleobj{0.85}{\leq}}

\newcommand{\s}{\mathbb{S}}

\newcommand{\im}{\mathrm{Im}\ }
\newcommand{\Wh}{\mathrm{Wh}}

\newcommand{\bdiffb}{\widetilde{\mathrm{Diff}}{\vphantom{\mathrm{Diff}}}^{\hspace{1pt}b}}
\newcommand{\Asp}{\mathbf{A}}

\newcommand{\bWhsp}{\widetilde{\mathbf{Wh}}{\vphantom{\mathrm{Wh}}}^{\diff}}

\newcommand{\pemb}{\operatorname{Emb}^{(\sim)}}

\newcommand{\Hsp}{\mathbf{H}}
\newcommand{\CEsp}{\mathbf{CE}}
\newcommand{\Whsp}{\mathbf{Wh}^{\diff}}

\newcommand{\Whsptop}{\mathbf{Wh}^{\mathrm{Top}}}

\numberwithin{equation}{section}

\newsavebox{\pullback}
\sbox\pullback{%
\begin{tikzpicture}%
\draw (0,0) -- (1ex,0ex);%
\draw (1ex,0ex) -- (1ex,1ex);%
\end{tikzpicture}}

\newsavebox{\pushout}
\sbox\pushout{%
\begin{tikzpicture}%
\draw (0,0) -- (0ex,1ex);%
\draw (0ex,1ex) -- (1ex,1ex);%
\end{tikzpicture}}

\newcommand{\newabstract}[1]{%
  \par\bigskip
  \csname otherlanguage*\endcsname{#1}%
  \csname captions#1\endcsname
  \item[\hskip\labelsep\scshape\abstractname.]
}
\usepackage[english]{babel}

\begin{document}
\begin{abstract}
We establish a pseudoisotopy result for embedding spaces in the line of that of Weiss and Williams for diffeomorphism groups. In other words, for $P\subset M$ a codimension at least three embedding, we describe the difference in a range of homotopical degrees between the spaces of block and ordinary embeddings of $P$ into $M$ as a certain infinite loop space involving the relative algebraic $K$-theory of the pair $(M,M-P)$. This range of degrees is the so-called concordance embedding stable range, which, by recent developments of Goodwillie--Krannich--Kupers, is far beyond that of the aforementioned theorem of Weiss--Williams. 

We use this result to obtain split fibre sequences in the concordance embedding stable range, with explicit, analysable base and fibre, which determine the homotopy type of spaces of long knots of codimension at least 3. This leads to explicit computations of homotopy groups, including torsion information, in that range. In doing so, we carry out an extensive analysis of certain geometric involutions in algebraic $K$-theory that may be of independent interest.

\end{abstract}
\maketitle

{
  \hypersetup{linkcolor=black}
  \tableofcontents
}

\section{Introduction}

The classical approach to study the homotopy type of the diffeomorphism group $\diff_\partial(M)$ of a compact, possibly with boundary, high-dimensional manifold $M^d$ (ie $d\geq 5$) is based on the so called \textit{surgery-pseudoisotopy} program, which focuses on the homotopy fibre sequence
\begin{equation}
\begin{tikzcd}\label{bdiffmoddiffsequence}
(\bdiff/\diff)_\partial(M)\rar & B\diff_\partial(M)\rar["i"] & B\bdiff_\partial(M).
\end{tikzcd}
\end{equation}
The right-hand term is the classifying space of the simplicial group $\bdiff_\partial(M)_\bullet$ of \textit{block diffeomorphisms} of $M$ (cf Definition \ref{simplicialboundeddefn}), an approximation to the ordinary diffeomorphism group of $M$ that closely resembles the behaviour of the topological monoid $h\mathrm{Aut}_\partial(M)$ of homotopy automorphisms of $M$; for instance, one of the defining properties of $\bdiff_\partial(-)$ is that it satisfies a natural equivalence $\bdiff_\partial(M\times I)\simeq \Omega\bdiff_\partial(M)$, which also holds for $h\mathrm{Aut}_\partial(-)$ (but is very much not true for $\diff_\partial(-)$). \textit{Surgery theory}, as developed by Browder, Novikov, Ranicki, Sullivan, Wall, et al., roughly studies the difference between $\bdiff_\partial(M)$ and $h\mathrm{Aut}_\partial(M)$ in terms of the relative mapping space $(G/O)_*^{(M,\partial M)}$ and the quadratic $L$-theory of the integral group ring $\Z[\pi_1(M)]$ of the fundamental group(oid) of $M$, making the homotopy type of $\bdiff_\partial(M)$ theoretically accessible via homotopy theory and $L$-theory.

It is in understanding the homotopy type of $(\bdiff/\diff)_\partial(M)$, the homotopy fibre of the map $i$, where \textit{pseudoisotopy theory} \cite{IgusaConcordanceStability,HatcherWag} comes into play. Originally, this theory was concerned with the study of the topological group $C(M)$ of \textit{concordances} or \textit{pseudoisotopies} of $M$ consisting of diffeomorphisms of $M\times I$ that restrict to the identity on a neighbourhood of $M\times\{0\}\cup \partial M\times I$; in other words, pseudoisotopies of $M$ are precisely the automorphisms of the trivial $h$-cobordism starting at $M$. The spaces $C(M)$ and $\bdiff/\diff(M)$ are intimately related, as was first made precise by Hatcher \cite[Proposition 2.1]{HatcherSSeq} through a spectral sequence. Hatcher's realisation eventually evolved into the following celebrated theorem of Weiss and Williams \cite[Theorem A]{WWI}, which will be of central importance all throughout this paper.

\begin{thm}[Weiss--Williams]\label{WWbdiffmoddiff} Let $M^d$ be a compact smooth $d$-manifold. There exists a map
$$
\Phi^\diff: (\bdiff/\diff)_\partial(M)\longrightarrow \Omega^{\infty}\left(\Hsp^s(M)_{hC_2}\right)
$$
which is $(\phi(d)+1)$-connected\footnote{Recall that a map of spaces is said to be \textit{$n$-connected} if, for every choice of basepoint in the domain, it induces isomorphisms on homotopy groups in degrees $*<n$ and a surjection in degree $*=n$.}, where $\phi(d)$ denotes the concordance stable range of dimension $d$ (which by Igusa's theorem \cite{IgusaConcordanceStability} is at least $\min(\frac{d-4}{3},\frac{d-7}{2})$).  
\end{thm}

The $C_2$-spectrum $\Hsp^s(M)$ is the $1$-connective cover of a spectrum $\Hsp(M)$ which is roughly built out of deloopings of spaces of smooth $h$-cobordisms (cf Notation \ref{FEBorthfunctorsnotn} and Section \ref{hcobsection}), and whose involution corresponds (up to a minus sign) to ``reversing the direction of an $h$-cobordism''. This latter spectrum has a close connection to algebraic $K$-theory: its infinite loop space $\Omega^\infty \Hsp(M)$ is equivalent to the (smooth) \textit{stable $h$-cobordism space} $\mathcal{H}(M)$ (cf Remark \ref{justificationrem2}), which fits in a fibre sequence of spaces \cite{WJR}
\begin{equation}\label{WaldhausenWhFibseq}
\begin{tikzcd}
\mathcal{H}(M)\rar & Q_+M:=\Omega^\infty\Sigma^\infty_+M\rar["\nu"] & A(M),
\end{tikzcd}
\end{equation}
where $A(M)$ denotes Waldhausen's $A$-theory space of $M$ (cf \cite{WaldhausenAtheory}). This sequence is natural in codimension-zero embeddings, and the map $\nu$ is (naturally) a split injection. Even though the homotopy type of $A(M)$ and its involutions are difficult to understand in general, much can be said when $M$ is homotopy equivalent to a point \cite{Rognesprime2, RognesOddprimes, Blumberg2019} or the circle \cite{HesselholtWhiteheadS1}, or when working rationally \cite{BurgheleaFiedII}.

\subsection{The surgery-pseudoisotopy program for spaces of embeddings}\label{SurgPseudosection}
The homotopy type of embedding spaces is intrinsically tied to that of diffeomorphism groups as seen via the \textit{isotopy extension theorem}. More precisely, let $M^d$ be as before and let $\iota: P\xhookrightarrow{} M$ be a compact submanifold that meets $\partial M$ transversely. Write $\emb_{\partial_0}(P,M)$ for the space of smooth embeddings of $P$ into $M$ which agree with $\iota$ in a neighbourhood of $\partial_0 P:= P\cap \partial M$ and send $\partial P- \partial_0 P$ to the interior of $M$. Then there is a homotopy fibre sequence
\begin{equation}\label{firstisotopyextensionseq}
\begin{tikzcd}
\emb_{\partial_0,\langle\iota\rangle}(P,M)\rar & B\diff_\partial(M-\nu P)\rar &B\diff_\partial(M),
\end{tikzcd}
\end{equation}
where $\nu P$ is a small tubular neighbourhood of the standard embedding $\iota: P\xhookrightarrow{}M$, and the subscript $\langle\iota\rangle$ stands for the collection of components of $\emb_{\partial_0}(P,M)$ that are hit by the restriction map $\iota^*: \diff_\partial(M)\to \emb_{\partial_0}(P,M)$ that sends a diffeomorphism $\phi$ to $\phi\circ\iota$. In this sense, embedding spaces are the corresponding ``relative analogues'' of diffeomorphism groups, and often their homotopy types become easier to study.

In this paper we would like to advertise a direct approach for studying the homotopy type of embedding spaces (in a range of degrees) which is analogous to the one for diffeomorphism groups previously surveyed. As before, one first analyses the space of block embeddings $\bemb_{\partial_0}(P,M)$ via relative surgery methods; the main result in this direction is due to Browder--Casson--Sullivan--Wall (cf \cite[Theorem 2.2.1]{GoodwillieKleinWeiss}), and asserts that, as long as the codimension of $P\subset M$ is at least $3$, then the space of block embeddings is the homotopy pullback of a diagram involving so called \textit{Poincaré block embeddings} and \emph{immersions} and ordinary block immersions. Due to the Smale--Hirsch immersion theorem, all the ingredients that come into the mix are accessible through homotopy theory and thus, up to extensions, so are block embeddings.

Following the same strategy as for the classical surgery-pseudoisotopy program, it remains to understand the difference between ordinary and block embeddings, ie the homotopy fibre
\begin{equation}\label{pseudoembdefn}
\emb^{(\sim)}_{\partial_0}(P,M):=\hofib_\iota(\emb_{\partial_0}(P,M)\to \bemb_{\partial_0}(P,M)),
\end{equation}
by means of pseudoisotopy theory. This space also fits in another homotopy fibre sequence
\begin{equation}\label{bdiffmoddiffisotopyextension}
\begin{tikzcd}
    \emb^{(\sim)}_{\partial_0}(P,M)\rar & (\bdiff/\diff)_\partial(M-\nu P)\rar["\mu"] & (\bdiff/\diff)_\partial(M)
\end{tikzcd}
\end{equation}
obtained as the fibre of the map from (\ref{firstisotopyextensionseq}) to its block analogue (see (\ref{blockIET})). What was previously fulfilled by Theorem \ref{WWbdiffmoddiff} for the pseudoisotopy part of diffeomorphism groups seems to be missing in the case of embedding spaces; the best result known in this direction is \textit{Morlet's lemma of disjunction} \cite[Theorem 3.1]{BurgLashRoth} which, in that reformulation, determines the connectivity of the map $\mu$. 

Our first main result fills in this gap in the surgery-pseudoisotopy program for embedding spaces, and describes the homotopy type of $\emb_{\partial_0}^{(\sim)}(P,M)$ in a range outside of the connectivity of $\mu$.

\begin{thmx}\label{EmbWWIThm}
There exists a map 
$$
\Phi^{\emb}: \emb^{(\sim)}_{\partial_0}(P,M)\longrightarrow \Omega^{\infty}\left(\CEsp(P, M)_{hC_2}\right)
$$
which is $\phi_{\cemb}(d,p)$-connected if $d\geq 4$ and the handle dimension $p$ of $P$ relative to $\partial_0P$ satisfies $p\leq d-3$. Here $\phi_{\cemb}$ is the concordance embedding stable range (see (\ref{concembstablerange})) and
$$
\CEsp(P, M):=\operatorname{hofib}(\Hsp(M-\nu P)\to \Hsp(M)).
$$
\end{thmx}

That is, under the assumptions of the statement, the homotopy type (eg homotopy/homology groups) of $\emb^{(\sim)}_{\partial_0}(P,M)$ in degrees $* < \phi_{\cemb}(d,p)$ agrees with that of the infinite loop space $\Omega^{\infty}(\CEsp(P,M)_{hC_2})$.

\begin{rem}
The involutions in the $h$-cobordism spectra involved in the statement of Theorem \ref{EmbWWIThm} are exactly those of Theorem \ref{WWbdiffmoddiff}, which naturally arise from Weiss' orthogonal calculus (see Sections \ref{OrthReviewSection} and \ref{ThetaO1appendix}). When $M$ is stably parallelisable (and localising away from $2$), we relate these involutions to well-known algebraic ones in Theorem \ref{absolutepropTWWvsTepsilon} and Corollary \ref{propTWWvsTepsilon} (cf Notation \ref{involutionnotn} for conventions). See also Corollary \ref{milnorinvcor} for the effect of these involutions on $\pi_0^s(\Hsp(M))=\Wh(\pi_1(M))$ in terms of Milnor's involution \cite{MilnorWhiteheadTorsion}.
\end{rem}

\begin{rem}[Splitting results and the Gromoll filtration]\label{IntroRemGromoll}
    As a consequence of Theorem \ref{EmbWWIThm}, we establish a splitting result (cf Theorem \ref{splittingthm}) for embedding spaces of manifolds containing interval factors reminiscent of work of Burghelea--Lashof \cite[Corollary E]{BurgLash}. This has remarkable consequences for the Gromoll filtration of embedding spaces (cf Definition \ref{GromollDefn} and Corollary \ref{CorGromollFiltration}).
\end{rem}

\begin{rem}[Topological version of Theorem \ref{EmbWWIThm}]\label{remembWWI}
Theorem \ref{WWbdiffmoddiff} and the results of \cite{WWI} admit topological analogues. Likewise, a topological version of Theorem \ref{EmbWWIThm} holds after some adjustments. First, smooth embedding spaces must be replaced by spaces of locally flat topological embeddings, and the $h$-cobordism spectra by their topological versions. Second, the result is only valid when $P$ has geometric codimension zero in $M$, since a topological analogue of Proposition \ref{positivecodimprop} cannot hold (see Remark \ref{codimzerotopremark}). Third, when $d=4$ and $CAT=\mathrm{Top}$, we additionally require $M$ to be $1$-connected (cf\ Remark \ref{HudsonTopRem}). One should also bear in mind, however, that the bound~(\ref{concembbound}) is, a priori, only valid for smoothable topological manifolds. See Remarks \ref{HudsonTopRem}, \ref{topIETrem}, and \ref{TOPsimplicialrem} for modified arguments in the topological setting. A $PL$ analogue of Theorem \ref{EmbWWIThm} should also exist after similar adjustments, though we leave the details to the reader (see also Warning \ref{SimplicialWarning}).
\end{rem}


We highlight two remarkable features of Theorem~\ref{EmbWWIThm} that make it particularly well-suited for computations.

\subsubsection{The concordance embedding stable range}\label{concembsection}
Fix $\iota: P\xhookrightarrow{} M$ as before. A \textit{concordance embedding} of $P$ into $M$ is an embedding $\varphi: P\times I\xhookrightarrow{}M\times I$ such that
    \begin{itemize}[itemsep=2pt, leftmargin=+.4in]
        \item[($a$)] $\varphi^{-1}(M\times\{i\})=P\times\{i\}$ for $i=0,1$ and
        \item[($b$)] $\varphi$ agrees with the inclusion $\iota\times \mathrm{Id}_I$ on a neighbourhood of $P\times \{0\}\cup \partial_0 P\times I$.
    \end{itemize}\vspace{1pt}
    We denote by $\cemb(P,M)$ the space of all such embeddings, topologised as a subspace of $\emb(P\times I, M\times I)$. There are stabilisation maps 
    $$\Sigma: \cemb(P,M)\to \cemb(P\times I, M\times I)$$
    given by taking the product of an embedding with $I$ and unbending corners  appropriately (cf \cite[Fig.~1]{ConcEmbStablerange}), and the \textit{concordance embedding stable range} of the pair $(M,P)$, denoted $\phi_{\cemb}(M,P)$, is the largest integer $k$ such that all the stabilisations in
$$
\begin{tikzcd}
    \cemb(P,M)\rar["\Sigma"] & \cemb(P\times I, M\times I)\rar["\Sigma"] & \cemb(P\times I^{2}, M\times I^{2})\rar["\Sigma"] &\dots
\end{tikzcd}
$$
    are $k$-connected. Then, the \textit{concordance stable range} for a tuple $(d,p)$ is
    \begin{equation}\label{concembstablerange}
    \phi_{\cemb}(d,p):=\min\left\{\phi_{\cemb}(M,P) : \text{$\dim M=d$ and \textit{h}-$\dim(P,\partial_0P)=p$}\right\}.
    \end{equation}
    Here, \textit{h}-$\dim(P,\partial_0P)$ denotes the \emph{handle dimension} of the pair $(P,\partial_0P)$ which is, by definition, the smallest number $p$ such that $P$ can be built from a closed collar on $\partial_0 P$ by attaching handles of index at most $p$. Goodwillie--Krannich--Kupers \cite{ConcEmbStablerange} have recently shown that if $p\leq d-3$, then 
    \begin{equation}\label{concembbound}
    \phi_{\cemb}(d,p)\geq 2d-p-5,\end{equation}
    which is far beyond Igusa's lower bound for the concordance stable range $\phi(d)$. In the concordance stable range $\phi(d)$, Theorem \ref{EmbWWIThm} is a consequence of Theorem \ref{WWbdiffmoddiff} and the isotopy extension sequence (\ref{bdiffmoddiffisotopyextension}), so our main contribution is improving the connectivity of the map $\Phi^{\emb}$ to the concordance embedding stable range $\phi_{\cemb}(d,p)$. We will also see in Remark \ref{rationallongknotsrem} that the bound (\ref{concembbound}) is sharp (this was apparent in \cite{ConcEmbStablerange}).

\begin{rem}
    In fact, our proof will show that the map $\Phi^{\emb}$ of Theorem \ref{EmbWWIThm} is $\phi_{\cemb}(M,P)$-connected under the codimension assumption $p\leq d-3$.
\end{rem}

\subsubsection{Relative algebraic \texorpdfstring{$K$}{K}-theory via trace methods}\label{relAtheorysubsubsection} Given a map of spaces $Y\to X$, let $A(Y\to X)$ denote the homotopy fibre of the induced map $A(Y)\to A(X)$. By (\ref{WaldhausenWhFibseq}), there is an equivalence of spaces
$$
\Omega A(M-P\to M)\simeq \Omega^{\infty}\CEsp(P,M)\times \Omega^2 Q(M/M-P).
$$
The codimension assumption on the embedding $\iota: P\subset M$ in the statement of Theorem \ref{EmbWWIThm} guarantees that the inclusion $M-P\to M$ is $2$-connected. This can be used to our advantage, as firstly it ensures that the spectrum $\CEsp(P,M)$ is connective (see Lemma \ref{ThetaF1nconnectivity}), but more importantly that, via \textit{trace methods}, the homotopy type of $A(M-P\to M)$ is far more accessible than that of $A(M-P)$ and $A(M)$ on their own.

\textit{Trace methods}  are concerned with the study of \emph{topological cyclic homology} \cite{BokHsiMadTC}, denoted $TC(-)$, and related invariants as an approximation to algebraic $K$-theory. This is something sensible to do by the seminal work of Dundas--Goodwillie--McCarthy \cite{DundasMcCarthy, Dundas1997, DundasGoodwillieMcCarthy}, who showed that the cyclotomic trace map provides an equivalence of relative theories $A(Y\to X)\simeq TC(Y\to X)$, so long as $Y\to X$ is $2$-connected. The treatment of TC by Nikolaus--Scholze \cite{NikolausScholze} provides even further computational control of this invariant. In the cases we are concerned with (spherical group rings), the homotopy type of TC was fully described by Bökstedt--Hsiang--Madsen \cite{BokHsiMadTC} in terms of the stable homotopy of the free loop space $L(-):=\mathrm{Map}(S^1,-)$ together with its natural $S^1$-action and cyclotomic structure. When working over the field of rational numbers, this whole story simplifies even further by Goodwillie's isomorphism \cite{GoodwillierelativeAtheory} 
\begin{equation}\label{rationalisomorphismAtheory}
\pi_*(A(Y\to X))\otimes \Q\cong HC_{*}(\Omega X, \Omega Y; \Q)\cong H^{S^1}_{*}(LX,LY;\Q), 
\end{equation}
where $HC_*$ denotes Connes' cyclic homology, and $H^{S^1}_*$ stands for the $S^1$-equivariant homology. 

So as we have just seen, the homotopy type of (the infinite loop space of) the connective spectrum $\CEsp(P,M)$ of Theorem \ref{EmbWWIThm} is pretty accessible in general. However, one still needs to deal with the involution appearing in the statement in order to apply the result, which is a rather technical task that, so far, had only been carried out rationally by Bustamante--Farrell--Jiang \cite{InvolutionAtheoryBustamanteFarrell} (they relate this involution to one on the right hand side of (\ref{rationalisomorphismAtheory})). Integrally, one has to proceed with more care; our analysis in Section \ref{pseudoisotopyknotsection} deals with this issue localised away from $2$ and when $M$ is stably parallelisable.

\subsection{The homotopy type of spaces of long knots} The homology and homotopy of \emph{spaces of long knots} $\emb_\partial(D^p,D^d)$ has been subject to extensive reasearch in recent years, especially through the lens of embedding calculus and its relation to the little disks operad and graph complexes. See for instance Volić \cite{VolicLongKnots}, Watanabe \cite{WatanabeLongKnots}, Sinha and Scannel \cite{ScannellSinha, DevPSinha2009}, Budney and Cohen \cite{BudneyIntegralLongKnots, BudneyCohen} for when $p=1$ and $d=3,4$ mainly, or more modern treatments as in Arone--Turchin \cite{AroneTurchinHomologyLK, AroneTurchinHomotopyLK}, Dwyer--Hess \cite{DwyerHess}, Boavida de Brito--Weiss \cite{BoavidadeBritoWeissLK}, at last culminating in the work of Fresse--Turchin--Willwacher \cite{FresseTurchinWillwacher} where a complete description of $\pi_*(\emb_\partial(D^p,D^d))\otimes \Q$ is given in terms of the homology of the \textit{hairy graph complex}. See also Boavida de Brito--Horel \cite{BritoHorelLongKnots} for some torsion computations in the homotopy groups of spaces of long knots when $p=1$.

The second main result of this paper is a full description of the homotopy type of $\emb_\partial(D^p, D^d)$ for $d-p\geq 3$ roughly in the concordance embedding stable range (\ref{concembstablerange}) and localised away from $2$. This is done by analysing the homotopy fibre sequence
\begin{equation}\label{knotfibseq1}
\begin{tikzcd}
\pemb_\partial(D^p,D^d)\rar & \emb_\partial(D^p, D^d)\rar & \bemb_{\partial}(D^p,D^d)
\end{tikzcd}
\end{equation}
following the surgery-pseudoisotopy program for embedding spaces surveyed in Section \ref{SurgPseudosection}, a crucial step of which is Theorem \ref{EmbWWIThm}. Given a finite dimensional virtual $G$-representation $\rho$ over $\R$, denote by $\s^\rho$ the representation sphere spectrum associated to it; we will consider its homotopy orbit spectrum $\s^\rho_{hG}$, which is equivalent to the Thom spectrum of the associated virtual vector bundle $EG\times_G\rho\to BG$. Let $\psi_m$ denote the real $m$-dimensional representation of the dihedral group $D_m$ (seen as a subgroup of the symmetric group $\Sigma_m$) given by permuting the factors of $\R^m$, and let $\sigma: C_2=\{\pm 1\}\xhookrightarrow{}\R^\times$ be the sign representation (also regarded as a $D_m$-representation by restriction along the determinant $D_m\xhookrightarrow{}O(2)\overset\det\to \{\pm 1\}= C_2$). 

\begin{thmx}\label{LongKnotsThm}
   For $p\leq d-3$ and $d\geq 5$, consider the virtual $D_m$-representations
$$
    \rho_m:=(d+1)(\sigma-1)+\psi_m\otimes(d-p-3+\sigma).
$$  
    Then the homotopy fibre sequence (\ref{knotfibseq1}), upon localising away from $2$ and taking $(\phi_{\cemb}(d,p)-1)$-th Postnikov sections, takes the form
 \begin{equation}\label{LKfibrationStatement}
 \begin{tikzcd}
\prod_{m\geqs 2} \Omega^{\infty}\big(\s^{\rho_m}_{hD_m}\big)\rar&\emb_\partial(D^p, D^d)\rar&\Omega^{p}\hofib\big(G(d-p)/O(d-p)\to G/O\big).
\end{tikzcd}
   \end{equation}
    The resulting sequence is split if $p\geq 2$, and splits after being looped once if $p=1$.
\end{thmx}

\begin{rem}\label{thmBrem}
    (i)\hspace{1.5pt} The spaces $G(n)/O(n)$ and $G/O$ appearing in (\ref{LKfibrationStatement}) denote the homotopy fibres of the natural maps $BO(n)\to BG(n)$ and $BO\to BG$, respectively, where $G(n)$ is the topological grouplike monoid of self-homotopy equivalences of $S^{n-1}$, and $G$ is the homotopy colimit of the suspension maps $G(n)\to G(n+1)$. Understanding the homotopy groups of these spaces roughly amounts to understanding unstable and stable homotopy groups of spheres.

    \noindent (ii)\hspace{1.5pt} The space $\emb_\partial(D^p,D^d)$ is an $\mathbb{E}_p$-algebra, and so it can indeed be localised. When $d-p\geq 3$, it is exactly $(2d-3p-4)$-connected by work of Budney \cite[Proposition 3.9]{BudneyIntegralLongKnots}. So when $2d-3p-4<0$, by \textit{localising} $\emb_\partial(D^p,D^d)$ we really mean localising each of its connected components one at a time (and \emph{not} localising the abelian group\footnote{If $d-p\geq 3$, the monoid $\pi_0(\emb_\partial(D^p,D^d))$ is indeed an abelian group: that it is a group follows from Hudson's Theorem \ref{Hudsonstheorem}, which says that it is isomorphic to $\pi_0(\bemb_\partial(D^p,D^d))\cong \pi_p(\bemb_\partial(*,D^{d-p}))$. That it is abelian when $p=1$ is well-known, and can be proved by the ``pulling one knot through another'' trick, as illustrated in~\cite[Fig.~2]{Budney2007}.} $\pi_0(\emb_\partial(D^p,D^d))$ directly). The same applies to the outer terms in (\ref{knotfibseq1}). 

    \noindent (iii)\hspace{1.5pt} When $p=1$ and $d=4$ (ie the lowest dimensional case of interest if $d-p\geq 3$), Theorem \ref{LongKnotsThm} holds after looping both (\ref{knotfibseq1}) and (\ref{LKfibrationStatement}); see Remark \ref{p1d4caserem}. Using this to study the homotopy groups of $\emb_\partial(D^1,D^4)$, however, yields weaker results than the ones of Budney \cite[Proposition 3.9]{BudneyIntegralLongKnots}. See Remark~\ref{BudneyComparisonRem} for further comparison of our work with Budney's, and for consequences of Theorem \ref{LongKnotsThm} for the Gromoll filtration of spaces of long knots.
\end{rem}

For $A$ an abelian group and $\ell$ a prime, write $A_{(\ell)}:=A\otimes \Z_{(\ell)}$, ie the localisation of $A$ at the prime $\ell$. We will compute some of the homotopy groups of the fibre in (\ref{LKfibrationStatement}) in Section \ref{htpyLKsection}, and hence deduce new torsion information about the homotopy groups of $\emb_\partial(D^p,D^d)$ in high-dimensions (ie $d\geq 5$). 

\begin{corx}[Propositions \ref{htpygroupsEmcoprimeprop} \& \ref{htpyat3prop}] \label{LongKnotsCor} For $d-p\geq 3$ and $\ell$ an odd prime, there are isomorphisms
$$
\pi_*(\emb_\partial(D^p,D^d))_{(\ell)}\cong \pi_{*+p}(\hofib(G(d-p)/O(d-p)\to G/O))_{(\ell)}\oplus\bigoplus_{m\geq 2}\pi_*^s(\s^{\rho_m}_{hD_m})_{(\ell)}
$$
in degrees $*\leq \phi_{\cemb}(d,p)-1$. When $m\geq 2$ and $\ell\nmid 2m$,
    \begin{itemize}
    \item if $d$ is even and $p$ is even, then
    $$
\pi_*^s(\s^{\rho_m}_{hD_m})_{(\ell)}\cong\left\{
\begin{array}{cl}
    \pi_{*-m(d-p-2)}^s\otimes \Z_{(\ell)}, & m=3,5,7,\dots\\
    0, & \text{otherwise}.
\end{array}
\right.
    $$

    \item if $d$ is odd and $p$ is odd, then
    $$
\pi_*^s(\s^{\rho_m}_{hD_m})_{(\ell)}\cong\left\{
\begin{array}{cl}
    \pi_{*-m(d-p-2)}^s\otimes \Z_{(\ell)}, & m=2,4,6,\dots\\
    0, & \text{otherwise}.
\end{array}
\right.
    $$

    \item if $d$ is even and $p$ is odd, then
    $$
\pi_*^s(\s^{\rho_m}_{hD_m})_{(\ell)}\cong\left\{
\begin{array}{cl}
    \pi_{*-m(d-p-2)}^s\otimes \Z_{(\ell)}, & m=5,9,13,\dots\\
    0, & \text{otherwise}.
\end{array}
\right.
$$ 

    \item if $d$ is odd and $p$ is even, then
    $$
\pi_*^s(\s^{\rho_m}_{hD_m})_{(\ell)}\cong\left\{
\begin{array}{cl}
    \pi_{*-m(d-p-2)}^s\otimes \Z_{(\ell)}, & m=3,7,11,\dots\\
    0, & \text{otherwise}.
\end{array}
\right.
    $$
\end{itemize}
If $\ell$ divides $m$, the computation of $\pi_*^s(\s^{\rho_m}_{D_m})_{(\ell)}$ must be treated case by case. When $m=\ell=3$ and $d-p=3$, the first few such homotopy groups are given in Table \ref{tablehtpygroupsd-p3} (cf Proposition \ref{htpyat3prop} for notation). 
\end{corx}

Rationally, this computation roughly recovers the homology of the $0$- and $1$-loop order part of the hairy graph complex appearing in \cite[Eq. 2]{FresseTurchinWillwacher} (see Remark \ref{rationallongknotsrem} for more details).

\subsection*{Structure of the paper} Sections \ref{section2} and \ref{section4} will be devoted to the proof of Theorem \ref{EmbWWIThm}. We start by briefly reviewing Weiss' theory of orthogonal calculus and then, in Section \ref{boundedsection}, we present the orthogonal functors that will play a role in the proof of Theorem \ref{EmbWWIThm}. In doing this, we will have to carefully describe the topology on spaces of bounded diffeomorphisms in such a way that we can employ the machinery of orthogonal calculus in this setting. After reducing to the codimension zero case in Section \ref{codimensionsection}, we use the results in the preceding section to define the map $\Phi^{\emb}$ in Section \ref{ThmAproofsubsection} and analyse its connectivity in \ref{connectivitysection}.

Section \ref{splittingknotsection} concerns the splitting result (Theorem \ref{splittingthm}) for embedding spaces mentioned in Remark \ref{IntroRemGromoll}, and its consequences for the Gromoll filtration (cf Corollary \ref{CorGromollFiltration}).

Section \ref{pseudoisotopyknotsection} deals with the analysis of the $C_2$-spectra involved in the statements of Theorems \ref{WWbdiffmoddiff} and \ref{EmbWWIThm}. The main results in this direction are Theorem \ref{absolutepropTWWvsTepsilon} and Corollary \ref{propTWWvsTepsilon}, where the involutions on these spectra are expressed (up to homotopy) in terms of the standard involution in algebraic $K$-theory.

Section \ref{longknotsection} is devoted to Theorem \ref{LongKnotsThm}, whose proof is a formal consequence of the results in the preceding sections. We then draw some conclusions on the homotopy groups of spaces of long knots in Section \ref{htpyLKsection}.

Appendix \ref{orthcalcappendix} deals with some subtleties regarding the definition of the first derivative of an orthogonal functor as an $O(1)$-spectrum, and with a technical argument in the proof of Proposition \ref{propunstablemapsorthcalc}.

In Appendix \ref{boundedappendix} we explore certain aspects related to spaces of bounded diffeomorphisms and embeddings. Namely in Section \ref{boundeddifftopmodelsection} we show that the topological models for these spaces introduced in Section \ref{boundedsection} coincide (up to weak equivalence) with the simplicial ones of Definition \ref{simplicialboundeddefn}. In Section \ref{bdiffbmodelsection} we give a ``moduli space of manifolds'' description for the classifying space of the bounded diffeomorphism group.

In Appendix \ref{AppendixB} we show that the $h$-cobordism stabilisation map anti-commutes with the involutions in these spaces. This is analogous to a result of Hatcher \cite[Appendix I, Lemma]{HatcherSSeq} and Burghelea--Lashof \cite[Corollary A7]{BurgLash} for spaces of concordance diffeomorphisms.

\subsection*{Acknowledgements}
The author is immensely grateful to his PhD supervisor Oscar Randal-Williams for suggesting the application to spaces of long knots of Theorem~\ref{EmbWWIThm}, and for his continuous support, discussions, and motivation throughout the project. He also thanks Manuel Krannich for pointing out Lemma~\ref{ThetaF1nconnectivity} and how it could be used to greatly simplify our original proof of Theorem~\ref{EmbWWIThm}, and Alexander Kupers for providing various comments and corrections when carefully reviewing the author's PhD thesis (which this paper was part of). The anonymous referees are thanked for their helpful comments, including suggesting a simpler proof of Proposition~\ref{propunstablemapsorthcalc} and encouraging the author to investigate consequences of Theorem~\ref{LongKnotsThm} for the Gromoll filtration of long knots of Budney \cite[Defn. 3.8]{BudneyIntegralLongKnots}. The author is also grateful to Mauricio Bustamante, Tom Goodwillie, Bjørn Jahren, and John Rognes for general helpful conversations. This work was supported by an EPSRC PhD Studentship, grant no.~2597647.

\section{Orthogonal calculus and spaces of bounded diffeomorphisms}\label{section2}
Much of the proof of Theorem \ref{WWbdiffmoddiff} in \cite{WWI} is an application of Weiss' \textit{orthogonal calculus} but in disguise, as this theory was not yet formalised at the time. In this section we briefly review the main aspects of this theory and develop some necessary tools required for the proof of Theorem \ref{EmbWWIThm}.

\subsection{A quick tour through orthogonal calculus} \label{OrthReviewSection}Weiss' \textit{orthogonal calculus} \cite{WeissOrthCalc} is a calculus of functors useful to understand objects of geometric flavour. It studies \textit{continuous} functors from the category $\mathcal J$ of real finite-dimensional inner product vector spaces and linear isometries to the category of (compactly generated weakly Hausdorff) spaces $\mathsf{Top}$. Such a functor $F: \mathcal J\to \mathsf{Top}$ is said to be \textit{continuous} if the evaluation map
$$
\mathrm{mor}_{\mathcal J}(U,V)\times F(U)\longrightarrow F(V)
$$
is continuous for all $U,V\in\mathcal J$. Here $\mathrm{mor}_{\mathcal J}(U,V)$ denotes the \textit{Stiefel manifold} of linear isometries from $U$ to $V$, so that $\mathcal J$ is enriched over $\mathsf{Top}$. We will work in a slightly different setup, where $\mathsf{Top}$ is replaced by the category $\mathsf{Top}_*$ of pointed spaces and $\mathcal J$ is replaced by the pointed topological category $\mathcal J_0$ with the same objects and with
$$
\mathrm{mor}_{\mathcal J_0}(U,V):=\mathrm{mor}_{\mathcal J}(U,V)_+,
$$
as morphism spaces. Similarly, a functor $F: \mathcal J_0\to \mathsf{Top}_*$ is \textit{continuous} if the evaluation map
$$
\mathrm{mor}_{\mathcal J_0}(U,V)\wedge F(U)\longrightarrow F(V)
$$
is continuous for all $U,V\in \mathcal J_0$. Such a functor $F(-)$ is also sometimes called an \textit{orthogonal functor}.

The machinery of orthogonal calculus associates to each such orthogonal functor $F(-)$ a sequence of (naïve) $O(k)$-spectra\footnote{By \emph{spectrum}, we will always mean pre-spectrum. For us, a \textit{$G$-spectrum} $X$ will be a sequence of based $G$-spaces $X_n$ together with $G$-equivariant maps $S^1\wedge X_n\to X_{n+1}$, where $G$ acts trivially on $S^1$.} $\Theta F^{(k)}$ for $k\geq 1$---the \textit{derivatives of} $F$---which fit in a tower
\begin{equation}\label{OrthTower}
\begin{tikzcd}[row sep=tiny, column sep=scriptsize]
     &&& \vdots\dar &\\
     &&& T_2F(-)\dar &\lar \Omega^\infty\left((S^{2\cdot(-)}\wedge \Theta F^{(2)})_{hO(2)}\right)\\
     &&& T_1F(-)\dar &\lar \Omega^\infty\left((S^{1\cdot(-)}\wedge \Theta F^{(1)})_{hO(1)}\right)\\
     F(-)\ar[rrr, "\eta_0"']\arrow[urrr]\arrow[uurrr] \arrow[uuurrr] &&&T_0F(-)\rar[equal, shorten >=12ex]&\hspace{-70pt}F(\R^\infty)
\end{tikzcd}
\end{equation}
of orthogonal functors---the \textit{Taylor tower}. Here
\begin{itemize}[itemsep = 3pt]
    \item $S^{k\cdot V}$ is the one point compactification of $k\cdot V:=\R^k\otimes V$, which is acted upon by $O(k)$ in the $\R^k$ component, and diagonally on the smash $S^{k\cdot V}\wedge \Theta F^{(k)}$,
    
    \item the right hand horizontal maps---the \textit{layers}---indicate the inclusions of the homotopy fibres of the subsequent vertical maps between the \textit{stages} $T_k F(-)$ of the tower,

    \item the zeroth stage $T_0F(-)$ is given by
    $
    T_0F(V):=\operatorname{hocolim}_k F(V\oplus \R^k)
    $, and thus admits a canonical equivalence from the constant orthogonal functor with \textit{value at infinity} $F(\R^\infty):=\operatorname{hocolim}_k F(\R^k)$. The map $\eta_0: F(V)\to T_0F(V)$ is simply the inclusion map.
\end{itemize}

In the proof of Theorem \ref{EmbWWIThm} we will analyse the Taylor tower (\ref{OrthTower}) only up to the first layer, so we shall now describe the spectrum $\Theta F^{(1)}$ in detail. For $V\in \mathcal J_0$, consider $$F^{(1)}(V):=\hofib(F(V)\to F(V\oplus \R)),$$ 
the homotopy fibre of the map induced by the standard inclusion $V\to V\oplus \R$. These spaces inherit an $O(1)$-action by declaring $-1\in O(1)$ to act on $V$ and $V\oplus \R$ by $-1$ on \textit{all} coordinates. There are $O(1)$-equivariant maps
\begin{equation}\label{orthstructuremap}
s_V: S^1\wedge F^{(1)}(V)\longrightarrow F^{(1)}(V\oplus \R),
\end{equation}
given, roughly, by performing a 180º rotation of $V\oplus \R^2$ about the $2$-plane $0\oplus \R^2$; we give an explicit model of this map in Section \ref{ExplicitStructureMapSection}. As notation suggests in (\ref{orthstructuremap}), $O(1)$ acts trivially on the suspension coordinate. (In general, we adopt the convention that $S^n$ denotes the $n$-sphere with the trivial $O(1)$-action.) Then, the $O(1)$-spectrum $\Theta F^{(1)}$ has $F^{(1)}(\R^n)$ as its $n$-th space, and $s_{\R^n}$ as the structure map.

\begin{rem}\label{orthcalcinvolutionrem}
This is not quite the $O(1)$-action described in \cite[Proposition 3.1]{WeissOrthCalc}; rather, it more closely follows the convention adopted in \cite[p. 601]{WWI}. We justify our convention choice in Section \ref{ThetaO1appendix}.
\end{rem}

For $V\in \mathcal J_0$, let $S(V)$ denote the unit sphere of $V$, seen as an unbased $O(1)$-space by the antipodal action. The following proposition will be the main ingredient for the construction of the map $\Phi^{\emb}$ of Theorem \ref{EmbWWIThm}.

\begin{prop}\label{propunstablemapsorthcalc}
Let $F: \mathcal J_0\to \mathsf{Top}_*$ be an orthogonal functor. For each $n\geq 0$, there are maps
\begin{equation}\label{unstableWWImaps}
\begin{tikzcd}
    \Phi_n^F: \hofib(F(0)\to F(\R^n))\rar["\eta_1"] &\hofib(T_1F(0)\to T_1F(\R^n))\simeq \Omega^\infty(S(\R^n)_+\wedge_{O(1)}\Theta F^{(1)})
\end{tikzcd}
\end{equation}
giving rise to a map of homotopy fibre sequences
$$
\begin{tikzcd}[column sep = 12pt]
    \hofib(F(0)\to F(\R^n))\rar\dar["\Phi_n^F"] &\hofib(F(0)\to F(\R^{n+1}))\rar\dar["\Phi_{n+1}^F"] &\hofib(F(\R^n)\to F(\R^{n+1}))=:\Theta F^{(1)}_n\dar["\mathrm{stab.}"] \\
    \Omega^\infty(S(\R^n)_+\wedge_{O(1)}\Theta F^{(1)})\rar & \Omega^\infty(S(\R^{n+1})_+\wedge_{O(1)}\Theta F^{(1)})\rar & \Omega^\infty(\Sigma^n\Theta F^{(1)}),
\end{tikzcd}
$$
where the vertical map ``$\mathrm{stab.}$'' is $\Theta F^{(1)}_n\xhookrightarrow{}\operatorname{hocolim}_k\Omega^k(\Theta F^{(1)}_{n+k})$. Letting $n\to \infty$ in (\ref{unstableWWImaps}), we get
\begin{equation}\label{stableWWImaps}
   \Phi^F_\infty: \hofib(F(0)\to F(\R^\infty))\longrightarrow \Omega^\infty(EO(1)_+\wedge_{O(1)}\Theta F^{(1)})=:\Omega^\infty(\Theta F^{(1)}_{hO(1)}). 
\end{equation}
\end{prop}

\begin{proof}
    Given $F: \mathcal{J}_0\to \mathsf{Top}_*$, write $L_1 F:=\hofib(T_1F\to T_0 F)$ for the first homogeneous layer of $F$. Consider also the orthogonal functor $$E(-):= \Omega^\infty\big((S^{(-)\cdot\sigma}\wedge \Theta F^{(1)})_{hO(1)}\big).$$
    By the classification of homogeneous functors of \cite[Theorem 9.1]{WeissOrthCalc}, there is a natural equivalence\footnote{Theorem 9.1 in \cite{WeissOrthCalc} states that \( L_1 F \) is naturally equivalent to the functor \( E' = \Omega^\infty\left((S^{(-)\cdot\sigma} \wedge \Theta^{\#} F^{(1)})_{hO(1)}\right) \), where \( \Theta^{\#} F^{(1)} \) is defined in \eqref{Thetasharpdefn}. Since this \( O(1) \)-spectrum is naturally equivalent to \( \Theta F^{(1)} \) by \eqref{Thetaheartzigzag}, it follows that \( E' \) is naturally equivalent to \( E \).} of functors $L_1 F\simeq E$, and hence a zig-zag of orthogonal functors
    $$
    \begin{tikzcd}
   F\rar["\eta_1"] &T_1F &\lar L_1 F\simeq E.  
    \end{tikzcd}
    $$
This zig-zag gives rise to the following commutative diagram, in which the vertical arrows induce maps of homotopy fibre sequences:
$$
\begin{tikzcd}
\hofib(F(0)\to F(\R^n))\rar\dar["\eta_1"] &\hofib(F(0)\to F(\R^{n+1}))\rar\dar["\eta_1"] &\Theta F^{(1)}_n\dar["\eta_1"]\\
\hofib(T_1F(0)\to T_1F(\R^n))\rar &\hofib(T_1F(0)\to T_1F(\R^{n+1}))\rar &\Theta (T_1F)^{(1)}_n\\
\hofib(L_1F(0)\to L_1F(\R^n))\rar\uar["\vsim"] &\hofib(L_1F(0)\to F(\R^{n+1}))\rar\uar["\vsim"] &\Theta (L_1F)^{(1)}_n\uar["\vsim"]\\
\hofib(E(0)\to E(\R^n))\rar\uar[dash, "\vsim"] &\hofib(E(0)\to E(\R^{n+1}))\rar\uar[dash, "\vsim"] &\Theta E^{(1)}_n.\uar[dash, "\vsim"]
\end{tikzcd}
$$

We claim that the bottom fibre sequence in this diagram is naturally equivalent to the bottom fibre sequence in the map of fibre sequences of the statement: indeed, for $n\geq 0$, write $\mathrm{Ind}_e^{O(1)}S^n$ for the wedge $S^n\vee S^n$ with the flip action. Then, for $V\in \mathcal{J}_0$ with $\dim V=v$, there is a diagram of $O(1)$-spaces
\begin{equation}\label{O(1)cofibrediagram}
\begin{tikzcd}
S(V)_+\rar\dar[equal] &\dar S(V\oplus \R)_+\rar\dar & \mathrm{Ind}_e^{O(1)} S^v\dar["a"]\\
S(V)_+\dar\rar& S^0\dar\rar &S^{V\cdot \sigma}\dar\\
*\rar & S^{(V\oplus \R)\cdot \sigma}\rar[equal] & S^{(V\oplus \R)\cdot \sigma},
\end{tikzcd}
\end{equation}
where every row and column is an $O(1)$-equivariant (homotopy) cofibre sequence. Here, the map $a: \mathrm{Ind}_e^{O(1)} S^v\cong S^V_0\vee S^V_1\to S^{V\cdot \sigma}$ sends $x\in V\subset S^{V}_i$ to $(-1)^ix$ for $i=0,1$. Applying $\Omega^\infty\big((-\wedge \Theta F^{(1)})_{hO(1)}\big)$ to \eqref{O(1)cofibrediagram}, we obtain a diagram for $V=\R^n$:
$$
\begin{tikzcd}
 \Omega^\infty\big(S(\R^n)_+\wedge_{O(1)}\Theta F^{(1)}\big)\rar\dar["\vsim"'] & \Omega^\infty\big(S(\R^{n+1})_+\wedge_{O(1)}\Theta F^{(1)}\big)\rar\dar & \Omega^\infty\big(\Sigma^n\Theta F^{(1)}\big)\dar\\
\hofib(E(0)\to E(\R^n))\dar\rar& E(0)\dar\rar &E(\R^n)\dar\\
*\rar & E(\R^{n+1}) \rar[equal] & E(\R^{n+1}),
\end{tikzcd}
$$
where every row and column is a homotopy fibre sequence. Observe that the top horizontal cofibre sequence of \eqref{O(1)cofibrediagram} consists of free \( O(1) \)-spaces, which explains the underived balanced smash product \( \wedge_{O(1)} \); note also that \( \mathrm{Ind}_e^{O(1)} S^n \wedge_{O(1)} X \simeq S^n \wedge X \cong \Sigma^n X \) for any \( O(1) \)-spectrum \( X \).

Consequently, we obtain a natural map of homotopy fibre sequences
\begin{equation}\label{snFdiagram}
\begin{tikzcd}[column sep = 12pt]
    \hofib(F(0)\to F(\R^n))\rar\dar["\Phi_n^F"] &\hofib(F(0)\to F(\R^{n+1}))\rar\dar["\Phi_{n+1}^F"] &\Theta F^{(1)}_n\dar["s_n^F"] \\
    \Omega^\infty(S(\R^n)_+\wedge_{O(1)}\Theta F^{(1)})\rar & \Omega^\infty(S(\R^{n+1})_+\wedge_{O(1)}\Theta F^{(1)})\rar & \Omega^\infty(\Sigma^n\Theta F^{(1)}).
    \end{tikzcd}
\end{equation}
It remains to argue that $s_n^F$ is naturally homotopic to the map $\mathrm{stab}^F_n:\Theta F^{(1)}_n\xhookrightarrow{}\hocolim_k \Omega^k(\Theta F^{(1)}_{n+k})$. Indeed, there is a commutative diagram
\begin{equation}\label{snFdiagram2}
\begin{tikzcd}
&&&\Omega^\infty(\Sigma^n\Theta F^{(1)})\dar[dash, "\vsim"']\\
    \Theta F^{(1)}_n\ar[rrru, "s_n^F", bend left = 8pt]\rar\dar[hook,"\mathrm{stab}^F_n"] & \Theta (T_1F)^{(1)}_n\dar[hook,"\mathrm{stab}^{T_1F}_n", "\vsim"'] &\lar["\sim"']\Theta (L_1F)^{(1)}_n\rar[dash, "\sim"]\dar[hook,"\mathrm{stab}^{L_1F}_n", "\vsim"'] &\Theta E^{(1)}_n \dar[hook,"\mathrm{stab}^E_n", "\vsim"']\\
    \Omega^\infty(\Sigma^n\Theta F^{(1)})\rar["\sim"] &\Omega^\infty(\Sigma^n\Theta (T_1F)^{(1)}) &\lar["\sim"']\Omega^\infty(\Sigma^n\Theta (L_1F)^{(1)})\rar[dash, "\sim"] &\Omega^\infty(\Sigma^n\Theta E^{(1)}),
\end{tikzcd}
\end{equation}
where the leftmost horizontal is an equivalence since $\eta_1:\Theta F^{(1)}\to \Theta(T_1F)^{(1)}$ is an equivalence by \cite[Theorem 6.3(bis)]{WeissOrthCalc}, whilst all but the leftmost vertical stabilisation maps are equivalences because the spectra involved are $\Omega$-spectra: indeed this is the case for $\Theta (T_1F)^{(1)}$ and $\Theta (L_1 F)^{(1)}$ by \cite[Corollary 5.12]{WeissOrthCalc}, whilst $\Theta E^{(1)}$ is the \textit{spectrification} of the pre-spectrum $\Theta F^{(1)}$. (This is the case for any functor of the form $H_X(-)= \Omega^\infty((S^{(-)\cdot\sigma}\wedge X)_{hO(1)})$, for $X$ an $O(1)$-(pre-)spectrum; the first derivative of $H_X$ is, non-equivariantly, the spectrification of $X$ by \cite[Section 7]{WeissOrthCalc}.) From this last observation, it also follows that the right vertical composite $\Omega^\infty(\Sigma^n\Theta F^{(1)})\xrightarrow{\sim} \Omega^\infty(\Sigma^n\Theta E^{(1)})$ is induced by the spectrification map $\Theta F^{(1)}\xrightarrow{\sim} \Theta E^{(1)}$. Thus, the sequence of equivalences in the lower row and rightmost column are all induced by equivalences of spectra
\begin{equation}\label{IdentityThetaF1}
\begin{tikzcd}
    \Theta F^{(1)}\rar["\sim"] &\Theta(T_1 F)^{(1)}&\lar["\sim"']  \Theta (L_1F)^{(1)}\rar[dash, "\sim"] & \Theta E^{(1)} & \lar["\sim"'] \Theta F^{(1)}.
\end{tikzcd}
\end{equation}
To argue that \( \mathrm{stab}_n^F \) and \( s_n^F \) are homotopic, we need to show that \eqref{IdentityThetaF1} lies in the homotopy class of the identity map of \( \Theta F^{(1)} \). For this, we need to be slightly more explicit about the equivalence \( L_1 F \simeq E \), which factors as
$$
L_1 F\overset{(!)}{\simeq}H_{L_1F}:=\Omega^\infty\big((S^{(-)\cdot \sigma}\wedge \Theta(L_1 F)^{(1)})_{hO(1)}\big)\xleftarrow{\sim}\Omega^\infty\big((S^{(-)\cdot \sigma}\wedge \Theta F^{(1)})_{hO(1)}\big)=E. 
$$
The right map in this zig-zag is induced by the natural map \( \Theta F^{(1)} \xrightarrow{\eta_1} \Theta(T_1F)^{(1)} \simeq \Theta(L_1F)^{(1)} \), whereas the equivalence \((!)\) is provided by \cite[Theorem~7.3]{WeissOrthCalc}; this equivalence satisfies that the induced one on first derivatives \( (!):\Theta(L_1 F)^{(1)} \simeq \Theta(H_{L_1F})^{(1)} \) becomes the identity when composed with the natural equivalence \( \Theta(H_{L_1F})^{(1)} \xleftarrow{\sim} \Theta(L_1F)^{(1)} \). It now easily follows that the zig-zag \eqref{IdentityThetaF1} represents the homotopy class of \( \mathrm{Id}_{\Theta F^{(1)}} \), and hence \( \mathrm{stab}_n^F \) and \( s_n^F \) are indeed homotopic, as desired.
\end{proof}

\begin{rem}
    In the proof of Theorem~\ref{EmbWWIThm}, we will only need that the map \( s_n^F \) in \eqref{snFdiagram} is as connected as the stabilisation map \( \mathrm{stab}_n^F \). This already follows from the commutative diagram \eqref{snFdiagram2}, which shows that \( s_n^F \) and \( \varphi^F_n\circ \mathrm{stab}_n^F \) are naturally homotopic, for some natural automorphism $\varphi_n^F$ of the codomain $\Omega^\infty(\Sigma^n\Theta F^{(1)})$. Thus, the last part of the proof above is not strictly necessary for our purposes. Nevertheless, we believe it is useful to have this technical point clarified.
\end{rem}

\subsection{The orthogonal functors of bounded diffeomorphisms}\label{boundedsection}

All throughout, let $\iota:P\xhookrightarrow{}M$ be as in the statement of Theorem \ref{EmbWWIThm}. In this section we present the orthogonal functors that will play a role in the proof of Theorem \ref{EmbWWIThm}. These are built out of spaces of bounded diffeomorphisms, for which we will present point-set topological models that agree up to weak equivalence with the more classical simplicial ones.

Let $V\in \mathcal J$ be an inner product finite-dimensional real vector space with associated norm $\norm{-}_V$, and let $Q$ and $Q'$ be smooth (possibly non-compact) manifolds equipped with proper maps $\pi: Q\to V$ and $\pi': Q'\to V$. For $t\geq 0$, a smooth map $f: Q\to Q'$ is said to be \textit{t-bounded} if the set $
\{\norm{\pi'(\hspace{2pt}f(q))-\pi(q)}_V: q\in Q\}\subset \R$ is bounded by $t$. More generally, $f$ is \textit{bounded} if it is $t$-bounded for some $t\geq 0$. If $Q=N\times V$ for some compact manifold $N$, $\pi$ will be assumed to be the projection to $V$.

\begin{defn}\label{topboundeddefn}
    Let $V\in \mathcal J$. The \textbf{space of bounded diffeomorphisms} of $M\times V$ relative to $\partial M\times V$ is 
    $$
\diff^b_\partial(M\times V):=\{(t,\phi)\in [0,\infty)\times \diff_\partial(M\times V): \text{$\phi$ is $t$-bounded}\},
    $$
    endowed with the subspace topology inherited from the product $[0,\infty)\times \diff_\partial(M\times V)$. Here $\diff_\partial(M\times V)$ is endowed with the weak Whitney $C^\infty$-topology. It is a group-like topological monoid under the rule $$(t,\phi)\cdot(t',\phi'):=(t+t',\phi\circ\phi').$$
    Similarly, we define the space $\mathrm{Homeo}^b_{\partial}(M\times V)$ of \textbf{bounded homeomorphisms} of $M\times V$ as a subspace of the product $\mathrm{Homeo}_{\partial}(M\times V)\times[0,+\infty)$, where $\mathrm{Homeo}_\partial(M\times V)$ is endowed with the compact-open topology.
\end{defn}

\begin{exm}\label{orthexampleWWI}
    Orthogonal calculus was largely inspired by the work of Weiss--Williams in \cite{WWI}, as can be seen in \cite[Digr. 3.8]{WWI}. For $U\in \mathcal J_0$, let $B\diff_\partial^b(M\times U)$ be the classifying space of the group-like topological monoid just introduced. Denote by $B(-)$ the orthogonal functor given by $B(U):=B\diff^b_\partial(M\times U)$ and, for $i: U\to V$ a morphism in $\mathcal J_0$, write $V=U\oplus U^\perp$ and let $B(i)$ be induced by the monoid homomorphism sending $(t,\phi)\in\diff^b_\partial(M\times U)$ to $(t,\phi\oplus\mathrm{Id}_{U^\perp})\in\diff^b_\partial(M\times U\oplus U^\perp)$. Then 
    $$
\Phi^{B}_\infty: \diff^b_\partial(M\times \R^\infty)/\diff_\partial(M)\longrightarrow \Omega^\infty(\Theta B^{(1)})
    $$
    should be\footnote{Though $\Phi^B_\infty$ is not visibly the same map as the one appearing in loc. cit., they share the same formal properties by Proposition \ref{propunstablemapsorthcalc}.} the map from \cite[Theorem C]{WWI}, and Proposition \ref{propunstablemapsorthcalc} recovers \cite[Proposition 3.1]{WWI} in this case.
\end{exm}

\begin{notn}\label{FEBorthfunctorsnotn}
    In the remainder of Section \ref{section2}, we will denote by $E(-)$ and $B(-)$ the orthogonal functors given on objects by
    \begin{equation}\label{functors}
E(V):=B\diff^b_{\partial}((M-\nu P)\times V),\qquad B(V):=B\diff^b_{\partial}(M\times V),
\end{equation}
where $\nu P$ is an open tubular neighbourhood of the embedding $\iota: P\subset M$, and on morphisms as in Example \ref{orthexampleWWI}. There is a natural transformation $E(-)\to B(-)$ given by extending a diffeomorphism by the identity on $\nu P\times (-)$, and the orthogonal functor $F(-):=\hofib(E(-)\to B(-))$ will play an especially important role in the proof of Theorem \ref{EmbWWIThm}. We will often use the following notation for the derivatives of these functors:
\begin{equation}\label{Hspectrumdefn}
\CEsp(P,M):=\Theta F^{(1)}, \qquad \Hsp(M-\nu P):=\Theta E^{(1)}, \qquad \Hsp(M):=\Theta B^{(1)}.
\end{equation}
Thus, $\CEsp(P,M)$ is, by definition, the homotopy fibre of the map $\Hsp(M-\nu P)\to \Hsp(M)$.
\end{notn}

\begin{rem}\label{justificationrem1}
    Let us comment on the notation in (\ref{Hspectrumdefn}). Write $H(M)$ for the space of smooth $h$-cobordisms starting at $M$ (cf Section \ref{hcobsection}). We will see in Remark \ref{justificationrem2} that there is an equivalence (of spaces)
   \begin{equation}\label{Hspinfloopsace}
\Omega^\infty \Hsp(M)\simeq \mathcal H(M):=\underset{k}{\hocolim}\  H(M\times D^k),
    \end{equation}
    where the colimit on the right hand side---the \textit{stable $h$-cobordism space}---is induced by the $h$-cobordism stabilisation maps of Appendix \ref{AppendixB}. The equivalence \eqref{Hspinfloopsace} is natural in codimension-zero embeddings, provided that we restrict to basepoint components. It should also be natural when considering all components, but establishing this seems more tedious. See Remark \ref{justificationrem2} for further discussion of this naturality.
    
    By the stable parametrised $h$-cobordism theorem of Waldhausen--Jahren--Rognes \cite{WJR}, the infinite loop space of the desuspension of the \textit{smooth Whitehead spectrum} $\Sigma^{-1}\Whsp(M)$ is also equivalent to $\mathcal H(M)$ (as ordinary spaces). Moreover, Weiss and Williams showed in \cite[Corollary 5.6]{WWI} that the spectra $\Theta B^{(1)}=\Hsp(M)$ and $\Sigma^{-1}\Whsp(M)$ also share the same negative homotopy groups, which led them to rename the former as the latter. This, though conjecturally true, was not fully justified since no equivalence between these two spectra was given. 
    
    We hope that the homotopy fibre sequence $\cemb(P,M)\to H(M-\nu P)\to H(M)$ and Proposition \ref{alexmapsprop} together explain why we denote $\Theta F^{(1)}$ by $\CEsp(P,M)$.
\end{rem}

Spaces of bounded diffeomorphisms are usually defined as the geometric realisation of certain simplicial groups/sets. Before we recall these simplicial models in Definition \ref{simplicialboundeddefn} below, let us fix some notation first. For a subset $S\subset \R^{p+1}$ and $\epsilon>0$, let $B_\epsilon(S)\subset \R^{p+1}$ denote the open $\epsilon$-ball around $S$. For $0<\epsilon\leq 1/2$ and for any face $\sigma\subset \Delta^p$, we fix radial identifications $\rho_\sigma: \partial\sigma(\epsilon):=B_\epsilon(\partial\sigma)\cap \sigma \cong \partial\sigma \times [0,\epsilon)$; let us first do it for $\sigma=\Delta^p$. Given $x=(t_0,\dots, t_p)\in\partial \Delta^p(\epsilon)$, let $j\in[p]$ be such that $t_j\leq t_i$ for every $i\in [p]$. Note that since $x$ cannot be the barycenter $b_p=(\tfrac{1}{p+1},\dots, \tfrac{1}{p+1})$ of $\Delta^p$ (since this lies at distance greater than $1/2\geq \epsilon$ from $\partial \Delta^p$), we must have that $t_j$ is strictly smaller than $\tfrac{1}{p+1}$. Then set 
$$
\rho_p: \partial\Delta^p(\epsilon)\overset{\cong}\longrightarrow \partial\Delta^p\times[0,\epsilon), \quad x=(t_0,\dots, t_p)\mapsto \left(x-\tfrac{t_j}{\tfrac{1}{p+1}-t_j}(b_p-x), d(x,\partial\Delta^p)\right),
$$
where $j=j(x)$ is as above, and $d(x,\partial\Delta^p)$ stands for the (Euclidean) distance between $x$ and $\partial \Delta^p$. For a general face $\sigma\subset \Delta^p$, fix the standard order-preserving identification $\eta_\sigma: \sigma\cong \Delta^{|\sigma|}$; then the radial identification $\rho_\sigma: \partial\sigma(\epsilon)\cong \partial\sigma\times[0,\epsilon)$ is 
$$
\begin{tikzcd}[column sep = 40pt]
    \rho_\sigma: \partial\sigma(\epsilon)\rar["\eta_\sigma"] &\partial \Delta^{|\sigma|}(\epsilon)\rar["\rho_{|\sigma|}"] & \partial \Delta^{|\sigma|}\times[0,\epsilon)\rar["\eta_{\sigma}^{-1}\times \mathrm{Id}_{[0,\epsilon)}"] &\partial\sigma\times[0,\epsilon).
\end{tikzcd}
$$

We will say that a continuous map $f: X\times \Delta^p\to Y\times \Delta^p$ over $\Delta^p$ (ie, such that $\mathrm{proj}_{\Delta^p}=\mathrm{proj}_{\Delta^p}\circ f$) satisfies the \textit{$\epsilon$-collaring condition} if for every face $\sigma\subset\Delta^p$,
$$
f\mid_{X\times \partial\sigma(\epsilon)}\equiv f\mid_{X\times \partial \sigma}\times \mathrm{Id}_{[0,\epsilon)}
$$
under the identifications $\rho_\sigma: \partial\sigma(\epsilon)\cong \partial\sigma\times[0,\epsilon)$ fixed above.

\begin{defn}\label{simplicialboundeddefn}
Let $V\in \mathcal J$. The \textbf{semi-simplicial group} $\diff_\partial^b(M\times V)_\bullet$ \textbf{of bounded diffeomorphisms} of $M\times V$ relative to $\partial M\times V$ has as $p$-simplices the set of diffeomorphisms of $\Delta^p\times M\times V$ over $\Delta^p$ which are bounded (with respect to $V$), that are the identity in a neighbourhood of $\Delta^p\times \partial M\times V$, and that satisfy the $\epsilon$-collaring condition for some $0<\epsilon\leq 1/2$. Face maps are determined by the coface maps of the cosimplicial space $\Delta^\bullet$. If we relax the condition on diffeomorphisms to be over $\Delta^p$ to only face-preserving (ie diffeomorphisms that send $\sigma\times M\times V$ to itself for every face $\sigma\subset\Delta^p$), we obtain the semi-simplicial group $\bdiffb_\partial(M\times V)_\bullet$ of \textbf{bounded block diffeomorphisms} of $M\times V$.
\end{defn}

\begin{warn}\label{SimplicialWarning}
    One could have defined the orthogonal functor $B(-)$ of Notation \ref{FEBorthfunctorsnotn}, for instance, to be
    $$
\mathcal J_0\longrightarrow \Top_*, \quad U\longmapsto B|\diff_\partial^b(M\times U)_\bullet|.
$$
This latter rule, however, does not give rise to a continuous functor in the sense of orthogonal calculus, ie, it is not enriched over $\Top_*$. A way to fix this is to replace $\Top_*$ by $\sSet_*$, $\mathcal J_0$ by a category $\mathcal J_0^\Delta$ enriched now over $\sSet_*$, and doing orthogonal calculus for $\sSet_*$-enriched functors $\mathcal J_0^\Delta\to \sSet_*$. This is morally the point of view taken by Weiss and Williams in \cite{WWI}, but orthogonal calculus for simplicially enriched functors has not yet been carried out rigorously, so we prefer to not pursue this approach. This also appears to be the main technicality in the $PL$ case: we do not know how to define an actual \emph{topological} (as opposed to simplicial) orthogonal functor out of spaces of bounded $PL$-homeomorphisms.
\end{warn}

The simplicial models of Definition \ref{simplicialboundeddefn} are more convenient to work with than the point-set topological ones of Definition \ref{topboundeddefn}. Moreover, we will need some results of \cite{WWI} that are stated in the simplicial setting, so we will have to argue that both models share the same weak homotopy type. 

\begin{repproposition}{myAmazingTheorem}
    There is a zig-zag of weak equivalences of semi-simplicial group-like monoids
    $$
    \begin{tikzcd}
       \diff^b_\partial(M\times V)_\bullet&\lar["\sim"']\cdot \rar["\sim"]&\Sing_\bullet(\diff_\partial^b(M\times V)).
    \end{tikzcd}
    $$
    In particular, there is a zig-zag of weak equivalences of group-like topological monoids connecting $|\diff^b_\partial(M\times V)_\bullet|$ and $\diff_\partial^b(M\times V)$.
\end{repproposition}

We defer the proof of this proposition to Section \ref{boundeddifftopmodelsection} in the appendix.

\section{Proof of Theorem \ref{EmbWWIThm}}\label{section4}
We now prove Theorem \ref{EmbWWIThm}. Section \ref{codimensionsection} will first reduce it to the case when $P$ is a codimension zero submanifold of $M$. Some necessary preliminaries will be presented in Section \ref{lastingrsection}. Finally the map $\Phi^{\emb}$ of Theorem \ref{EmbWWIThm} and its connectivity will be analysed in Sections \ref{ThmAproofsubsection} and \ref{connectivitysection}.

Before we move on to the next section, let us record a disjunction result for concordance embeddings known as \textit{Hudson's concordance-implies-isotopy theorem} \cite[Theorem 2.1, Addendum 2.1.2]{Hudson}. 
\begin{thm}[Hudson]\label{Hudsonstheorem}
    The space $\cemb(P,M)$ is connected if $p\leq d-3$. Equivalently, the natural map $\pi_0(\emb_{\partial_0}(P,M))\to \pi_0(\bemb_{\partial_0}(P,M))$ is an isomorphism.
\end{thm}

\begin{rem}\label{HudsonTopRem}
    Hudson's theorem also holds in the $PL$ setting \cite[Theorem 1.5]{Hudson}. As long as $M$ is $1$-connected if $d=\dim M=4$, it also holds in the topological setting \cite{Pedersen1976}.
\end{rem}

\subsection{Reduction to geometric codimension zero embeddings} \label{codimensionsection}
Let $\iota: P\xhookrightarrow{}M$ be as in the statement of Theorem \ref{EmbWWIThm}. It will be convenient to be able to assume that $P\subset M$ is a codimension zero submanifold (though of handle codimension at least $3$). The following result deals with this technicality, and shows that the difference between block and ordinary \emph{smooth} embeddings is insensitive to the geometric codimension.

\begin{prop}\label{positivecodimprop}
Let $M^d$ be a compact smooth Riemannian manifold and $\iota: P^p\xhookrightarrow{} M^d$ a neat submanifold that is closed as a subspace. Let $\overline{\nu}P$ be the closed disk bundle of the normal bundle $\nu_\iota$ of the embedding $\iota$, and let $\hat\iota: \overline{\nu}P\xhookrightarrow{} M$ be the induced embedding. Then the square
\begin{equation}\label{noncodim0embeddingsfibseq}
\begin{tikzcd}
    \emb_{\partial_0,\hat\iota}(\overline{\nu}P,M)\rar["\mathrm{res}_P"]\dar[hook] & \emb_{\partial_0,\iota}(P,M)\dar[hook]\\
    \bemb_{\partial_0,\hat\iota}(\overline{\nu}P,M)\rar["\widetilde{\mathrm{res}}_P"] &\bemb_{\partial_0,\iota}(P,M).
\end{tikzcd}
\end{equation}
is homotopy cartesian. Here the subscripts $\iota$ or $\hat\iota$ in the embedding spaces stand for the path component consisting of embeddings isotopic to $\iota$ or $\hat\iota$ (relative to $\partial_0P$).

Equivalently, by taking vertical homotopy fibres in (\ref{noncodim0embeddingsfibseq}) and noting Hudson's Theorem \ref{Hudsonstheorem} and that $\mathrm{res}_P$ and $\widetilde{\mathrm{res}}_P$ are surjective, there is a weak equivalence
$$
\hofib_\iota(\emb_{\partial_0}(P,M)\xhookrightarrow{}\bemb_{\partial_0}(P,M))\simeq \hofib_{\hat\iota}(\emb_{\partial_0}(\overline{\nu}P,M)\xhookrightarrow{}\bemb_{\partial_0}(\overline{\nu}P,M)).
$$
\end{prop}

\begin{proof}
    We will show that the horizontal homotopy fibre of the vertical inclusions in (\ref{noncodim0embeddingsfibseq}) can be identified, up to equivalence, with the identity map of the topological group $\mathrm{Aut}_{\partial_0}(\nu_\iota)$ of bundle automorphisms of $\nu_\iota$ which are standard near $\partial_0P$. In particular, the total homotopy fibre of (\ref{noncodim0embeddingsfibseq}) will be weakly contractible.
    
    We first deal with the top horizontal homotopy fibre. Consider the fibration
    $$
E:=\left\{\begin{tikzcd}
    \nu_{\iota}\rar[hook, "G"]\dar & \tau_M\dar\\
    P\rar[hook, "\varphi"] &M
\end{tikzcd}\;\middle|\;
\begin{aligned}
    &\varphi\in\emb_{\partial_0,\iota}(P,M),\\
    &G\in \operatorname{BunInj}_{\partial_0}(\nu_\iota,\tau_M),\\
    &D\varphi\oplus G: \tau_P\oplus\nu_{\iota}\cong \varphi^*\tau_M.
\end{aligned}
\right\}\overset{r}\longrightarrow \emb_{\partial_0,\iota}(P,M),\quad (G,\varphi)\mapsto \varphi.
    $$
    Taking derivatives at the zero section of $\overline{\nu}P$ defines a map $D: \emb_{\partial_0,\hat\iota}(\overline{\nu}P,M)\to E$ over $\emb_{\partial_0,\iota}(P,M)$. A homotopy inverse $E\to \emb_{\partial_0,\hat\iota}(\overline{\nu}P,M)$ to $D$ can be defined using the exponential map. Therefore the homotopy fibre of $\mathrm{res}_P$ is equivalent to the fibre of $r$ (observe that $r$ is a fibration). Now $\iota^*\tau_M$ is already identified with $\tau_P\oplus\nu_\iota$, so the fibre $F:=r^{-1}(\iota)$ can be described as the subspace of bundle automorphisms of $\tau_P\oplus\nu_\iota$ over $P$ which are the identity on the tangent summand $\tau_P$ (and near $\partial_0P$). As the space of bundle maps $\nu_\iota\to\tau_P$ over $P$ is contractible, it follows that the inclusion $\mathrm{Aut}_{\partial_0}(\nu_\iota)\xhookrightarrow{}F$ is a homotopy equivalence.

    The argument for the bottom map of (\ref{noncodim0embeddingsfibseq}) is similar but trickier; we work with the simplicial model of block embeddings of Definition \ref{simplicialboundeddefn}. First let $\xi$ and $\pi$ be vector bundles over spaces $B$ and $B'$, respectively, and fix some bundle map $I: \xi\to \pi$. For any closed subset $\partial_0\subset B$, let $\operatorname{\widetilde{BunMap}}_{\partial_0}(\xi,\pi)_\bullet$ denote the semi-simplicial set whose $n$-simplices consist of bundle maps $G:\Delta^n\times \xi\to \tau_{\Delta^n}\boxplus \pi:=(\tau_{\Delta^n}\times B')\oplus (\Delta^n\times \pi)$ such that 
    \begin{itemize}
        \item $G$ agrees with $\mathbf{0}_{\Delta^n}\boxplus I$ near $\Delta^n\times\partial_0$, where $\mathbf{0}_{\Delta^n}: \epsilon^0_{\Delta^n}\cong \Delta^n\to \tau_{\Delta^n}$ is the inclusion as the zero section, and

        \item for every face $\sigma\subset \Delta^n$, $G(\sigma\times \xi)\subset \tau_\sigma\boxplus \pi\subset \tau_{\Delta^n}\boxplus \pi$.
    \end{itemize}
    Given a map $i: B\to B'$ which agrees with the underlying map of $I$ on $\partial_0\subset B$, let $\operatorname{\widetilde{BunMap}}_{\partial_0}(\xi,\pi;i)_\bullet$ the semi-simplicial subset consisting of those bundle maps $G$ whose underlying map on the base spaces $\Delta^n\times B\to \Delta^n\times B'$ is $\Id_{\Delta^n}\times i$. Let $\operatorname{\widetilde{BunInj}}_{\partial_0}(\xi,\pi)_\bullet$ and $\operatorname{\widetilde{BunInj}}_{\partial_0}(\xi,\pi;i)_\bullet$ be the semi-simplicial subsets of those bundle maps that are fibrewise injective. Then again, taking derivatives at the zero section of $\Delta^\bullet\times \overline{\nu}P$ yields a simplicial map $\widetilde{D}_\bullet$ from $\bemb_{\partial_0,\hat\iota}(\overline{\nu}P,M)_\bullet$ to a semi-simplicial set $\widetilde{E}_\bullet$ whose $n$-simplices are
    $$
\widetilde{E}_n:=\left\{\begin{tikzcd}
    \Delta^n\times\nu_{\iota}\rar[hook, "G"]\dar & \tau_{\Delta^n}\boxplus\tau_M\dar\\
    \Delta^n\times P\rar[hook, "\varphi"] &\Delta^n\times M
\end{tikzcd}\;\middle|\;
\begin{aligned}
    &\varphi\in\bemb_{\partial_0,\iota}(P,M)_n,\\
    &G\in \operatorname{\widetilde{BunInj}}_{\partial_0}(\nu_\iota,\tau_M)_n,\\
    D\varphi\oplus G&: \tau_{\Delta^n}\boxplus(\tau_P\oplus\nu_{\iota})\cong \varphi^*(\tau_{\Delta^n}\boxplus\tau_M).
\end{aligned}
\right\},
    $$
    and whose face maps are given by restriction to face strata. The map $\widetilde{r}_\bullet:\widetilde{E}_\bullet\to \bemb_{\partial_0}(P,M)_\bullet$ given by $\widetilde{r}(G,\varphi):=\varphi$ is now a Kan fibration, and $\widetilde{r}_\bullet\circ\widetilde{D}_\bullet=\widetilde{\mathrm{res}}_P$. By a similar argument as in the previous case, the homotopy fibre of $\widetilde{\mathrm{res}}_P$ is equivalent to the fibre of $\widetilde{r}_\bullet$. Using the canonical identification
    $$
(\Id_{\Delta^n}\times\iota)^*(\tau_{\Delta^n}\boxplus \tau_M)=\tau_{\Delta^n}\boxplus\iota^*\tau_M\cong \tau_{\Delta^n}\boxplus(\tau_P\oplus\nu_\iota),
    $$
    the fibre $\widetilde{F}_\bullet:=\widetilde{r}^{-1}_\bullet(\iota)$ is isomorphic to the semi-simplicial subset of $\operatorname{\widetilde{BunInj}}_{\partial_0}(\nu_\iota, \iota^*\tau_M; \Id_P)_\bullet$ of bundle maps 
     $$
     G=G_{\Delta^n}\oplus G_\tau\oplus G_\nu: \Delta^n\times \nu_\iota\longrightarrow \tau_{\Delta^n}\boxplus(\tau_P\oplus \nu_\iota)=(\tau_{\Delta^n}\times P)\oplus(\Delta^n\times\tau_P)\oplus(\Delta^n\times\nu_\iota)
     $$
for which $G_\nu$ is an isomorphism. Thus it follows that 
$$
\widetilde{F}_\bullet=\operatorname{\widetilde{BunMap}}_{\partial_0}(\nu_\iota,\tau_P;\Id_P)_\bullet\times\operatorname{Aut}_{\partial_0}(\nu_\iota)_\bullet,
$$
where the boundary condition on $\operatorname{\widetilde{BunMap}}_{\partial_0}(\nu_\iota,\tau_P;\Id_P)_\bullet$ forces bundle maps to be zero near $\Delta^\bullet\times\partial_0P$. Clearly $\operatorname{\widetilde{BunMap}}_{\partial_0}(\nu_\iota,\tau_P;\Id_P)_\bullet$ is weakly contractible: indeed given an $n$-cycle $G$ in this semi-simplicial set, a nullhomotopy of $G$ is roughly given by regarding $\Delta^{n+1}$ as $(\Delta^n\times [0,1],\Delta^{n}\times\{0\})$ and applying $t\cdot G$ on $\Delta^{n}\times\{t\}$, for $0\leq t\leq 1$. Therefore $|\widetilde{F}_\bullet|\simeq |\mathrm{Aut}_{\partial_0}(\nu_\iota)_\bullet|=\mathrm{Aut}_{\partial_0}(\nu_\iota)$, as required.
\end{proof}

\begin{rem}\label{codimzerotopremark}
    Proposition \ref{positivecodimprop} is false in the topological (and $PL$) setting. First of all, a locally flat embedding $\iota: P^p\xhookrightarrow{}M^d$ does not always admit a normal microbundle (cf \cite{RourkeSanderson}; they do admit one stably though \cite[Theorem B]{Hirschmicrobundle}). But even if it did, the statement would still not hold in general: the homotopy fibre of $\emb_{\partial_0}^{\rmTop}(\overline{\nu}P,M)\to \emb_{\partial_0}^{\rmTop}(P,M)$ is a section space of a bundle over $P$ whose fibre is the topological group $\rmTop(d,p)$ of homeomorphisms of $\R^d$ that fix pointwise the subspace $\R^p\times \{0\}$, whereas the homotopy fibre of $\bemb_{\partial_0}^{\rmTop}(\overline{\nu}P,M)\to \bemb_{\partial_0}^{\rmTop}(P,M)$ is a similar section space, but of a bundle whose fibre is the colimit \[
    \begin{tikzcd}[column sep = 35pt]
        \widetilde{\rmTop}(d-p):=\operatorname{colim}\big(\rmTop(d,p)\rar["-\times \Id_{\R}"]&\rmTop(d+1,p+1)\rar["-\times \Id_{\R}"]&\rmTop(d+2,p+2)\rar["-\times \Id_{\R}"]&\dots\big).
    \end{tikzcd}
    \]
    The map $\rmTop(d,p)\to \widetilde{\rmTop}(d-p)$ is \emph{not} an equivalence. (Crucially, though, the smooth analogues $O(d-p)=O(d-p,0)\to \dots\to O(d+n,p+n)$ are indeed equivalences.)

    To see this, consider the case $(M,P)=(D^d,D^p)$ for $p\leq d-3$. Both $\emb_{\partial}^{\mathrm{Top}}(D^p,D^d)$ and $\bemb_{\partial}^{\mathrm{Top}}(D^p,D^d)$ are contractible by the Alexander trick. However, using the topological version of Theorem \ref{EmbWWIThm} (cf Remark \ref{remembWWI}), we will see in Remark \ref{TopLongKnotRem} that the homotopy fibre of the map
    $$
\emb_{\partial_0}^{\mathrm{Top}}(D^p\times D^{d-p}, D^d)\longrightarrow \bemb_{\partial_0}^{\mathrm{Top}}(D^d\times D^{d-p},D^d)
    $$
    is not contractible. In particular, the topological analogue of the square (\ref{noncodim0embeddingsfibseq}) cannot possibly be homotopy cartesian in this case.
\end{rem}

\subsection{Last ingredients}\label{lastingrsection} From now on, let $\iota: P^d\xhookrightarrow{}M^d$ be a codimension zero closed embedding that meets $\partial M$ transversely in $\partial_0 P$, and denote by $p$ the handle dimension of $P$ relative to $\partial_0 P$; we will write $\overline{M-P}$ instead of the isotopy equivalent manifold $M-\nu P$ to emphasise that $P$ has codimension zero in $M$. It suffices to prove Theorem \ref{EmbWWIThm} in this case by Proposition \ref{positivecodimprop}. We now present the last necessary prelimary results.

\subsubsection{Parametrised isotopy extension theorem}
The \textit{parametrised isotopy extension theorem} states that for $\varphi_t: P\xhookrightarrow{} M$ any continuous family of embeddings parametrised by $t\in \Delta^k$ (with $P$ \textit{compact}), there exists a continuous family of diffeomorphisms $\{\phi_t\}_{t\in \Delta^k}$ of $M$ (which are the identity away from a compact set of $M$) such that $\phi_0=\mathrm{Id}_M$ and $\phi_t(\varphi_0(x))=\varphi_t(x)$ for all $(x,t)\in P\times \Delta^k$. Moreover, if  $K\subset\Delta^k$ is some contractible subcomplex containing the $0$-th vertex and $\{\phi'_t\}_{t\in K}$ is another continuous family of diffeomorphisms of $M$ parametrised by $K$ such that $\phi'_0=\Id_M$ and $\phi'_t(\varphi_0(x))=\varphi_t(x)$ for all $(x,t)\in P\times K$, then we can arrange $\{\phi_t\}_{t\in \Delta^k}$ as above to agree with $\{\phi'_t\}_{t\in K}$ on $K$. A consequence of this fact due to Palais \cite{Palais} (see \cite{Lima1963/64} for a simple proof) is that the restriction map $\diff_\partial(M)\to \emb_{\partial_0}(P,M)$ is a locally trivial fibre bundle with $\diff_{\partial}(\overline{M-P})$ as fibre. Such a fibration can be delooped to the homotopy fibre sequence
\begin{equation}\label{normalIET}
\begin{tikzcd}
    \emb_{\partial_0,\langle\iota\rangle}(P,M)\rar & B\diff_\partial(\overline{M-P})\rar &B\diff_\partial(M),
\end{tikzcd}
\end{equation}
where the subscript $\langle\iota\rangle$ stands for the union of all the components in $\emb_{\partial_0}(P,M)$ that contain embeddings of the form $\phi\circ\iota$ for $\phi\in \diff_\partial(M)$. By replacing $P$ and $M$ in (\ref{normalIET}) by $P\times I$ and $M\times I$, and modifying the boundary conditions, we get a similar homotopy fibre sequence
\begin{equation}\label{ConcIET}
\begin{tikzcd}
    \cemb(P,M)\rar & BC(\overline{M-P})\rar &BC(M).
\end{tikzcd}
\end{equation}
Note that $\cemb(M,P)$ is connected by Hudson's Theorem \ref{Hudsonstheorem}. Finally, there is a block analogue of (\ref{normalIET}).

\begin{prop}
There is a homotopy fibre sequence
    \begin{equation}\label{blockIET}
\begin{tikzcd}
    \bemb_{\partial_0,\langle\iota\rangle}(P,M)\rar &B\bdiff_{\partial}(\overline{M-P})\rar&B\bdiff_{\partial}(M).
    \end{tikzcd}
\end{equation}
\end{prop}

\begin{proof}
    There is a right action of the simplicial group $\bdiff_\partial(\overline{M-P})_\bullet$ on $\bdiff_\partial(M)_\bullet$; we will write $\bdiff_\partial(M)_\bullet/\bdiff_\partial(\overline{M-P})_\bullet$ for the simplicial set of (level-wise) cosets of this right action. The geometric realisation $|\bdiff_\partial(M)_\bullet/\bdiff_\partial(\overline{M-P})_\bullet|$ of this simplicial set is homotopy equivalent to the homotopy fibre of the right map of (\ref{blockIET}), so it suffices to show that the action map
    $$
a: \bdiff_\partial(M)_\bullet/\bdiff_\partial(\overline{M-P})_\bullet\longrightarrow \bemb_{\partial_0,\langle\iota\rangle}(P,M)_\bullet, \quad [\phi]\longmapsto \phi\circ \iota
    $$
    is an isomorphism. It is visibly injective, for if $\phi\circ\iota=\psi\circ\iota$ for $\phi,\psi\in \bdiff_\partial(M)_\bullet$, then $\psi^{-1}\circ\phi\in \bdiff_\partial(\overline{M-P})_\bullet$ and hence $[\psi]=[\psi\circ\psi^{-1}\circ\phi]=[\phi]$ in $\bdiff_\partial(M)_\bullet/\bdiff_\partial(\overline{M-P})_\bullet$.

    For surjectivity, let $\varphi$ be some $k$-simplex in $\bemb_{\partial_0,\langle\iota\rangle}(P,M)_\bullet$. Then there exists some $\phi\in \bdiff_\partial(M)_k$ for which   $\varphi$ and $\phi\circ\iota$ lie in the same component in $\bemb_{\partial_0}(P,M)_\bullet$. Then $\varphi':=\phi^{-1}\circ\varphi\in \bemb_{\partial_0,\iota}(P,M)_k$ and, in fact, we can arrange that its restriction to the zero-th vertex $\varphi'_0$ is $\iota$ by rechoosing $\phi$ (if necessary) using the isotopy extension theorem. Then applying the isotopy extension theorem to $\varphi'$ restricted to each of the faces that contains the $0$-th vertex, inductively on the dimension of the face, we obtain some $\Phi'\in \bdiff_\partial(M)_k$ such that $\Phi'\mid_{P\times \Delta^k}\equiv\varphi'$. Then $\Phi:=\phi\circ\Phi'\in \bdiff_\partial(M)_k$ is such that $\Phi\mid_{P\times \Delta^k}\equiv \varphi$, as desired.
\end{proof}

\begin{rem}\label{topIETrem}
    There also exist topological and $PL$ versions of the isotopy extension theorem (cf \cite[Corollary~1.4]{TopIsotopyExtensionTheorem} and \cite{PLHudsonIET}, respectively). The same proof as above also works in the topological or $PL$ setting. 
\end{rem}

\begin{rem}[Speculative]
    Weiss and Williams point out in \cite[Section 1]{WWI} that an analogue of the (parametrised) isotopy extension theorem in the bounded setting does not hold (see \cite[Chapter 8, Exercise 9]{Hirsch1976} for a counterexample in codimension $2$). However, we believe that a weaker version of the theorem should still hold: namely, for $V\in \mathcal J_0$ define the bounded embedding space $\emb^b_{\partial_0}(P\times V,M\times V)$ as in Definition \ref{topboundeddefn}. Then, there should be a homotopy fibre sequence
    $$
\begin{tikzcd}[scale = 0.95, column sep = 12pt]
    \emb^b_{\partial_0,\langle\iota\rangle}(P\times V, M\times V)\rar & E(V)=B\diff_\partial^b(\overline{M-P}\times V)\rar &B(V)=B\diff_\partial^b(M\times V),
\end{tikzcd}
    $$
where $E(-)$ and $B(-)$ are as in Notation \ref{FEBorthfunctorsnotn}. We will not give a proof of this claim, as it seems rather technical and we will not need it for the argument of Theorem \ref{EmbWWIThm}. The reader may however find it useful to think of the orthogonal functor $F(-):=\hofib(E(-)\to B(-))$ as $\emb^b_{\partial_0,\langle\iota\rangle}(P\times (-), M\times (-))$.
    
\end{rem}

\subsubsection{Alexander trick-like equivalences} For $V\in \mathcal J_0$, let $D(V)\subset V$ denote the correspoding closed unit disk (so that $D^k=D(\R^k)$). The following is proved in Propositions 1.8, 1.10 and 1.12 of \cite{WWI}. Even though we state it for the orthogonal functor $B(-)$ of Notation \ref{FEBorthfunctorsnotn}, it of course holds for $E(-)$ too.

\begin{prop}\label{alexpropBfunctor}
    For $V\in \mathcal J_0$, the Alexander trick-like map
    $$
\mathrm{alex}: C(M\times D(V))\overset{\sim}\longrightarrow \Omega^{V\oplus \R}B^{(1)}(V)=\Omega^{V\oplus \R}\left(\diff^b_\partial(M\times V\oplus\R)/\diff^b_\partial(M\times V)\right)
    $$
is a weak equivalence. Moreover, there is a homotopy commutative diagram
\begin{equation}\label{commsquareBconc}
\begin{tikzcd}
    C(M\times D(V))\dar["\mathrm{alex}", "\vsim"']\rar["\Sigma"] & C(M\times D(V)\times D^1)\rar[equal,"\sim"] &C(M\times D(V\oplus\R))\dar["\mathrm{alex}", "\vsim"']\\
    \Omega^{V\oplus \R} B^{(1)}(V)\ar[rr,"s^{\vee}_V"]&& \Omega^{V\oplus \R^2}B^{(1)}(V\oplus\R),
\end{tikzcd}
\end{equation}
where $\Sigma$ denotes the usual concordance stabilisation map and $s^\vee_V$ is the adjoint of the structure map (\ref{orthstructuremap}) for the orthogonal spectrum $\Theta B^{(1)}=\Hsp(M)$.
\end{prop}
We will describe the map ``alex'' below, but first note a few remarks on this statement and its consequences.

\begin{rem}\label{deloopingrem}
Both the domain and codomain of the map ``alex'' of Proposition \ref{alexpropBfunctor} are group-like $\mathbb{E}_1$-spaces; the former by composition of concordance diffeomorphisms, and the latter by the loop space structure induced by $\Omega^{\R}(-)$.  In Section \ref{hcobsection}, we construct a (non-connected) delooping of this map (see (\ref{alexhcob})).

It seems likely that the homotopy commutative square (\ref{commsquareBconc}) can also be delooped in a similar manner. Proving this, however, is quite technical and we will not need it in any case. What we will need instead is the observation that if $M$ is replaced by $M\times I$ in Proposition \ref{alexpropBfunctor}, the whole statement can be delooped once with respect to the $\mathbb{E}_1$-structures induced by \emph{stacking in the $I$-direction}. This is straightforward to check from the proofs in \cite{WWI}. 
\end{rem}

\begin{rem}\label{justificationrem2}
    It follows from this proposition that there is a natural\footnote{Natural for codimension-zero embeddings.} equivalence
    $$
\Omega^{\infty+1}\Hsp(M)\simeq \mathcal C(M):=\underset{k}{\hocolim}\ C(M\times D^k).
    $$
As pointed out in Remark \ref{deloopingrem}, this equivalence can be delooped once if we replace $M$ by $M\times I$. Moreover, the (non-equivariant) homotopy types of both $\Hsp(-)$ and $\mathcal{C}(-)$ are invariant under crossing with $I$, namely, there are natural equivalences $\Hsp(M\times I)\simeq \Hsp(M)$ (by Lemma \ref{MIKlem} below) and $\mathcal{C}(M\times I)\simeq \mathcal{C}(M)$ (by definition). By this line of reasoning, we obtain natural equivalences
$$
\Omega^\infty_0\Hsp(M)\simeq \Omega^\infty_0\Hsp(M\times I)\simeq B\mathcal{C}(M\times I)\simeq B\mathcal{C}(M)
$$
By \cite[Proposition 2.1]{VogellInvolution}, $B\mathcal C(M)$ is also naturally equivalent to the basepoint component of the space of stable $h$-cobordisms $\mathcal H(M)$ of Remark \ref{justificationrem1}, and by \cite[Corollary 5.6]{WWI} and the $s$-cobordism theorem, the groups $\pi_0^s(\Hsp(M))$ and $\pi_0(\mathcal{H}(M))$ are both isomorphic to the Whitehead group $\Wh(\pi_1M)$. Since there is an (non-natural) equivalence \emph{of spaces} $\Omega^\infty X\simeq \Omega^\infty_0 X\times \pi_0^s(X)$, we obtain the promised equivalence (\ref{Hspinfloopsace})
    $$
\Omega^\infty\Hsp(M)\simeq \mathcal{H}(M).
    $$
    This, of course, ought to be an equivalence of infinite loop spaces, but that seems to be more difficult to see. Making the above equivalence natural for codimension-zero embeddings requires establishing the (non-connected) delooped analogues of \eqref{commsquareBconc}, using the delooped Alexander trick-like maps (\ref{alexhcob}) constructed in Section \ref{hcobsection} (which satisfy this naturality by construction). However, this is a rather tedious task that we do not undertake.
\end{rem}

\begin{proof}[Proof of Proposition \ref{alexpropBfunctor}]
This is proved in Propositions 1.8, 1.10 and Lemma 1.12 of \cite{WWI}. Let us just explain how the Alexander trick-like map
    $$
\mathrm{alex}: C(M\times D(V))\longrightarrow \Omega^{V\oplus\R}\left(\diff^b_\partial(M\times V\oplus\R)/\diff^b_\partial(M\times V)\right)
    $$
    is defined: given a concordance diffeomorphism $\phi: M\times D(V)\times I\cong M\times D(V)\times I$, extend it by $\phi\mid_{M\times D(V)\times \{1\}}\times \Id_{[1,+\infty)}$ on $M\times D(V)\times [1,+\infty)$ and by the identity elsewhere to obtain a bounded self-diffeomorphism $\widehat\phi$ of $M\times V\oplus \R$; then shift it along $V\oplus \R$ to obtain a $(V\oplus \R)$-fold loop in $\diff^b_\partial(M\times V\oplus\R)/\diff^b_\partial(M\times V)$. We refer to loc. cit. for the rest of the proofs.
\end{proof}

Taking fibres of Proposition \ref{alexpropBfunctor} for $E(-)$ and $B(-)$ yields the first part of the analogous result for $F(-)$.

\begin{prop}\label{alexmapsprop}
For $V\in \mathcal J_0$, there are weak equivalences
$$
\mathrm{alex}: \Omega\cemb(P\times D(V),M\times D(V))\overset\sim\longrightarrow \Omega^{1+V}F^{(1)}(V),
$$
making the following diagram commute up to homotopy:
\begin{equation}\label{commsquareFconc}
\begin{tikzcd}
    \Omega\cemb(P\times D(V),M\times D(V))\dar["\mathrm{alex}", "\vsim"']\rar["\Sigma"] & \Omega\cemb(P\times D(V\oplus \R),M\times D(V\oplus \R))\dar["\mathrm{alex}", "\vsim"']\\
    \Omega^{1+V} F^{(1)}(V)\ar[r,"s^{\vee}_V"]& \Omega^{1+V\oplus \R}F^{(1)}(V\oplus\R),
\end{tikzcd}
\end{equation}
where $\Sigma$ is the concordance embedding stabilisation map of Section \ref{concembsection}. Moreover, if $p\leq d-3$, there is a natural equivalence
\begin{equation}\label{firstderivativeF}
\Omega^\infty(\CEsp(P,M)):=\Omega^\infty(\Theta F^{(1)})\simeq \mathcal{CE}\mathrm{mb}(P,M):=\underset{k}{\hocolim}\ \cemb(P\times D^k, M\times D^k).
\end{equation}
\end{prop}

To establish (\ref{firstderivativeF}), we will need the following result, which was suggested to us by Manuel Krannich.

\begin{lem}\label{ThetaF1nconnectivity}
    For $p\leq d-3$ and $d+n\geq 5$, the space $\Theta F^{(1)}_n=F^{(1)}(\R^n)$ is $n$-connected.
\end{lem}
\begin{proof}
It suffices to show that the map $\Theta E^{(1)}_n\to \Theta B^{(1)}_n$, call it $\lambda$, is such that $\pi_*(\lambda)$ is
$$
\text{($a$) surjective if $*=n+1$,}\quad \text{($b$) injective if $*=0$,} \quad \text{($c$) an isomorphism if $1\leq *\leq n$.}
$$
For ($a$), observe that $\Omega^{n+1}\lambda$ is, up to equivalence, the natural map of concordance spaces $C(\overline{M-P}\times D^n)\to C(M\times D^n)$ by Proposition \ref{alexpropBfunctor}. By exactness of
$$
\begin{tikzcd}
    \pi_0(C(\overline{M-P}\times D^n))\rar["\pi_{n+1}(\lambda)"] &\pi_0(C(M\times D^n))\rar & \pi_0(\cemb(P\times D^n, M\times D^n))=*,
\end{tikzcd}
$$
where the equality on the right is the statement of Hudson's Theorem \ref{Hudsonstheorem}, it follows that $\pi_{n+1}(\lambda)$ is surjective.

    For $(b)$ and $(c)$, consider the commutative diagram
    $$
\begin{tikzcd}
    \Theta E^{(1)}_n\rar["\lambda"]\dar["\mathrm{stab.}"]&\Theta B^{(1)}_n\dar["\mathrm{stab.}"]\\
   \Omega^{\infty}(\Sigma^n\Theta E^{(1)})\rar["\eta"] &\Omega^{\infty}(\Sigma^n\Theta B^{(1)}).
\end{tikzcd}
    $$
 We claim that the map of (non-connective) spectra $\Hsp(\overline{M-P})\to \Hsp(M)$ underlying $\eta$ is an isomorphism in $\pi_*^s$ for $*\leq 0$: indeed, the inclusion $\overline{M-P}\xhookrightarrow{}M$ is $2$-connected and $\pi_{*}^s(\Hsp(-))$ for $*\leq 0$ (see (\ref{kappagroups})) only depends on $\pi_1(-)$ by \cite[Corollary 5.6]{WWI}. So $\eta$ itself satisfies ($b$) and ($c$). By \cite[Corollary 5.8]{WWI}, both vertical maps are injective in $\pi_0$ and isomorphisms in $\pi_{1\leqs*\leq n}$ if $d+n\geq 5$. Claims $(b)$ and $(c)$ now follow.
\end{proof}

\begin{proof}[Proof of Proposition \ref{alexmapsprop}]
It remains to deloop the natural equivalence $$\Omega^{\infty+1}(\CEsp(P,M))\simeq \Omega\mathcal{CE}\mathrm{mb}(P,M)$$ obtained from the squares (\ref{commsquareFconc}), so to yield (\ref{firstderivativeF}). We do this as in Remark \ref{justificationrem2}.  

First observe that both $\Omega^\infty\CEsp(P,M)$ and $\mathcal{CE}\mathrm{mb}(P,M)$ are connected under the codimension assumption---the former by Lemma \ref{ThetaF1nconnectivity} and the latter by Hudson's Theorem \ref{Hudsonstheorem}. Just like in Remark \ref{deloopingrem}, the homotopy commutative square (\ref{commsquareFconc}) can be delooped if we replace $\iota:P\xhookrightarrow{}M$ by $\iota\times\Id_{I}:P\times I\xhookrightarrow{}M\times I$. Finally, observe that there are natural equivalences $\CEsp(P\times I,M\times I)\simeq \CEsp(P,M)$ (by Lemma \ref{MIKlem}) and $\mathcal{CE}\mathrm{mb}(P\times I,M\times I)\simeq \mathcal{CE}\mathrm{mb}(P,M)$ (by definition). We thus obtain the desired chain of natural equivalences
$$
\Omega^\infty\CEsp(P,M)\simeq \Omega^\infty\CEsp(P\times I,M\times I)\simeq \mathcal{CE}\mathrm{mb}(P\times I,M\times I)\simeq \mathcal{CE}\mathrm{mb}(P,M).
$$
\end{proof}

\subsection{The map \texorpdfstring{$\Phi^{\emb}$}{Phi emb} of Theorem \ref{EmbWWIThm}}\label{ThmAproofsubsection} Recall the map $\Phi^F_\infty$ of Proposition \ref{propunstablemapsorthcalc} for $F(-)$. Noting that $F(0)\simeq \emb_{\partial_0,\langle\iota\rangle}(P,M)$ by the isotopy extension sequence (\ref{normalIET}), this map, up to equivalence, takes the form
$$
\Phi^F_\infty: \hofib(\emb_{\partial_0,\langle\iota\rangle}(P,M)\to F(\R^\infty))\longrightarrow \Omega^\infty(\CEsp(P,M)_{hC_2}).
$$
To obtain $\Phi^{\emb}$, we need to replace $F(\R^\infty)$ by $\bemb_{\partial_0,\langle\iota\rangle}(P,M)$ above (and deal with some path component considerations). This turns out to be possible by a principle similar to that of \cite[Rem. 3.5]{WWI}.

\begin{prop}\label{hofibembprop}
    There is a map
    $$
    \begin{tikzcd}
\hofib_\iota(\emb_{\partial_0}(P,M) \to \bemb_{\partial_0}(P,M)) \rar &\hofib(F(0)\to F(\R^\infty))
\end{tikzcd}
    $$
    which is an equivalence if $d\geq 5$ and $d-p\geq 3$. If $d=4$ and $d-p\leq 3$, the map becomes an equivalence upon looping once.
\end{prop}
\begin{proof}
    First observe that the map
    \[
\hofib_\iota(\emb_{\partial_0,\langle\iota\rangle}(P,M)\to\bemb_{\partial_0,\langle\iota\rangle}(P,M))\longrightarrow \hofib_\iota(\emb_{\partial_0}(P,M)\to\bemb_{\partial_0}(P,M))
    \]
    is an inclusion of path components, and it is an equivalence if $d-p\leq 3$ by Hudson's Theorem \ref{Hudsonstheorem}. We obtain a map in the opposite direction by sending every component that is not hit to the basepoint. Therefore, it suffices to construct a homotopy commutative diagram
\begin{equation}\label{replacingFRinftysquare}
\begin{tikzcd}
    F(0)\rar &F(\R^\infty)\rar["i", "\sim"'] &\widetilde{F}(\R^\infty)\\
    \emb_{\partial_0,\langle\iota\rangle}(P,M)\uar["\vsim", "(\ref{normalIET})"']\ar[hook, rr] &&\bemb_{\partial_0,\langle\iota\rangle}(P,M),\uar[hook, "j"']
\end{tikzcd}
    \end{equation}
where the map $j$ will be an inclusion of path components if $d-p\leq 3$ and $d\geq 5$. (This will be the case when $d=4$ after looping.)

To that end, let $\mathcal J_0^\delta$ denote the underlying ordinary category of the topological category $\mathcal J_0$, and write $\widetilde{E}(-)$ and $\widetilde{B}(-)$ for the functors $\mathcal J_0^\delta\to \mathsf{Top}_*$ given by
$$
\widetilde E(V):=B|\bdiffb_{\partial}(\overline{M-P}\times V)_\bullet|,\qquad \widetilde B(V):=B|\bdiffb_{\partial}(M\times V)_\bullet|.
$$
Set $\widetilde{F}(-):=\hofib(\widetilde{E}(-)\to \widetilde{B}(-))$. Then the map $i$ of (\ref{replacingFRinftysquare}) arises as the map on homotopy fibres in
    $$
\begin{tikzcd}
    F(\R^\infty)\rar\dar[dashed, "i"] &E(\R^\infty)=B\diff^b_{\partial}(\overline{M-P}\times \R^\infty)\rar\dar[hook, "\vsim"] &B(\R^\infty)=B\diff^b_{\partial}(M\times \R^\infty)\dar[hook, "\vsim"]\\
    \widetilde{F}(\R^\infty)\rar &\widetilde{E}(\R^\infty)=B\bdiffb_{\partial}(\overline{M-P}\times \R^\infty)\rar &\widetilde{B}(\R^\infty)=B\bdiffb_{\partial}(M\times \R^\infty).
\end{tikzcd}
    $$
The middle and right vertical maps are equivalences by \cite[Theorem B]{WWI}, so $i$ is too by the five lemma. 

The map $j$ of (\ref{replacingFRinftysquare}) arises as the map on homotopy fibres in
    $$
\begin{tikzcd}
    \widetilde{F}(\R^\infty)\rar &\widetilde{E}(\R^\infty)\rar &\widetilde{B}(\R^\infty)\\
    \widetilde{F}(0)\simeq \bemb_{\partial_0,\langle\iota\rangle}(P,M)\uar[dashed, "j"]\rar &\widetilde{E}(0)=B\bdiff_{\partial}(\overline{M-P})\ar[ur, phantom, "(\dagger)"]\uar[hook]\rar &\widetilde{B}(0)=B\bdiff_{\partial}(M).\uar[hook]
\end{tikzcd}
    $$
Then the square (\ref{replacingFRinftysquare}) is the homotopy fibre of the map between the similar (strictly commutative) squares associated to $E(-)$ and $B(-)$, and so it is homotopy commutative by construction.

It remains to show that $j$ is an equivalence or, equivalently, that the square $(\dagger)$ is homotopy cartesian. Write $\widetilde{F}^{(1)}(V):=\hofib(\widetilde F(V)\to \widetilde F(V\oplus \R))$, and similarly for $\widetilde{E}^{(1)}(V)$ and $\widetilde{B}^{(1)}(V)$. In other words,
$$
\widetilde{E}^{(1)}(V):=\frac{\bdiffb_\partial(\overline{M-P}\times V\oplus \R)}{\bdiffb_\partial(\overline{M-P}\times V)}, \qquad \widetilde{B}^{(1)}(V):=\frac{\bdiffb_\partial(M\times V\oplus \R)}{\bdiffb_\partial(M\times V)}.
$$
For a group $\pi$ and an integer $j\leq 1$, set $\kappa_j(\pi):=\pi_j^s(\Whsp(B\pi))$. More explicitly,
\begin{equation}\label{kappagroups}
\kappa_j(\pi)=\left\{\begin{array}{cc}
    \Wh_1(\pi), & j=1, \\
    \widetilde{K}_0(\Z\pi), & j=0,\\
    K_j(\Z\pi), & j\leq -1.
\end{array}
\right.
\end{equation}
It was shown in \cite[Corollary 5.5]{WWI} (see also \cite{PedersenAnderson}) that, for a certain $C_2$-action on $\kappa_j(\pi)$, there are maps for $n\geq 0$
\begin{align*}
    \beta&:\pi_*(\widetilde{E}^{(1)}(\R^n))\longrightarrow H_{*}(C_2; \kappa_{1-n}(\pi_1(M-P))),\\
    \beta&:\pi_*(\widetilde{B}^{(1)}(\R^n))\longrightarrow H_{*}(C_2; \kappa_{1-n}(\pi_1(M))), 
\end{align*}
which, as long as $d+n\geq 5$, are injective if $*=0$ and isomorphisms if $*\geq 1$. Moreover, it is not difficult to see from its proof that these are compatible, in the sense that the square
$$
\begin{tikzcd}
    \pi_*(\widetilde{E}^{(1)}(\R^n))\rar\dar["\beta"] & \pi_*(\widetilde{B}^{(1)}(\R^n))\dar["\beta"]\\
    H_{*}(C_2; \kappa_{1-n}(\pi_1(M-P)))\rar["\cong"] & H_{*}(C_2; \kappa_{1-n}(\pi_1(M)))
\end{tikzcd}
$$
is commutative. The lower horizontal map is an isomorphism because the fundamental groups of $M-P$ and $M$ can be identified under the obvious inclusion by the assumption that $p\leq d-3$. Hence, as $\widetilde{F}^{(1)}(V)\to \widetilde{E}^{(1)}(V)\to \widetilde{B}^{(1)}(V)$ is a homotopy fibre sequence for all $V$, it follows that $\widetilde{F}^{(1)}(\R^n)$ is weakly contractible for all $n\geq 0$ with $d+n\geq 5$. If $d\geq 5$, using the homotopy fibre sequences
$$
\begin{tikzcd}
     \hofib(\widetilde{F}(0)\to\widetilde F(\R^n))\rar & \hofib(\widetilde{F}(0)\to\widetilde F(\R^{n+1}))\rar & \widetilde{F}^{(1)}(\R^n)\simeq *
\end{tikzcd}
$$
for $n\geq 0$, we must have by induction that $\hofib(\hspace{1pt}j:\widetilde{F}(0)\to\widetilde F(\R^\infty))$ is contractible, ie, that $j$ is an inclusion of path components, as desired.

If $d=4$, the argument above shows that $\widetilde{F}^{(1)}(\R^n)\simeq *$ for $n\geq 1$. We claim further that $\Omega \widetilde{F}^{(1)}(0)$ is also contractible---if so, looping \eqref{replacingFRinftysquare} once, we obtain a commutative diagram
\[
\begin{tikzcd}
    \Omega F(0)\rar &\Omega F(\R^\infty)\rar["i", "\sim"'] &\Omega\widetilde{F}(\R^\infty)\\
    \Omega \emb_{\partial_0,\langle\iota\rangle}(P,M)\uar["\vsim", "(\ref{normalIET})"']\ar[hook, rr] &&\Omega \bemb_{\partial_0,\langle\iota\rangle}(P,M),\uar[hook, "\Omega j"']
\end{tikzcd}
\]
where the rightmost vertical map $\Omega j$ is an inclusion of path components, and hence the conclusion of the statement would hold after looping once if $d-p\leq 3$.

To show that $\Omega \widetilde{F}^{(1)}(0) \simeq *$, it suffices to prove that both $\Omega \widetilde{E}^{(1)}(0)$ and $\Omega \widetilde{B}^{(1)}(0)$ are contractible. This follows from the block analogue of Proposition~\ref{alexpropBfunctor} (whose proof in \cite[Section 1]{WWI} applies verbatim in the block setting), which identifies $\Omega \widetilde{E}^{(1)}(0) \simeq \widetilde{C}(\overline{M-P})$ and $\Omega \widetilde{B}^{(1)}(0) \simeq \widetilde{C}(M)$. As the space of block concordances of a manifold $\widetilde{C}(X)$ is well-known to be contractible, the claim follows.
\end{proof}

\begin{defn}\label{Phiembdefn}
    The map $\Phi^{\emb}$ of Theorem \ref{EmbWWIThm} is the zig-zag 
\begin{alignat*}{1}
 \hofib_\iota(\emb_{\partial_0}(P,M)\xhookrightarrow{}\bemb_{\partial_0}(P,M))\longrightarrow \hofib(F(0)\to F(\R^\infty)) & \qquad
\text{by Proposition \ref{hofibembprop},}\\          
\xrightarrow{\Phi_\infty^F}\Omega^\infty(\Theta F^{(1)}_{hO(1)})=\Omega^\infty(\CEsp(P,M)_{hC_2}) &\quad\text{by (\ref{stableWWImaps}) and (\ref{Hspectrumdefn}).}
\end{alignat*}
By Proposition \ref{hofibembprop}, the first map is an equivalence if $d-p\geq 3$ and $d\geq 5$ (or $d=4$ after looping once).
\end{defn}

\subsection{Connectivity of the map \texorpdfstring{$\Phi^{\emb}$}{Phi emb}}\label{connectivitysection} Let $d-p\geq 3$. In this section we show that the map $\Phi^{\emb}$ just defined is $\phi_{\cemb}(M,P)$-connected, at last establishing Theorem \ref{EmbWWIThm} (modulo the proof of Proposition \ref{myAmazingTheorem}). First assume that $d\geq 5$, so that by Proposition \ref{hofibembprop}, the connectivity of $\Phi^{\emb}$ is that of $\Phi^{F}_\infty$; we show by induction on $n\geq 0$ that the maps $\Phi^{F}_n$ of Proposition \ref{propunstablemapsorthcalc} are at least $\phi_{\cemb}(M,P)$-connected. Note that this is clear for $n=0$, as both the domain and codomain are contractible. 

Suppose now that $\Phi^F_n$ is $\phi_{\cemb}(M,P)$-connected for some $n\geq 0$. To show that $\Phi^F_{n+1}$ has this connectivity, it suffices to show that the map $\mathrm{stab.}: \Theta F^{(1)}_n\to \Omega^\infty(\Sigma^n\Theta F^{(1)})$ of Proposition \ref{propunstablemapsorthcalc} is $(\phi_{\cemb}(M,P)+n)$-connected. But $\Theta F^{(1)}_n$ is $n$-connected by Lemma \ref{ThetaF1nconnectivity} and $\Sigma^n\Theta F^{(1)}$ is $(n+1)$-connective. So it suffices to show that $\Omega^{n+1}(\mathrm{stab.})$ is $(\phi_{\cemb}(M,P)-1)$-connected. This follows from the homotopy commutative diagram
\begin{equation}\label{connectivityCEmbsquare}
\begin{tikzcd}
\Omega\cemb(P\times D^n, M\times D^n)\dar["\mathrm{alex}", "\vsim"']\rar[hook] & \Omega\mathcal{CE}\mathrm{mb}(P,M)\dar["\vsim"', "(\ref{firstderivativeF})"]\\
\Omega^{n+1}\Theta F^{(1)}_n\rar["\Omega^{n+1}(\mathrm{stab.})"] & \Omega^{\infty+1}(\Theta F^{(1)}).
\end{tikzcd}
\end{equation}
By definition, the connectivity of the top horizontal map is 
$$\phi_{\cemb}(M\times D^n,P\times D^n)-1\geq \phi_{\cemb}(M,P)-1.$$
One obtains the above homotopy commutative diagram from stacking together squares of the form (\ref{commsquareFconc}) with $V=\R^k$ for $k\geq n$. This establishes Theorem \ref{EmbWWIThm} when $d\geq 5$.

Now if $d=4$, we shall show that $\Omega \Phi^{\emb}$ is $(\phi_{\cemb}(M,P)-1)$-connected---this would show that $\Phi^{\emb}$ is indeed $\phi_{\cemb}(M,P)$-connected as both the domain and codomain of the map are connected (the former by Hudson's Theorem \ref{Hudsonstheorem}, and the latter as $\CEsp(P,M)$ is $1$-connective by Lemma \ref{ThetaF1nconnectivity}, and $(-)_{hC_2}$ preserves connectivity). By Proposition \ref{replacingFRinftysquare}, the connectivity of $\Omega \Phi^{\emb}$ is that of $\Omega \Phi^{F}_\infty$ so, just like before, we need only show that the maps $\Omega \operatorname{stab.}: \Omega \Theta F^{(1)}_n\to \Omega^{\infty+1}(\Sigma^n\Theta F^{(1)})$ are $(\phi_{\cemb}(M,P)+n-1)$-connected for $n\geq 0$. When $n=0$, this follows from the homotopy commutative diagram \eqref{connectivityCEmbsquare} (for $n=0$), and for $n\geq 1$, the same argument as before applies since $\Theta F^{(1)}_n$ is indeed $n$-connected for $n\geq 1$ when $d\geq 4$. This concludes the proof of Theorem \ref{EmbWWIThm}.
\qed

\section{A splitting result for embedding spaces and the Gromoll filtration}\label{splittingknotsection}

In this section we derive, as a consequence of Theorem \ref{EmbWWIThm}, a general splitting result\footnote{This should be compared with the analogous result of Burghelea and Lashof \cite[Corollary~E]{BurgLash}.} for embedding spaces of manifolds with interval factors. This will be used for the splitting part of Theorem \ref{LongKnotsThm}. Later we discuss consequences for the Gromoll filtration of embedding spaces (cf\ Definition \ref{GromollDefn}). Throughout, let $\iota: P^p \subset M^d$ be as in the statement of Theorem \ref{EmbWWIThm}.

For $D(-)$ any of $\diff_\partial^b(-\times V)$ with $V\in \mathcal J_0$, $\bdiff_\partial(-)$ or $C(-)$, there are \textit{graphing maps}
$$
\Gamma: \Omega D(M)\longrightarrow D(M\times I)
$$
given (roughly) by regarding a $1$-parameter family of automorphisms of $M$ as an automorphism of $M\times I$ itself. These are natural with respect to codimension zero embeddings. Moreover, these maps can be delooped as, up to homotopy, they intertwine the (group-like) $\mathbb{E}_1$-structures of concatenating loops for the domain, and stacking automorphisms in the $I$-direction for the codomain. There are similar maps
\begin{equation}\label{graphmapemb}
\Gamma: \Omega E(P,M)\longrightarrow E(P\times I, M\times I)
\end{equation}
for $E(-,-)$ denoting any of $\emb_{\partial_0}(-,-)$, $\bemb_{\partial_0}(-,-)$, $\emb_{\partial_0}^{(\sim)}(-,-)$ (see (\ref{pseudoembdefn})) or $\cemb(-,-)$. In what follows, we will write $\Gamma$ for any map of this same nature.

\begin{rem}\label{colsmrem}
    Most of the functors $D(-)$ and $E(-,-)$ above either admit a point-set topological model or a simplicial model. In the first case, the graphing maps just introduced are really zig-zags of maps
    $$
    \begin{tikzcd}
    \Omega E(P,M) &\lar[hook', "\sim"'] \Omega^{\mathsf{col}, \mathsf{sm} }E(P,M)\rar["\Gamma"] &E(P\times I,M\times I),
\end{tikzcd}
$$
where, for $X$ a pointed (Fréchet) manifold, here $\Omega^{\mathsf{col}, \mathsf{sm}}X$ stands for the space of smooth loops $\gamma: S^1\to X$ which are \textit{collared} in the sense that there exists some neighbourhood of $1\in S^1$ which is sent by $\gamma$ to the basepoint in $X$. The inclusion $\Omega^{\mathsf{col}, \mathsf{sm} }E(P,M)\xhookrightarrow{}\Omega E(P,M)$ is an equivalence by smooth approximation of continuous functions. 

In the simplicial case, the graphing maps $\Gamma$ are the geometric realisations of the simplicial maps
$$
\Gamma_\bullet: (\Omega E(P,M))_\bullet\longrightarrow E(P\times I, M\times I)_\bullet
$$
that send a $q$-simplex in $(\Omega E(P,M))_\bullet$ (seen as a $(q+1)$-simplex $g\in E(P,M)_\bullet$ whose $0$-th face and vertex are the basepoint $*\in E(P,M)_\bullet$) to the $q$-simplex in $E(P\times I,M\times I)_\bullet$ obtained from $g$ by expanding out the $0$-th vertex of $\Delta^{q+1}$ to a $q$-dimensional simplex (ie regarding $\Delta^{q+1}$ as $(\Delta^q\times I, \Delta^q\times\{0\})$). 

In the cases when $D(-)$ or $E(-,-)$ admit both models, one verifies that these two graphing maps agree up to homotopy. We ignore both of these technicalities in most of what follows.
\end{rem}

\subsection{Splitting results} Observe that there is a graphing map
\begin{equation}\label{pseudoGraphingMap}
\Gamma^{(\sim)}: \Omega \pemb_{\partial_0}(P,M)\longrightarrow \pemb_{\partial_0}(P\times I,M\times I)
\end{equation}
obtained as the homotopy fibre of the ordinary and block graphing maps.

\begin{prop}\label{pseudographnullhomotopyprop}
    If $d-p\geq 3$, then the pseudoisotopy graphing map $\Gamma^{(\sim)}$ of \eqref{pseudoGraphingMap} is nullhomotopic after localising away from $2$ and taking $(\phi_{\cemb}(d+1,p+1)-1)$-th Postnikov sections.
\end{prop}

\begin{proof}
    By Proposition \ref{positivecodimprop}, we may assume that $\dim P=\dim M=d$. Then, resembling Notation \ref{FEBorthfunctorsnotn}, let $\Omega F(-)$ and $FI(-)$ denote the orthogonal functors given by
\begin{align*}
\Omega F(V):&=\Omega\hofib(B\diff^b_\partial(\overline{M-P}\times V)\to B\diff^b_\partial(M\times V)),\\
FI(V):&=\hofib(B\diff^b_\partial(\overline{M-P}\times I\times V)\to B\diff^b_\partial(M\times I\times V)).
\end{align*}
By taking fibres of the (delooped) graphing maps introduced at the beginning of the section, we obtain a natural transformation $\Gamma: \Omega F(-)\to FI(-)$ of orthogonal functors, giving rise to a map of $O(1)$-spectra $\Gamma: \Theta(\Omega F)^{(1)}\longrightarrow \Theta FI^{(1)}$ and a commutative diagram
\begin{equation}\label{OmegaFIdiagram}
\begin{tikzcd}
\pemb_{\partial_0}(P\times I, M\times I)\rar["\Phi^{FI}_\infty"] & \Omega^\infty(\Theta FI^{(1)}_{hO(1)})\rar["\mathrm{Trf}_{O(1)}"] & \Omega^\infty\Theta FI^{(1)}\\
\Omega\pemb_{\partial_0}(P, M)\uar["\Gamma"]\rar[" \Phi^{\Omega F}_\infty"] & \Omega^\infty(\Theta (\Omega F)^{(1)}_{hO(1)})\uar["\Gamma"]\rar["\mathrm{Trf}_{O(1)}"] & \Omega^\infty\Theta(\Omega F)^{(1)}\uar["\Gamma"],
\end{tikzcd}
\end{equation}
where $\mathrm{Trf}_{O(1)}$ is the $O(1)$-\textit{transfer map}. This map is injective in the homotopy category of infinite loop spaces at odd primes (it splits the quotient map $X\to X_{hC_2}$). Therefore, since the map $\Phi^{FI}_\infty$ is $\phi_{\cemb}(d+1,p+1)$-connected by Section \ref{connectivitysection} (and thus becomes an equivalence after taking $(\phi_{\cemb}(d+1,p+1)-1)$-th Postnikov sections), it will suffice to show that the rightmost vertical map in (\ref{OmegaFIdiagram}) is nullhomotopic. By (\ref{firstderivativeF}), we have that $\Omega^\infty\Theta(\Omega F)^{(1)}\simeq \Omega\mathcal{CE}\mathrm{mb}(P,M)$ and $\Omega^\infty\Theta FI^{(1)}\simeq \mathcal{CE}\mathrm{mb}(P\times I,M\times I)$ and, under these equivalences, the right vertical map in (\ref{OmegaFIdiagram}) then becomes the graphing map
\begin{equation}\label{splittingpropstablecembmap}
\Gamma: \Omega\mathcal{CE}\mathrm{mb}(P,M)\longrightarrow \mathcal{CE}\mathrm{mb}(P\times I,M\times I).
\end{equation}
This is because both the concordance stabilisation map and the Alexander trick-like map of Proposition \ref{alexmapsprop} that give rise to the previous equivalences commute on the nose with the graphing maps, ie, the following diagrams commute:
$$
\begin{tikzcd}
\Omega\cemb(P\times D^{n+1}, M\times D^{n+1})\rar["\Gamma"]& \cemb(P\times I\times D^{n+1}, M\times I\times D^{n+1})\\
    \Omega\cemb(P\times D^n, M\times D^n)\rar["\Gamma"]\uar["\Sigma"] & \cemb(P\times I\times D^n, M\times I\times D^n),\uar["\Sigma"]\\
    \Omega^n\Theta(\Omega F)^{(1)}_n\rar["\Gamma"] & \Omega^n\Theta FI^{(1)}_n\\
    \Omega\cemb(P\times D^n, M\times D^n)\rar["\Gamma"]\uar["\mathrm{alex}", "\vsim"'] & \cemb(P\times I\times D^n, M\times I\times D^n).\uar["\mathrm{alex}", "\vsim"']
\end{tikzcd}
$$

So in order to show that (\ref{splittingpropstablecembmap}) is nullhomotopic, it suffices to argue that it is so unstably, ie that the graphing maps
$$
\Gamma:\Omega\cemb(P\times D^n, M\times D^n)\longrightarrow \cemb(P\times I\times D^n, M\times I\times D^n) 
$$
are nullhomotopic for all $n\geq 0$. Replacing $M\times D^n$ by $M$, we may assume $n=0$. This claim is a consequence of the following trick, due to Oscar Randal-Williams: there is a ``$U$-shaped graphing map''
$$
U\Gamma: \Omega \cemb(P,M)\longrightarrow \cemb(P\times I, M\times I),
$$
which is homotopic to the standard $\Gamma$ by pulling down the $U$-shape to the base of the concordance. This homotopy is illustrated in Figure \ref{Ushapeconcfig}, where we replace $\cemb(-,-)$ by standard concordances $C(-)$ because it is easier to depict, but the idea is the same. Observe that, throughout the homotopy, there are no issues about smoothness in the upper corners because the concordances are equal to the identity near these. Here we are explicitly using the collared condition imposed by the functor $\Omega^{\mathsf{col,sm}}(-)$ (see Remark \ref{colsmrem}).

\begin{figure}[h]
    \centering
   \includegraphics[scale = 0.65]{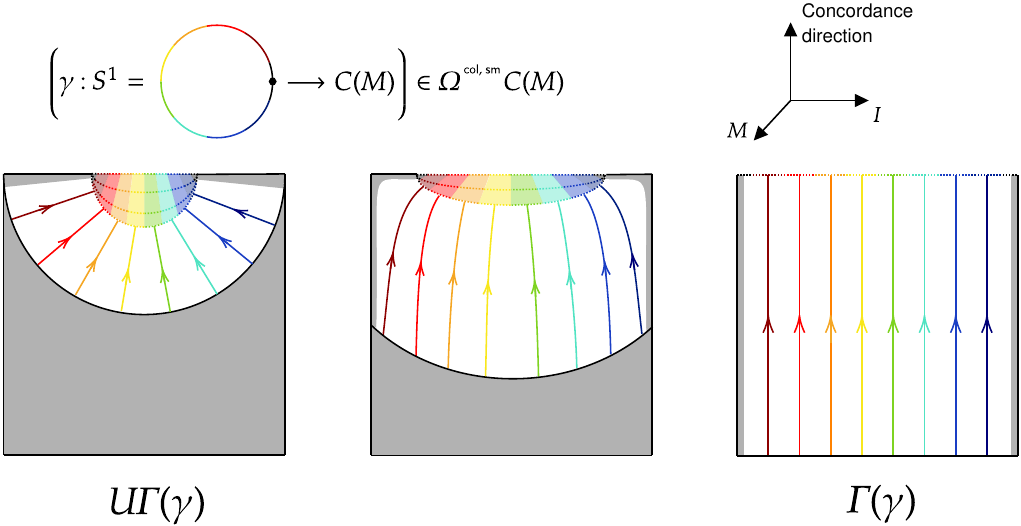}
    \caption{Images of $\gamma\in \Omega C(M)$ under the graphing maps $\Gamma$ and $U\Gamma$, and the homotopy between them. The concordances are equal to the identity on grey shaded regions.}
    \label{Ushapeconcfig}
\end{figure}

But clearly $U\Gamma$ factors through the path space $\mathrm{Map}(I,\cemb(P,M))$, and hence there is a homotopy commutative diagram
$$
\begin{tikzcd}[column sep =14pt, row sep = 8pt]
    \cemb(P\times I,M\times I)\rar[equal] & \cemb(P\times I,M\times I)&&\\
    &&\mathrm{Map}(I,\cemb(P,M))\rar["\mathrm{ev}_0", "\sim"']\ar[ul, "U\Gamma"] & \cemb(P,M),\\
    \Omega \cemb(P,M)\ar[uu, "\Gamma"]\rar[equal] & \Omega \cemb(P,M)\ar[uu, "U\Gamma"]\ar[hook, ur]\ar[urr, bend right= 7pt, "\mathrm{ev}_{0}= *"']&&
\end{tikzcd}
$$
which exhibits the leftmost vertical map as nullhomotopic, as desired. This establishes the statement.
\end{proof}


As a consequence of this proposition, we obtain the following splitting result for embedding spaces.

\begin{thm}\label{splittingthm}
Let $I$ and $J$ both denote closed intervals. For $p\leq d-3$, $N:=\phi_{\cemb}(d+1,p+1)-2$ and $N':=\phi_{\cemb}(d+2,p+2)-1$, there are equivalences away from $2$
\begin{align*}
&\tau_{\leqs N} \big(\Omega \emb_{\partial_0}(P\times I, M\times I)\big)\simeq_{[\frac{1}{2}]}\tau_{\leqs N}\left(\Omega\bemb_{\partial_0}(P\times I,M\times I)\times \Omega \pemb_{\partial_0}(P\times I,M\times I)\right),\\
&\tau_{\leqs N'} \emb_{\partial_0}(P\times I\times J, M\times I\times J)\simeq_{[\frac{1}{2}]}\tau_{\leqs N'}
\begin{pmatrix}
    \bemb_{\partial_0}(P\times I\times J,M\times I\times J)\\
    \times\\
    \pemb_{\partial_0}(P\times I\times J,M\times I\times J)
\end{pmatrix}.
\end{align*}
These splittings are compatible with graphing maps, in the sense that
\[
\begin{tikzcd}[scale cd= 0.9]
    \tau_{\leqs N} \big(\Omega \emb_{\partial_0}(P\times I, M\times I)\big)\dar["\Gamma"]\rar[phantom, "\simeq_{[\frac{1}{2}]}"] &\tau_{\leqs N}\left(\Omega\bemb_{\partial_0}(P\times I,M\times I)\times \Omega \pemb_{\partial_0}(P\times I,M\times I)\right)\dar["\widetilde{\Gamma}\times \Gamma^{(\sim)}\ \sim\ \widetilde{\Gamma}\times *"]\\
    \tau_{\leqs N} \emb_{\partial_0}(P\times I\times J, M\times I\times J)\rar[phantom, "\simeq_{[\frac{1}{2}]}"] & \tau_{\leqs N} \left(
    \bemb_{\partial_0}(P\times I\times J,M\times I\times J) \times 
    \pemb_{\partial_0}(P\times I\times J,M\times I\times J)\right) 
\end{tikzcd}
\]
is homotopy commutative. (The homotopy in the rightmost vertical map follows from Proposition \ref{pseudographnullhomotopyprop}.)
\end{thm}

\begin{proof}
Suppose given a map of fibration sequences 
$$
\begin{tikzcd}
    F'\rar & E'\rar["p'"] & B'\\
    F\rar\uar["f\simeq *"] &E\rar["p"]\uar["e"]&B\uar["b", "\vsim"']
\end{tikzcd}
$$
such that $f$ is nullhomotopic and $b$ is an equivalence. If $\delta: \Omega B\to F$ (and similarly for $\delta'$) denotes the connecting map, then it follows that $\delta' \circ \Omega b\simeq f\circ \delta\simeq *$, and thus $\Omega b$ lifts, up to homotopy, to a map $\tilde\sigma: \Omega B\to \Omega E'$. Then for $(\Omega b)^{-1}$ any homotopy inverse to $\Omega b$, the map $\sigma:= \tilde\sigma\circ (\Omega b)^{-1}: \Omega B'\to \Omega E'$ is a homotopy section of the fibration $\Omega F'\to \Omega E'\to \Omega B'$ and so provides a splitting $
\Omega E'\simeq \Omega B'\times \Omega F'
$. This observation, applied to the map of fibre sequences obtained from
$$
\begin{tikzcd}
\pemb_{\partial_0}(P\times I,M\times I)\rar & \emb_{\partial_0}(P\times I, M\times I)\rar & \bemb_{\partial_0}(P\times I, M\times I)\\
\Omega\pemb_{\partial_0}(P,M)\uar["\Gamma^{(\sim)}"]\rar & \Omega\emb_{\partial_0}(P, M)\uar["\Gamma"]\rar &\Omega\bemb_{\partial_0}(P,M)\uar["\widetilde{\Gamma}", "\vsim"']
\end{tikzcd}
$$ 
by localising away from $2$ and taking Postnikov $(N+1)$-th sections, yields the first equivalence in the statement by Proposition \ref{pseudographnullhomotopyprop}. (Note that, for any space $X$, we have $\tau_{\leqs N}(\Omega X)\simeq \Omega\tau_{\leqs N+1} X$.)

To obtain the second equivalence, observe that for $E(-,-)$ any of the mapping spaces involved in the proof of Proposition \ref{pseudographnullhomotopyprop}, the space $E(P\times J,M\times J)$ is a group-like topological monoid with respect to stacking in the $J$-direction. Then replacing $(M,P)$ by $(M\times J, P\times J)$, one checks that each of the steps in the argument of Proposition \ref{pseudographnullhomotopyprop} can be delooped with respect to this $\mathbb{E}_1$-structure. This results in getting rid of the loopings in the first equivalence of the statement, thus yielding the second one.

For the compatibility part of the statement, note that given a diagram of fibre sequences
\[
    \begin{tikzcd}
    F''\rar & E''\rar["p''"] & B''\\
    F'\rar\uar["f'"] &E'\rar["p'"]\uar["e'"]&B\uar["b'", "\vsim"']\\
    F\rar\uar["f\simeq *"] &E\rar["p"]\uar["e"]&B,\uar["b", "\vsim"']
\end{tikzcd}
\]
the splitting of the top fibre sequence (upon looping) provided by the nullhomotopy $f'\circ f\simeq f'\circ *=*$ is compatible with that of the middle fibre sequence. If, moreover, $f'$ is nullhomotopic and we have a homotopy of nullhomopies of $f'\circ f$ from $f'\circ f\simeq f'\circ *=*$ to $f'\circ f\simeq *\circ f=*$, then the splitting of the top fibre sequence induced by $f'\simeq *$ is compatible with that induced by $f'\circ f\simeq f'\circ *=*$. The compatibility part in the statement of the theorem follows from this observation, since it is clear from construction that the two nullhomotopies of the composition $\Omega^2\cemb(P,M)\to \Omega \cemb(P\times J, M\times J)\to \cemb(P\times I\times J, M\times I\times J)$ are themselves homotopic (both come from different deformation retractions of the space $I\times J$). This finishes the proof of the theorem.
\end{proof}

The following result will be used to establish the splitting part of Theorem \ref{LongKnotsThm}.

\begin{cor}\label{splittinglongknotcor}
    For $2\leq p\leq d-3$ and $N:=\phi_{\cemb}(d,p)-1$, there is an equivalence away from $2$
$$
\tau_{\leqs N} \emb_{\partial}(D^p, D^d)\simeq_{[\frac{1}{2}]}\tau_{\leqs N}\left(\bemb_{\partial}(D^p,D^d)\times \pemb_{\partial}(D^p,D^d)\right).
$$
For $p=1$, such equivalence exists only after looping, ie, 
$$
\Omega\tau_{\leqs N} \emb_{\partial}(D^1, D^d)\simeq_{[\frac{1}{2}]}\Omega\tau_{\leqs N}\left(\bemb_{\partial}(D^1,D^d)\times \pemb_{\partial}(D^1,D^d)\right).
$$
\end{cor}
\begin{proof}
    When $p\geq2$, set $(M,P)=(D^{d-2}, D^{p-2})$ in the second equivalence of Theorem \ref{splittingthm}. For $p=1$, set $(M,P)=(D^{d-1},D^0)$ in the first equivalence of the same theorem.
\end{proof}

\subsection{The Gromoll filtration}

For $n\geq 5$, Gromoll introduced in \cite{GromollFiltration} a descending filtration on the group of exotic $(n+1)$-spheres $\Theta_{n+1}\cong\pi_0(\diff_\partial(D^n))$ given by
\[
0=\Gamma^{n+1}\Theta_{n+1}\subset \Gamma^n\Theta_{n+1}\subset \dots\subset \Gamma^1\Theta_{n+1}\subset \Gamma^0\Theta_{n+1}=\Theta_{n+1},
\]
where $\Gamma^j\Theta_{n+1}$ is, by definition, the image of the $j$-th graphing homomorphism
\[
\pi_0(\Gamma^j):\pi_j(\diff_\partial(D^{n-j}))\longrightarrow \pi_0(\diff_{\partial}(D^n))\cong\Theta_{n+1}.
\]

 When $n-m\geq 3$, Budney \cite[Defn. 3.7]{BudneyIntegralLongKnots} considered an analogous filtration on the abelian group $\pi_0(\emb_{\partial}(D^m,D^n))$ of isotopy classes of long knots. We will consider an even more general case.

\begin{defn}\label{GromollDefn}
    Fix $\iota: P^p\subset M^d$ as in Theorem \ref{EmbWWIThm}. A homotopy class $x\in \pi_k(\emb_{\partial_0}(P\times D^n, M\times D^n))$ is said to have \textbf{Gromoll degree $j$}, for $0\leq j\leq n$, if it is in the image of the $j$-th graphing homomorphism
    \[
    \pi_k(\Gamma^j): \pi_{k+j}(\emb_{\partial_0}(P\times D^{n-j}, M\times D^{n-j}))\longrightarrow \pi_{k}(\emb_{\partial_0}(P\times D^{n}, M\times D^{n}))
    \]
    but, if $j<n$, not in the image of the $(\hspace{1pt}j+1)$-st graphing homomorphism $\pi_k(\Gamma^{j+1})$. We will say the same for homotopy classes localised away from $2$, that is, in the group $\pi_k(\emb_{\partial_0}(P\times D^n, M\times D^n))[\frac{1}{2}]$.
\end{defn}

As a consequence of Theorem \ref{splittingthm}, we can determine the Gromoll degree of many homotopy classes away from $2$ of the embedding spaces $\emb_{\partial_0}(P\times D^n, M\times D^n)$.

\begin{cor}\label{CorGromollFiltration}
    Let $d-p\geq 3$, let $k$ be an integer such that $k\leq 2d-p-6+n$, and set
    \[
    j:=\min\left(n-1,\lfloor \tfrac{2d-p-6+n-k}{2}\rfloor\right).
    \]
    Then, any $k$-th homotopy class $x\in \pi_k(\emb_{\partial_0}(P\times D^n, M\times D^n))[\frac{1}{2}]$ has Gromoll degree either $0$ or $\geq j$. Moreover, letting $x=(a,b)$ under the splitting\footnote{The splitting of homotopy groups also holds when $n=1$ and $k=0$ since $\pemb_{\partial_0}(P\times D^n, M\times D^n)$ is connected by Hudson's Theorem \ref{Hudsonstheorem}.}
        \[
        \pi_k(\emb_{\partial_0}(P\times D^n, M\times D^n))[\tfrac{1}{2}]\cong \pi_k(\bemb_{\partial_0}(P\times D^n, M\times D^n))[\tfrac{1}{2}]\oplus \pi_k(\pemb_{\partial_0}(P\times D^n, M\times D^n))[\tfrac{1}{2}]
        \]
        of Theorem \ref{splittingthm}:
        
    \begin{itemize}[itemsep=3pt]
        \item[(i)] If $a=0$, ie, $x$ is in the image of
        \[
        \pi_k(\pemb_{\partial_0}(P\times D^n, M\times D^n))[\tfrac{1}{2}]\longrightarrow \pi_k(\emb_{\partial_0}(P\times D^n, M\times D^n))[\tfrac{1}{2}],
        \]
        then the Gromoll degree of $x$ is zero. 

        \item[(ii)] If $b=0$, then the Gromoll degree of $x$ is $\geq j$.
    \end{itemize}
\end{cor}

\begin{proof}
    First note that, by the bound \eqref{concembbound} of \cite{ConcEmbStablerange}, we have $\phi_{\cemb}(d+n,p+n)-1\geq 2d-p-6+n$ and that $j$ is the biggest integer $\leq n-1$ such that 
    \[
    k+j\leq 2d-p-6+n-j\leq \phi_{\cemb}(d+n-j,p+n-j)-1.
    \]
Thus, if $0<\ell\leq j$, it follows by the last part of Theorem \ref{splittingthm} that the $\ell$-th graphing homomorphism
    \[
    \pi_k(\Gamma^\ell): \pi_{k+\ell}(\emb_{\partial_0}(P\times D^{n-\ell}, M\times D^{n-\ell}))[\tfrac{1}{2}]\longrightarrow \pi_{k}(\emb_{\partial_0}(P\times D^{n}, M\times D^{n}))[\tfrac{1}{2}]
    \]
    is of the form
    \[
   \begin{matrix}
    \Id\\
    \oplus\\
    0
\end{matrix}: \begin{matrix}
    \pi_{k+\ell}(\bemb_{\partial_0}(P\times D^{n-j},M\times D^{n-j}))[\tfrac{1}{2}]\\
    \oplus\\
    \pi_{k+\ell}(\pemb_{\partial_0}(P\times D^{n-\ell},M\times D^{n-\ell}))[\tfrac{1}{2}]
\end{matrix}\longrightarrow
\begin{matrix}
    \pi_{k}(\bemb_{\partial_0}(P\times D^{n},M\times D^{n}))[\tfrac{1}{2}]\\
    \oplus\\
    \pi_{k}(\pemb_{\partial_0}(P\times D^{n},M\times D^{n}))[\tfrac{1}{2}],
\end{matrix} \quad \begin{pmatrix}
    a\\b
\end{pmatrix}
\longmapsto \begin{pmatrix}
    a\\0
\end{pmatrix}.
    \]
    It follows that $x=(a,b)$ is in the image of $\pi_k(\Gamma^\ell)$ for some $0<\ell\leq j$ if and only if this is the case for $\ell=j$ and $x=(a,0)$. This proves the first assertion of the statement and part (ii).

    Part (i) follows from the following elementary claim: given a commutative diagram of groups
    \[
    \begin{tikzcd}
        0\rar &A\rar["i",hook] &B\rar[two heads] &C\rar &0\\
        &A'\uar["0"]\rar & B'\rar\uar["\beta"] & C,\uar[equal] &
    \end{tikzcd}
    \]
    where the top is a short exact sequence and the bottom is only exact at $B'$, then $\im i\cap \im \beta=0$.
\end{proof}

\begin{rem}
    We will compare this result to those of Budney \cite{BudneyIntegralLongKnots} on the Gromoll filtration of $\pi_0(\emb_{\partial}(D^p,D^d))$ in Remark \ref{BudneyComparisonRem}.
\end{rem}


\section{Involutions in algebraic \texorpdfstring{$K$}{K}-theory}\label{pseudoisotopyknotsection} 

The aim of this section is to explore the involutions of the $C_2$-spectra involved in the statements of Theorems \ref{WWbdiffmoddiff} and \ref{EmbWWIThm}, and to express them in terms of simpler and more computable involutions coming from algebraic $K$-theory---the main result in this direction is Theorem \ref{absolutepropTWWvsTepsilon}, which is further simplified by Proposition \ref{Aofsuspensionprop} in the case of a suspension. This will then be used in Section \ref{longknotsection} to study the case $(M,P)=(D^d,D^p)$. As we will shortly see in Section \ref{homotopyinvsection}, it will be significantly helpful to invert the prime $2$ in the analysis of these involutions. Let us now introduce the notation that will be relevant in this section.

\begin{notn}\label{involutionnotn}($i$)\hspace{1.5pt} For $M$ a compact (smooth) manifold and $\iota: P\subset M$ a compact submanifold, recall from Notation \ref{FEBorthfunctorsnotn} the definitions of the $C_2$-spectra $\Hsp(M)$, the \textbf{$h$-cobordism spectrum} of $M$, and $\CEsp(P,M)$, the \textbf{concordance embedding spectrum} of $\iota:P\xhookrightarrow{}M$. We refer to their involutions by $\tau_{WW}$, for Weiss--Williams.

($ii$)\hspace{1.5pt} Given a space $X$ and a spherical fibration $\xi$ over $X$ equipped with a section, Vogell defined in \cite[p. 300]{VogellInvolution} an involution $\tau_\xi$ on\footnote{Vogell defined $\tau_\xi$ on the $A$-theory space $A(X)$, but this involution can be upgraded to $\Asp(X)$ by specifying it on the Waldhausen category of retractive spaces over $X$ ``with $\xi$-duality'' and appealing to the definition of algebraic $K$-theory via the $S_\bullet$-construction.} $\Asp(X)$ by means of Spanier--Whitehead duality with respect to the Thom spectrum of $\xi$; we will write $\Asp(X;\xi)$ for the corresponding $C_2$-spectrum. When $\xi= \epsilon:= X\times S^0$ is the trivial $0$-dimensional sphere bundle, $\tau_\epsilon$ fits in a commutative square
\[
\begin{tikzcd}
    \Sigma^\infty_+ X\rar["\nu"]\dar[equal] & \Asp(X)\dar["\tau_\epsilon"]\rar[two heads] & \Whsp(X)\dar[dashed, "\tau_\epsilon"]\\
    \Sigma^\infty_+X\rar["\nu"] & \Asp(X)\rar[two heads] & \Whsp(X),
\end{tikzcd}
\]
and hence, on cofibres, induces the dashed vertical arrow---this is an involution on $\Whsp(X)$, which we shall also denote by $\tau_\epsilon$. We will refer to $\tau_\epsilon$ as the \textbf{canonical involution} of $K$-theory, and sometimes write $\Asp(X)$ and $\Whsp(X)$ for $\Asp(X;\epsilon)$ and $\Whsp(X;\epsilon)$. We will recall a construction of $\tau_\epsilon$ in terms of Spanier--Whitehead duality in Section \ref{canonicalinvsection}.
\end{notn}

\subsection{Homotopy involutions}\label{homotopyinvsection} A \textit{homotopy involution} $\tau$ on a space or infinite loop space or spectrum $X$ is a self-map $\tau: X\to X$ whose square $\tau^2$ is homotopic to the identity $\Id_X$. In this section we explain why, in the stable setting and once the prime $2$ is inverted, an involution carries the same amount of information as its underlying homotopy involution. This will be very useful when comparing the $C_2$-spectra $\Hsp(M)$ and $\Sigma^{-1}\Whsp(M;\epsilon)$ (see Proposition \ref{propTWWvsTepsilon}). Let us fix some notation first.

\begin{notn}\label{equivariantnotn}
    ($i$)\hspace{1.5pt}Let $\mathsf{C}$ denote any of $\mathsf{Top}_*$, $\Omega^\infty$-$\mathsf{Top}$ or $\mathsf{Sp}$, and let $X, X'\in \mathsf{C}$ be equipped with homotopy involutions $\tau$ and $\tau'$, respectively. A map $f:X\to X'$ will be said to be \textbf{homotopy $C_2$-equivariant}, or \textbf{$C_2$-equivariant up to homotopy}, if $f\tau\simeq \tau'f$. If $X$ and $X'$ can be connected by a zig-zag of homotopy $C_2$-equivariant weak equivalences, we will say that $X$ and $X'$ are \textbf{homotopy $C_2$-equivariantly equivalent} and write 
    $$X\approx X'.$$
    A \textbf{$C_2$-equivariant equivalence} will always mean a zig-zag of weak equivalences which are $C_2$-equivariant.

    \noindent ($ii$)\hspace{1.5pt} An \textbf{$H$-group} $(X,\mu)$ is a group-like $\mathbb{A}_3$-space (ie a homotopy associative $H$-space such that $\pi_0(X)$ is a group with respect to $\mu$). Given $H$-spaces $(X,\mu)$ and $(X',\mu')$, a based map $f: X\to X'$ will be said to be \textbf{monoidal up to homotopy}, or simply an \textbf{$H$-map}, if the following diagram is homotopy commutative:
    $$
\begin{tikzcd}
    X\times X\rar["f\times f"]\dar["\mu"] &X'\times X'\dar["\mu'"]\\
    X\rar["f"]&X'. 
\end{tikzcd}
    $$
    An \textbf{equivalence of $H$-groups} will mean a zig-zag of $H$-maps that are additionally weak equivalences. In practice, all $H$-groups we will consider are actually $\mathbb{E}_1$-groups (ie group-like $\mathbb{E}_1$-spaces), and all $H$-maps can be upgraded to $\mathbb{E}_1$-maps even though we will not need this. 

    \noindent($iii$)\hspace{1.5pt} Given a $C_2$-object $X$ in an appropriate category, the symbol $X_{hC_2}$ will stand for $\operatorname{hocolim}_{BC_2} X$.
\end{notn}

In the cases of interest to us and once the prime $2$ is inverted, taking homotopy $C_2$-orbits with respect to a homotopy involution turns out to make sense.

\begin{prop}\label{hC2construction}
    Let $X$ denote a spectrum or infinite loop space (aka connective spectrum), and let $\tau$ be a homotopy involution on $X$. Suppose that multiplication by two is invertible on $X$, ie $2: X\overset{\sim}\longrightarrow X$ is an equivalence, and define
    $$
E(X,\tau):=\hocolim\big(\begin{tikzcd}
  X\rar["\tfrac{1+\tau}{2}"]& X\rar["\tfrac{1+\tau}{2}"]&\dots
\end{tikzcd}\big),
    $$
    where $\frac{1+\tau}{2}$ really stands for the zig-zag $\begin{tikzcd}[column sep=15pt]
        X\rar["1+\tau"] &X&\lar["2"', "\sim"] X
    \end{tikzcd}$. Then if $\tau$ is an actual involution on $X$, there is a natural equivalence away from two
    $$
X_{hC_2}\simeq_{[\frac{1}{2}]} E(X,\tau).
    $$
\end{prop}

\begin{proof}
    Let us assume that $X$ is a $C_2$-spectrum (the other case is completely analogous). We also assume that $2$ is inverted. Observe now that as $t\cdot\tfrac{1+t}{2}=\tfrac{1+t}{2}$ in $\Z[C_2]$, then the following commutes up to homotopy
$$
\begin{tikzcd}
    X\ar[r, "\frac{1+\tau}{2}"]\ar[dr, "q"]&X\rar["\frac{1+\tau}{2}"]\dar["q"]&\dots\ar[dl, "q"']\\
    &X_{hC_2},&
\end{tikzcd}
$$
where $q: X\to X_{hC_2}=\mathrm{hocolim}_{BC_2}X$ is the map on colimits induced by the inclusion of categories $\{*\}\xhookrightarrow{}BC_2$. We thus obtain a map $\eta_{(X,\tau)}: E(X,\tau)\to X_{hC_2}$. The homotopy orbits spectral sequence for $X$, together with the assumption that $2$ is inverted, gives a natural isomorphism $\pi_*(X_{hC_2})\cong H_0(C_2;\pi_*(X))\cong \pi_*(X)_{C_2}$. Also by definition, we have that $\pi_*(E(X,\tau))\cong \mathrm{Im}(\tfrac{1+\tau}{2}: \pi_*(X)\to \pi_*(X))$. Under these identifications, the map $\pi_*(\eta_{(X,\tau)})$ is the natural isomorphism (away from $2$) sending an element $\beta=\tfrac{1+\tau}{2}\alpha\in \pi_*(E(X,\tau))$ to $[\beta]=[\alpha]\in \pi_*(X)_{C_2}$. So $\eta_{(X,\tau)}$ is the desired equivalence $X_{hC_2}\simeq_{[\frac{1}{2}]}E(X,\tau)$.
\end{proof}

\begin{cor}\label{corinvolutionsuptohomotopy}
    Let $X$ and $X'$ be $C_2$-spectra and let the prime $2$ be inverted. 
    \begin{itemize}
        \item[(i)] If there is a homotopy $C_2$-equivariant equivalence $X\approx X'$, then there is an equivalence of spectra
    $$
X_{hC_2}\simeq_{[\frac12]} X'_{hC_2}.
    $$

    \item[(ii)] If there is only a homotopy $C_2$-equivariant equivalence $\Omega^\infty X\approx \Omega^\infty X'$ of \textbf{$H$-spaces}, then we still have an equivalence of spaces
    $$
\Omega^\infty(X_{hC_2})\simeq_{[\frac{1}{2}]} \Omega^\infty(X'_{hC_2}).
    $$
    \end{itemize}
\end{cor}

\begin{proof}
Let us only deal with (ii) (as (i) is analogous and easier). Assume without loss of generality that the equivalence $\Omega^\infty X\approx \Omega^\infty X'$ of $H$-spaces is induced by a single homotopy $C_2$-equivariant $H$-map $g: \Omega^\infty X\xrightarrow{\sim} \Omega^\infty X'$. Then, the diagram \textit{of spaces}
$$
\begin{tikzcd}
    \Omega^\infty X\rar["\frac{1+\tau}{2}"]\dar["g", "\vsim"']&\Omega^\infty X\rar["\frac{1+\tau}{2}"]\dar["g", "\vsim"'] &\dots\\
    \Omega^\infty X'\rar["\frac{1+\tau'}{2}"]&\Omega^\infty X'\rar["\frac{1+\tau'}{2}"] & \dots
\end{tikzcd}
$$
commutes up to homotopy, where $\tau$ and $\tau'$ are the involutions of $X$ and $X'$, respectively. Since the fogetful map from infinite loop spaces to spaces preserves directed colimits, the diagram above (upon taking horizontal colimits) induces an equivalence of spaces $$E(\Omega^\infty X, \tau)\simeq E(\Omega^\infty X', \tau').$$ 

The claim now follows from Proposition \ref{hC2construction} and because the natural map $(\Omega^\infty X)_{hC_2}\to \Omega^\infty(X_{hC_2})$ is an equivalence away from $2$ (this is a consequence of the homotopy orbits spectral sequence).
\end{proof}

\begin{rem}
    The upshot of part (ii) of the previous corollary is that, given a $C_2$-spectrum $X$ that is local away from $2$, the homotopy type of $\Omega^\infty(X_{hC_2})$ \textit{as a space} is completely determined by the homotopy type of the space $\Omega^\infty X$ and the homotopy classes of the maps $\tau:\Omega^\infty X\to\Omega^\infty X$ and $+: \Omega^\infty X\times \Omega^\infty X\to \Omega^\infty X$.
\end{rem}

The following result, though unrelated to what has been discussed so far in this section, will be useful later on. Given an $\mathbb{E}_1$-space $X$, we will write $X^{\mathrm{op}}$ for $X$ equipped with the opposite $\mathbb{E}_1$-structure. An \textit{anti-involution} $\tau$ on an $\mathbb{E}_1$-space $X$ is an $\mathbb{E}_1$-map $\tau: X\to X^{\mathrm{op}}$ whose square equals the identity of $X$ (noting that $(X^{\mathrm{op}})^{\mathrm{op}}\cong X$). Up to equivalence, there is a standard way of delooping such an anti-involution.

\begin{lem}\label{antiinvolutionlem}
Let $X$ be an $\mathbb{E}_1$-space. There is a natural equivalence
$$
\iota: B(X^{\mathrm{op}})\simeq BX. 
$$
such that, for any anti-involution $\tau$ on $X$, the composition
$$
\begin{tikzcd}
\overline{B}\tau: BX\rar["B\tau"] &B(X^{\mathrm{op}})\overset{\iota}\simeq BX
\end{tikzcd}
$$
is an involution on $BX$.
\end{lem}
\begin{proof}
For each $k\geq 0$, the map $\mathbb{E}_1(k)\to \pi_0(\mathbb{E}_1(k))$ is an equivalence, and hence there is a natural zig-zag of equivalences of $\mathbb{E}_1$-algebras
$$
\begin{tikzcd}
B(\pi_0(\mathbb{E}_1),\mathbb{E}_1,X)&\lar["\sim"']B(\mathbb{E}_1,\mathbb{E}_1,X)\rar["\sim"] & X.
\end{tikzcd}
$$
But the $\mathbb{E}_1$-structure on the left hand side factors through the associative operad $\mathcal{A}ss:=\pi_0(\mathbb{E}_1)$, so for simplicity, we may assume that $X$ is strictly associative. The equivalence $\iota$ is then induced on the realisation of the nerve $N_\bullet X$ by the maps
    $$
X^{q}\times \Delta^q\longrightarrow X^{q}\times \Delta^q, \quad (x_1,\dots, x_q,r)\longmapsto (x_q,\dots, x_1, \Phi_q(r)),
    $$
    where $\Phi_q: \Delta^q\cong \Delta^q$ is the linear homeomorphism induced by reversing the order of the vertices. It is easy to check that the map $\overline{B}\tau$ indeed defines an involution on $BX$.\end{proof}


\subsection{From the \texorpdfstring{$h$}{h}-cobordism spectrum to spaces of \texorpdfstring{$h$}{h}-cobordisms}\label{hcobsection} All throughout this section, assume that $d=\dim M\geq 5$; this condition will not be a problem later, as all of the results in this section will be used only once our original manifold $M$ has been stabilised sufficiently many times.

We now recall Vogell's model for spaces of $h$-cobordisms (cf \cite[p. 296]{VogellInvolution}). A \textit{partition} of a manifold $M^d$ is a triple $(W,F,V)$, where $W$ is a codimension zero submanifold of $M\times[-1,1]$, $V$ is the closure of the complement of $W$ and $F^d:=W\cap V$. For technical reasons, we require $F$ to be standard near $\partial M\times [-1,1]$, and that it intersects it in $\partial M\times\{0\}$. Let $H(M)_\bullet$ denote the simplicial set a $p$-simplex of which is a (locally trivial smooth) family of partitions of $M$ parametrised by $\Delta^p$ such that $W$ is an $h$-cobordism from $M\times \{-1\}\times \Delta^p$ to $F$. Set $H(M):=|H(M)_\bullet|$ and write $H^s(M)\subset H(M)$ for the connected component containing the trivial partition $*=(M\times [-1,0], M\times \{0\}, M\times[0,1])$. There is a canonical involution $\iota_H$ given by turning upside down partitions. Namely
$$
\iota_H: H(M)\longrightarrow H(M), \quad \rho=(W,F,V)\longrightarrow \rho^*:=(V^*,F^*,W^*),
$$
where $W^*$, $F^*$ and $V^*$ are respectively the images of $W$, $F$ and $V$ under the reflection $r=\mathrm{Id}_M\times -1$. For the smooth case, we will also need a small variant of this $h$-cobordism space, denoted $H_{\mathsf{col}}(M)$, a point of which consists of a partition $\rho=(W,F,V)\in H(M)$ together with a bicollar of $F$ for $W$ and $V$ which is standard near $\partial M\times[-1,1]$. The forgetful map $H_{\mathsf{col}}(M)\to H(M)$ is a weak equivalence by the contractibility of the space of collars. In this section we construct a homotopy $C_2$-equivariant $(\phi(d)+1)$-connected map
\begin{equation}\label{totalalexhcobmap}
    \mathrm{alex}: H(M)\longrightarrow \Omega^\infty \Hsp(M)
\end{equation}
which generalises the map $B(\mathrm{alex})$ of Proposition \ref{alexpropBfunctor}. We first recall an important construction.

\subsubsection{The geometric Eilenberg swindle} An $h$-cobordism $W: M\hcob M'$ induces a unique (up to contractible choice) bounded diffeomorphism
\begin{equation}\label{ES}
    ES_W: M\times \R\cong M'\times \R
\end{equation}
as follows: choose embeddings $i_r: W\xhookrightarrow{} M\times I$ rel $M\times \{0\}$ and $i_\ell: W\xhookrightarrow{}M'\times I$ rel $M'\times \{1\}$. Both of these embeddings are unique up to isotopy, for given another such embedding $i'_r: W\xhookrightarrow{}M\times I$ rel $M\times \{0\}$, the embeddings $\Id_{-W}\cup_M i_r$ and $\Id_{-W}\cup_M i'_r$, where $-W:M'\hcob M$ is an inverse of $W$, are isotopic by the contractibility of the space of collars. But then, so are $\Id_{W\cup_{M'}-W}\cup_M i_r$ and $\Id_{W\cup_{M'}-W}\cup_M i'_r$, and these, in turn, are isotopic to $i_r$ and $i'_r$ (rel $M$) by choosing an identification $W\cup_{M'}-W\cong M\times I$. Similarly for $i_\ell$.

Now write $V_r:= \overline{M\times I - i_r(W)}$ and $V_\ell:=\overline{M'\times I-i_\ell(W)}$; both of these manifolds are $h$-cobordisms $M'\hcob M$, and in fact they are diffeomorphic relative to \emph{both} ends since
$$
V_\ell\cong V_\ell\cup_{M} M\times I\cong V_\ell\cup_{M} W\cup_{M'}V_r\cong M'\times I\cup_{M'} V_r\cong V_r\quad \mathrm{rel}\ M'\sqcup M.
$$
Then the \textit{Eilenberg swindle} diffeomorphism $ES_W$ is given by the composition
\[
ES_W:M\times \R=\dots \cup_{M'}V_r\cup_M W\cup_{M'} V_r\cup_{M}\ensuremath{\dots}\cong\dots \cup_{M'}V_\ell\cup_M W\cup_{M'} V_\ell\cup_{M}\dots =M'\times \R.
\]
Clearly, by construction, $ES_W$ is bounded by $1$ and unique up to contractible choice.

\subsubsection{The map (\ref{totalalexhcobmap})}

Fix an embedding $M^d\subset \R^N\subset \R^\infty$ and recall that $B\diff_\partial(M)$ admits a model as the moduli space of manifolds embedded in $\R^\infty$ which are abstractly diffeomorphic to $M$ relative to the boundary $\partial M$. Similarly $B\diff_\partial^b(M\times \R)$ is the moduli space of manifolds embedded in $\R^\infty\times \R$ which are abstractly diffeomorphic to $M\times \R\subset \R^\infty\times \R$ \textit{boundedly} with respect to the $\R$-direction and relative to the boundary $\partial M\times \R$ (this is proved in Appendix \ref{bdiffbmodelsection}). For the remaining of this section, we will denote by $\underline\R$ the bounded direction, ie, the last coordinate in $\R^\infty\times \R=:\R^\infty\times\underline\R$. There is a natural map $-\times\underline\R: B\diff_\partial(M)\xhookrightarrow{} B\diff^b_\partial(M\times\underline{\R})$ given by sending a manifold $N\subset \R^\infty$ to $N\times\underline \R\subset \R^\infty\times\underline \R$. In fact, in light of (\ref{ES}), this map extends to
$$
-\times \underline{\R}: \coprod_{[M']} B\diff_\partial(M')\longrightarrow B\diff^b_\partial(M\times \underline{\R}), \quad N\longmapsto N\times \underline{\R},
$$
where the coproduct in the domain runs over all diffeomorphism classes of manifolds $M'$ with boundary $\partial M$ that are $h$-cobordant to $M$ rel $\partial M$. The map $-\times\underline{\R}$ is the value of the morphism $0\to \underline{\R}$ in $\mathcal{J}_0$ under
$$
\widehat{B}: \mathcal{J}_0\longrightarrow \mathsf{Top}_*, \quad \widehat{B}(V):=\left\{\begin{array}{cc}
    \coprod_{[M']} B\diff_\partial(M'), & V=0, \\[5pt]
    B(V)=B\diff^b_\partial(M\times V), & \text{otherwise}.
\end{array}
\right.
$$
Clearly $\widehat{B}$ is an orthogonal functor and we will write $\widehat{\Hsp}(M)$ for its first derivative, which is canonically equivalent to $\Hsp(M)$. The Alexander trick-like map (\ref{totalalexhcobmap}) will factor through
$
\widehat{\Hsp}(M)_0\xhookrightarrow{}\Omega^\infty \widehat{\Hsp}(M)\simeq \Omega^\infty \Hsp(M)
$.

Suppose we are given some partition $\rho=(W,F,V)\in H(M)$ of $M\times \underline{[-1,1]}\subset \R^N\times \underline{\R}$. Then $W$ is an $h$-cobordism from $M$ to $F$ rel boundary, and so the manifold $F\subset \R^N\times \R\subset \R^\infty$ gives rise to a point in $B\diff_\partial(F)\subset \widehat{B}(0)$; more precisely, the image of the embedding
$$
i_\rho: F\subset M\times I\subset \R^N\times \underline \R\cong \R^{N+1}\subset \R^\infty,
$$
is a point in $B\diff_\partial(F)$, where the isomorphism $\R^{N}\times\underline\R\cong \R^{N+1}$ identifies $\underline \R$ with the last coordinate in $\R^{N+1}$. We now construct a point in $B\diff_\partial^b(M\times \underline{\R})$ by \emph{extending $W$ towards infinity}. Consider the embedding of $F\times [0,1]$ into $\R^N\times \R\times \underline \R$ given by
$$
R: F\times [0,1]\xhookrightarrow{} \R^N\times \R\times \underline\R, \quad (x,t)\longmapsto \underline{e}+(\Id_{\R^N}\times Q_{-\pi t/2})(x-\underline{e}),
$$
where $Q_\theta: \R\times \underline\R\cong \R\times \underline\R$ is the rotation matrix $\begin{psmallmatrix}
    \cos\theta & -\sin\theta\\ \sin\theta & \cos\theta
\end{psmallmatrix}$, and $\underline e$ denotes the unit length vector in $\underline \R$. Write $r:=R\mid_{F\times \{1\}}: F\hookrightarrow \R^{N+1}\times \underline{\{1\}}$, and consider 
$$
S: F\times [1,+\infty)\xhookrightarrow{}\R^{N+1}\times \underline \R, \quad (x,t)\longmapsto \left\{
\begin{array}{cc}
    r(x)+(1-t)\cdot e_{N+1}+(t-1)\cdot\underline{e}, & t\in[1,2], \\
    r(x)-e_{N+1}+(t-1)\cdot\underline{e}, & t\geq 2.
\end{array}
\right.
$$
Finally consider the region $D\subset\R\times \underline \R$ given by tuples $(u,v)$ with
$$
u\geq 0, \quad u\leq 2-v, \quad \text{and if $0\leq u\leq 1$, then $u\leq (1-(v-1)^2))^{1/2}$.}
$$
Then we define a topological manifold $\widehat{a}(\rho)\subset \R^{N+1}\times\underline\R$, depicted in Figure \ref{hcobmapfigure}, by
$$
\widehat{a}(\rho):=M\times\{0\}\times(-\infty\times -1]\cup W\cup R(F\times [0,1])\cup S(F\times[1,+\infty))\cup \partial M\times D.
$$

\begin{figure}[h]
    \centering
    \includegraphics[scale = 0.54]{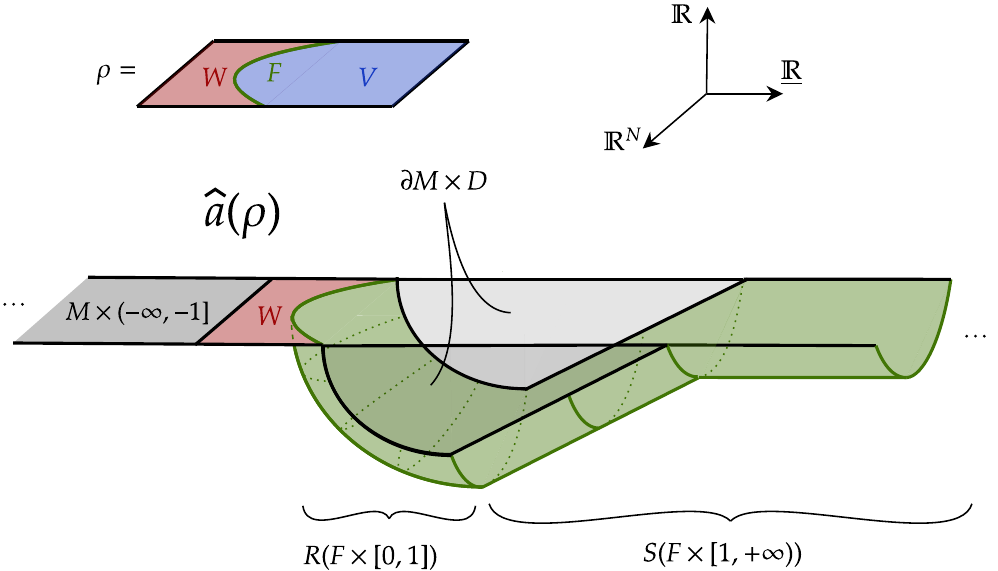}
    \caption{Depiction of the topological manifold $\widehat{a}(\rho)$ for $\rho\in H^s(M)$.}
    \label{hcobmapfigure}
\end{figure}

Now if $\rho$ is a collared partition, ie a point in $H_{\mathsf{col}}^s(M)$, one can use the collar of $F$ to smooth out the corners of the topological manifold $\widehat{a}(\rho)$, and thus obtain a smooth manifold $\widetilde{a}(\rho)\subset \R^{N+1}\times \underline\R$ with the same boundary as $M\times \underline\R$, and which is boundedly diffeomorphic to $M\times \R$ relative to the boundary by a one-sided Eilenberg swindle argument. This construction can be done simplex-wise in $H_{\mathsf{col}}(M)_\bullet\simeq H(M)_\bullet$, and so up to weak equivalence gives rise to an Alexander trick-like map
\begin{align}
\mathrm{alex}: H(M)&\overset{\sim}\longrightarrow \widehat{\Hsp}(M)_0:=\hofib\left(\coprod_{[M']} B\diff_\partial(M')\to B\diff_\partial^b(M\times \underline{\R})\right),\label{alexhcob}\\
\rho=(W,F,V)&\longmapsto \left(i_\rho(F),\ \gamma_W: [-\infty,\infty]\ni t\mapsto \left\{\begin{array}{cc}
    M\times \underline\R, & t=-\infty, \\
    \widetilde{a}(\rho)-t\cdot\underline{e}, & -\infty<t<+\infty,\\
    i_\rho(F)\times\underline\R, & t=+\infty.
\end{array}\right.\right),\nonumber
\end{align}
where we can regard the path $\gamma_W$ as a $1$-simplex in $B\diff_\partial^b(M\times \underline{\R})_\bullet$ from $M\times \underline{\R}$ to $F\times \underline{\R}$. Then (\ref{totalalexhcobmap}) is
$$
\begin{tikzcd}
    \mathrm{alex}:H(M)\rar["(\ref{alexhcob})", "\sim"'] & \widehat{\Hsp}(M)_0\rar[hook]&\Omega^\infty\widehat{\Hsp}(M)\simeq \Omega^\infty\Hsp(M)
\end{tikzcd}
$$

\begin{prop}\label{alexequivlemma}
    The map (\ref{alexhcob}) is indeed an equivalence. Therefore, (\ref{totalalexhcobmap}) is $(\phi(d)+1)$-connected.
\end{prop}

\begin{proof}
Noting the equivalence $H^s(M)\simeq BC(M)$ (cf \cite[Proposition 2.1]{VogellInvolution}), the map (\ref{alexhcob}) is, up to homotopy, a (non-connected) delooping of the Alexander trick-like equivalence $C(M)\simeq \Omega(\diff_\partial^b(M\times \R)/\diff(M))$ of \cite[Proposition 1.10]{WWI}, and therefore it is an equivalence on basepoint components. 

Given a diffeomorphism class $[M']$ of manifolds $h$-cobordant to $M$ (rel boundary), denote by $H(M,M')$ the collection of path components in $H(M)$ consisting of (collared) partitions $\rho=(W,F,V)$ with $F\in [M']$. A choice of basepoint $\rho_0=(W_0,F_0,V_0)\in H(M,M')$, a bicollar $c_0: F_0\times [-\epsilon,\epsilon]\xhookrightarrow{} M\times [-1,1]$ and a diffeomorphism $\phi_0:M'\cong F_0$ gives rise to an equivalence $H(M',M')\overset{\sim}\to H(M,M')$ which sends a partition $\rho'=(W',F',V')$ of $M'\times [-1,1]$ to the partition of $M\times [-1,1]$ whose $F$-part is the image of $F'$ under
$$
M'\times [-1,1]\xrightarrow{\phi_0\times \cdot\epsilon} F_0\times [-\epsilon,\epsilon]\xhookrightarrow{c_0} M\times [-1,1].
$$
By the $s$-cobordism theorem, a homotopy inverse $H(M,M')\xrightarrow{\sim} H(M',M')$ is given by the same kind of map for a choice of basepoint $\rho'_0=(W'_0,F'_0,V'_0)\in H(M',M')$ such that $(\phi_0\times \Id_{[-1,1]})(W'_0)$ is an $h$-cobordism starting at $F_0$ with the same torsion as $V_0$ (the inverse of $W_0$). 

Observe also that the choices $(\rho_0,c_0)$ and $\phi_0$ above give a preferred path in $B\diff^b_\partial(M\times\underline{\R})=B\diff^b_\partial(M'\times\underline{\R})$ from $M\times \underline{\R}$ to $M'\times \underline{\R}$; namely, it is the composition of $\gamma_{W_0}$, as defined in (\ref{alexhcob}), with the mapping cylinder of $\phi_0\times \Id_{\underline{\R}}$. These preferred paths give rise to the ``change of basepoint'' equivalences in the right column of the homotopy commutative diagram
$$
\begin{tikzcd}[row sep = 10pt]
    H(M)\ar[rr, "(\ref{alexhcob})"]\dar[equal] &&\hofib_{M\times \underline{\R}}\left(\coprod_{[M']} B\diff_\partial(M')\to B\diff_\partial^b(M\times \underline{\R})\right)\dar[phantom, "\vsimeq"]\\
    \coprod_{[M']}H(M,M')&&\coprod_{[M']}\hofib_{M'\times \underline{\R}}\left(B\diff_\partial(M')\to B\diff_\partial^b(M'\times \underline{\R})\right)\\
    \coprod_{[M']} H(M',M')\uar[phantom, "\vsimeq"]\ar[rr, "\coprod_{[M']}\mathrm{alex}"] &&\coprod_{[M']} \Hsp(M')_0,\uar[equal]
\end{tikzcd}
$$
where $\mathrm{alex}: H(M',M')\to \Hsp(M')_0\subset \widehat{\Hsp}(M')_0$ is the restriction to $H(M',M')\subset H(M')$ of the map (\ref{alexhcob}) for $M=M'$. Moreover if $d\geq 5$, this map is an isomorphism in $\pi_0$ (in fact it is an equivalence by the argument above); indeed, the inverse
$$
\pi_0\left(\diff^b_\partial(M'\times \underline{\R})/\diff_\partial(M')\right)\longrightarrow \pi_0(H(M',M'))\subset \Wh(M)
$$
sends the coset $[\phi]$ of some bounded diffeomorphism $\phi\in \diff^b_\partial(M'\times \underline{\R})$ (say bounded by $1/2$ for simplicity) to a partition of $M'\times[-1,1]$ whose $W$-part is the $h$-cobordism obtained as the region in $M'\times \underline{\R}$ between $M'\times \{0\}$ and $\phi(M'\times \{1\})$ (cf \cite[Corollary 5.4]{WWI}). It follows that the lower horizontal map, and hence (\ref{alexhcob}), is an isomorphism in $\pi_0$. This proves the first claim.

The second claim is a consequence of the fact that $\widehat{\Hsp}(M)_0\xhookrightarrow{}\Omega^\infty\widehat{\Hsp}(M)$ is $(\phi(d)+1)$-connected: indeed this map is $\phi(d)$-connected upon looping once by \cite[Lemma 1.12]{WWI}, and is an isomorphism in $\pi_0$ by the analysis above and \cite[Proposition 1.8 \& Corollary 5.3]{WWI}.
\end{proof}

\begin{prop}\label{alexHcoblem}
    The map (\ref{alexhcob}) is $C_2$-equivariant up to homotopy. Therefore, so is (\ref{totalalexhcobmap}). 
\end{prop}
\begin{proof}
We will give an argument only in the topological setting; in the smooth setting one works with $H_{\mathsf{col}}(M)$ instead to smooth out corners, and uses smooth approximations of the continuous functions that will appear in proof below. We will however state the argument in the smooth setting to simplify notation. We will also assume at any point in the argument where it is necessary that a partition $(W,F,V)$ is equipped with some collar of $F$ in $W$ and $V$. We adopt the convention that $\pm\infty + r\equiv\pm \infty$ for any real number $r\in \R$.

The Weiss--Williams involution on $\widehat{\Hsp}(M)_0$ is induced by the identity on $\coprod_{[M']}B\diff_\partial(M')$ and the involution $U\mapsto U^*=(\mathrm{Id}_{\R^\infty}\times (-1)_{\underline\R})(U)$ on $B\diff^b_\partial(M\times \underline\R)$. Then for $\rho=(W,F,V)\in H^s(M)$,
$$
\tau_{WW}\circ\mathrm{alex}(W,F,V)=(i_\rho(F),\gamma_W^*), \qquad \mathrm{alex}\circ\iota_H(W,F,V)=(i_{\rho^*}(F^*),\gamma_{V^*}),
$$
where $\gamma_W^*(t):=(\gamma_W(t))^*$. We have depicted the paths $\gamma_{W}^*$ and $\gamma_{V^*}$ in Figure \ref{pathshcob}. We need to find a path $\eta:[-1,1]\to B\diff_\partial(M)$ from $i_\rho^*(F^*)$ to $i_\rho(F)$ and a homotopy $\{H_s(-)\}_{-1\leq s\leq 1}$ from $\gamma_{V^*}(-)$ to $\gamma_W^*(-)$ such that $H_s(-\infty)=M\times \underline\R$ and $H_s(+\infty)=\eta(s)\times \underline\R$ for all $s\in[-1,1]$.

\begin{figure}[h]
    \centering
    \includegraphics[scale = 0.65]{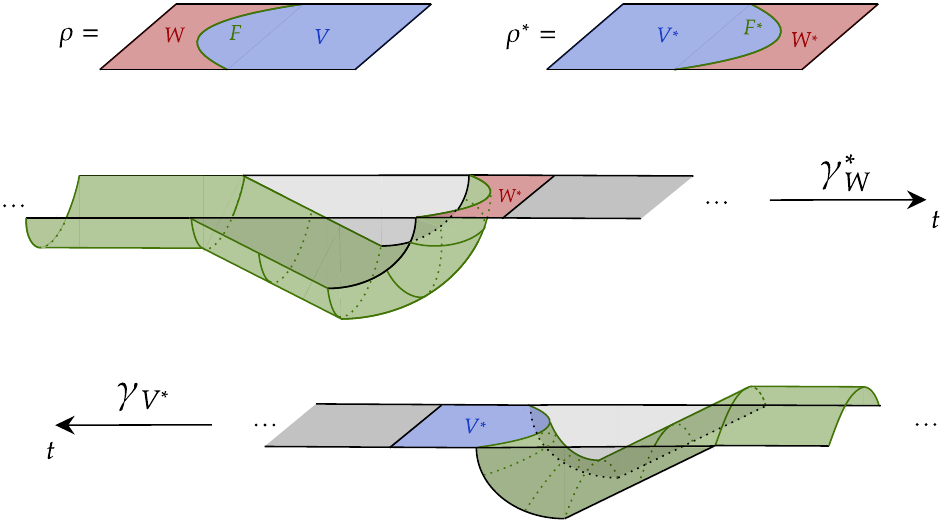}
    \caption{Paths $\gamma_{W}^*$ and $\gamma_{V^*}$ in $B\diff^{b}_{\partial}(M\times \underline{\mathbb{R}})$. The arrow indicates the direction of the path as time increases.}
    \label{pathshcob}
\end{figure}

For $\eta$, we use the last two coordinates in $\R^{N+2}$ to do a half rotation of that plane. More explictly, $\eta(s):=(\Id_{\R^N}\times Q_{\pi\cdot (s+1)/2})(i_\rho(F))$ where $Q_{\theta}: \R^2\cong \R^2$ is as before. 

View $\R^\infty$ as $\R^\infty\times\underline{\{0\}}\subset \R^\infty\times \underline\R$ and write $N:=\bigcup_{s\in[-1,1]} \eta(s)+s\cdot\underline{e}$. For $X\subset \R^\infty\times \underline \R$, write $X\mid_{[a,b]}$ for $X\cap (\R^\infty\times\underline{[a,b]})$. Then consider the compact manifold 
$$
U_\rho:= \left(\widehat{a}(\rho^*)\mid_{[-1,2]}-3\cdot\underline e\right)\cup N\cup \left((\widehat{a}(\rho)\mid_{[-1,2]})^*+3\cdot\underline{e}\right)
$$
depicted in Figure \ref{toboganxd}. Using the contractibility of $\emb_{\partial}(F^*\times \underline{[-3,3]},\R^\infty\times \underline{[-3,3]})$, we obtain a path from $\overline{U_\rho\setminus (V^*\cup W^*)}$ (the green part in Figure \ref{toboganxd}) to a scaled (in the $\underline{\R}$-direction) version of the bicollar of $F^*$ in $W^*$ and $V^*$. Rescaling this bicollar back to normal whilst dragging $V^*$ and $W^*$ in the process, we obtain a path $\psi$ from $U_\rho$ to $M\times \underline{[-4,4]}=M\times\underline{[-4,-1]}\cup V^*\cup W^*\cup M\times\underline{[1,4]}$ in the moduli space of manifolds inside $\R^{N+2}\times \underline{[-4,4]}$ which are diffeomorphic to $M\times\underline{[-4,4]}$ relative to its boundary $M\times\underline{\{-4\}}\cup \partial M\times\underline{[-4,4]}\cup M\times\underline{\{4\}}$.

\begin{figure}[h]
    \centering
    \includegraphics[scale = 0.7]{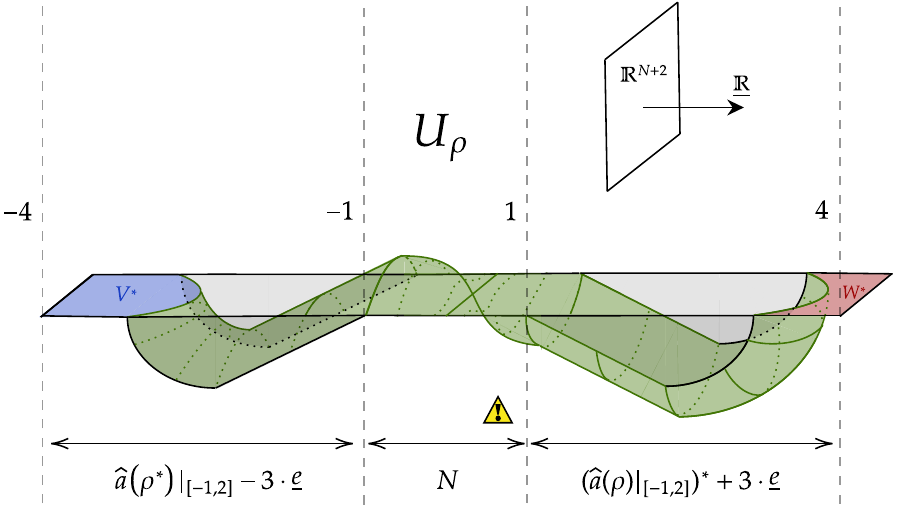}
    \caption{Depiction of the manifold $U_\rho$. Proceed with caution: the part of the picture corresponding to $N$ takes place in an extra dimension that we are unable to depict accurately.}
    \label{toboganxd}
\end{figure}

We now describe the homotopy $\{H_s(-)\}_{-1\leq s\leq 1}$. Fix some homeomorphism $l: [-1,1]\cong [-\infty,+\infty]$, and assume that the path $\psi$ from $M\times \underline{[-4,4]}$ to $U_\rho$ just described is parametrised by $[-\infty,+\infty]$. Then $H_s(-)$ is the concatenation of two paths $H_s^{(1)}(-)$ and $H_s^{(2)}(-)$ in $B\diff^b_\partial(M\times \underline\R)$: the path $H_s^{(1)}(-)$ performs $\psi(-)$ on $M\times\underline{[-4,4]}+l(s)\cdot\underline{e}\subset M\times \underline{\R}$ (if $s=\pm 1$, $H^{(1)}_s(-)$ is constant on $M\times\underline \R$). The path $H_s^{(2)}(-)$ starts at $H^{(1)}_s(+\infty)$, and sends $H^{(1)}_s(+\infty)\mid_{(-\infty, l(s)+s]}$ and $H^{(1)}_s(+\infty)\mid_{[l(s)+s, +\infty)}$ towards $\R^{N+2}\times\underline{\pm\infty}$, respectively, extending by $H^{(1)}_s(+\infty)\mid_{l(s)+s}$ times an interval of diverging length. The resulting paths $H_{\pm1}(-)=H^{(1)}_{\pm 1}(-)\cdot H^{(2)}_{\pm 1}(-)$ are reparametrisations of $\gamma_{V^*}$ and $\gamma_{W^*}$ (the reparametrisations only depend on our choice of homeomorphism $l:[-1,1]\cong[-\infty,+\infty]$ and the parametrisation of the path $\psi$). Thus $\eta$ and $H$ give rise to the required homotopy $\tau_{WW}\circ\mathrm{alex}\simeq \mathrm{alex}\circ\iota_H$.
\end{proof}

\begin{cor}\label{milnorinvcor}
    The Weiss-Williams involution on $\pi_0^s(\Hsp(M))\cong\pi_0(H(M))\cong\Wh(\pi_1M)$ corresponds to the rule $\kappa\mapsto (-1)^{d-1}\overline{\kappa}$, where $\overline{(-)}$ is Milnor's involution \cite{MilnorWhiteheadTorsion} on $\Wh(\pi_1(M))$ (see Warning \ref{Milnorinvolutionwarning}).
\end{cor}
\begin{proof}
    The isomorphism $\pi_0(H(M))\cong \Wh(\pi_1M)$ sends the class $[\rho]$ of a partition $\rho=(W,F,V)$ to the Whitehead torsion $\tau(W,M)$ of $W$ with respect to $M$. The claim follows from Proposition \ref{alexHcoblem}, the \textit{duality formula} of \cite[Section 10]{MilnorWhiteheadTorsion} and the fact that, upon identifying $\pi_1 M\cong \pi_1 W\cong \pi_1 F$, we have $\tau(V,F)=-\tau(W,M)$ (cf \cite[Lemma 7.8]{MilnorWhiteheadTorsion}).
\end{proof}

Recall that $H^s(M)\simeq BC(M)$. Now if $P\subset M$ is a codimension zero embedding and $p\leq d-3$ (in the notation of Theorem \ref{EmbWWIThm}), then $\cemb(P,M)\simeq \hofib(H^s(\overline{M-P})\to H^s(M))$ by the isotopy extension sequence (\ref{ConcIET}), and therefore $\cemb(P,M)$ inherits an involution $\iota_H$ up to weak equivalence. The map (\ref{alexhcob}) is functorial with respect to codimension zero embeddings, so the following diagram is commutative:
\begin{equation}\label{functorialityalex}
\begin{tikzcd}
    H^s(\overline{M-P})\dar\rar["\mathrm{alex}"] &\Hsp(\overline{M-P})_0\dar\rar[hook] & \Omega^\infty(\Hsp(\overline{M-P}))\dar\\
    H^s(M)\rar["\mathrm{alex}"] &\Hsp(M)_0\rar[hook] & \Omega^\infty(\Hsp(M)).
\end{tikzcd}
\end{equation}

\begin{cor}\label{WWhcobcor}
The vertical homotopy fibre of the horizontal compositions in (\ref{functorialityalex}) gives a map
$$
\mathrm{alex}: \cemb(P,M)\longrightarrow \Omega^\infty(\CEsp(P,M))
$$
which is $\phi_{\cemb}(d,p)$-connected and $C_2$-equivariant up to homotopy. 
\end{cor}
\begin{proof}
     The connectivity of this map is the content of Proposition \ref{alexmapsprop}. It is homotopy $C_2$-equivariant since both $\mathrm{alex}: H(M)\to \Omega^\infty(\Hsp(M))$ and $\mathrm{alex}: H(\overline{M-P})\to \Omega^\infty(\Hsp(\overline{M-P}))$ are by Proposition \ref{alexHcoblem}.
\end{proof}

\begin{warn}\label{vogellwarning}
    There is a canonical involution $\iota_C$ in the concordance space $C(M)$ given by turning upside down a concordance and precomposing by the inverse of the top diffeomorphism (see eg \cite[p. 296]{VogellInvolution}). The restriction map $C(M)\to \cemb(P,M)$ is \emph{not} $C_2$-equivariant with respect to $\iota_C$ and $\iota_H$---rather, it is anti-equivariant. This may seem to contradict \cite[Proposition 2.2]{VogellInvolution}, but what Vogell really proves there is that there is a homotopy $C_2$-equivariant equivalence $C(M)\approx \Omega^\sigma H(M):=\mathrm{Map}_*(S^\sigma, H(M))$, where we recall $S^\sigma$ stands for the representation sphere of the $1$-dimensional sign representation $\sigma$. This is due to an extra flip in the loop component that he introduces at the end of the proof of the proposition. 
\end{warn}

\subsection{From \texorpdfstring{$h$}{h}-cobordism spaces back to \texorpdfstring{$A$}{A}-theory}\label{Vogellinvsection} 

Given a spherical fibration $\xi$ over $M$ equipped with a section, fibrewise smashing a retractive space over $M$ with $\xi$ gives rise to a functor $-\cdot \xi: \Asp(M)\to \Asp(M)$ which, by \cite[Proposition 2.5]{VogellInvolution}, makes the following diagram homotopy commutative:
\begin{equation}\label{Vogellcommutativity}
\begin{tikzcd}
\Asp(M)\rar["\tau_{\epsilon}"]\ar[dr, "\tau_\xi"'] & \Asp(M)\dar["-\cdot \xi"]\\
&\Asp(M),
\end{tikzcd}
\end{equation}
where $\epsilon:=\epsilon^0=M\times S^0$ is the trivial $0$-dimensional sphere bundle over $M$. When $\xi=\epsilon^d=M\times S^d$ is the trivial $d$-spherical fibration, the functor $-\cdot \xi$ corresponds to $\Sigma^d_M(-): \Asp(M)\to \Asp(M)$, the $d$-fold \textit{fibrewise suspension over $M$} (cf \cite[p. 281]{VogellInvolution}). By the \textit{additivity theorem} of \cite[Proposition 1.6.2]{WaldhausenAtheory} applied to the Waldhausen category of retractive spaces over $M$, it follows that $\Sigma^d_M$ acts (up to homotopy) as $(-1)^d$ on $\Asp(M)$, and thus by (\ref{Vogellcommutativity})
\begin{equation}\label{Vogellinvtrivbundle}
\xi=M\times S^d\implies \Asp(M;\xi)\approx S^{d\cdot(\sigma-1)}\wedge \Asp(M;\epsilon).
\end{equation}

Vogell also introduced \cite[p. 299]{VogellInvolution} a model for the homotopy fibre sequence 
\begin{equation}\label{WJRsequence}
\begin{tikzcd}
    \mathcal{H}(M)\rar & Q_+ M\rar & A(M):=\Omega^\infty \Asp(M)
\end{tikzcd}
\end{equation}
of the parametrised $h$-cobordism theorem of Waldhausen--Jahren--Rognes \cite{WJR} (see Remark \ref{VogellWJRrem}), and equipped each of the terms in the sequence with compatible homotopy involutions. The one on $\mathcal{H}(M)$ is compatible with $\iota_H$ on $H(M)\subset \mathcal{H}(M)$.

    \begin{rem}\label{VogellWJRrem}
Vogell's work precedes (by more than 25 years) that of Waldhausen--Jahren--Rognes, so let us explain how both fit together. In \cite[p.~299]{VogellInvolution}, Vogell presents a commutative square with compatible homotopy involutions in each of the terms, and this square is equivalent to the one considered by Waldhausen \cite[p.~8]{WaldManifoldApproach} in his ``manifold approach'' paper. This latter square is, up to equivalence, of the form
\[
\begin{tikzcd}
    \mathcal{H}(M) \arrow{r} \arrow{d} & Q \arrow{d} \\
    * \arrow{r} & A,
\end{tikzcd}
\]
for some spaces $Q$ and $A$, and the goal of that paper was to argue that (i) the square becomes homotopy cartesian upon plus-constructing the vertical right map, and (ii) that $A^+ \simeq A(M)$; these are, respectively, Propositions 5.5 and 5.4 in \cite{WaldManifoldApproach}. While (ii) was fully proved there, only an outline of the argument for (i) was provided, with forward references to a preliminary version of \cite{WJR}.

As stated at the very end of page 22 in \cite{WJR}, the square considered by Waldhausen (after plus-constructing the vertical right arrow) is equivalent to the homotopy cartesian square of \cite[Proposition~1.4.8]{WJR}, which in turn is equivalent to one giving rise to the fibre sequence~(\ref{WJRsequence}).

All of the above takes place in the $PL$-setting, but as explained in \cite[Pages~15--16]{WJR}, these arguments also deal with the remaining categories $\mathrm{Top}$ and $\mathrm{Diff}$.
\end{rem}

Going back to Vogell's work, his homotopy involution\footnote{Vogell refers to $\mathscr{T}$ as a \emph{weak} involution in the sense that it is a homotopy involution when restricted to ``any compactum'' (cf \cite[Lemma 2.4]{VogellInvolution}) or, in better words, to each stage of the colimit in \cite[p. 299]{VogellInvolution} modelling $A(M)$.} $\mathscr T$ on $A(M)$ is further showed in \cite[Corollary 2.10]{VogellInvolution} to agree up to equivalence with the involution $\tau_\xi$ when $\xi$ is the $d$-spherical fibration associated to the once stabilised tangent bundle $TM\oplus \epsilon^1$ of $M^d$.

The upshot of this discussion then is that when $M^d$ is stably parallelisable, the Weiss--Williams involution is compatible with $(-1)^d\tau_\epsilon$ in the following sense.

\begin{thm}\label{absolutepropTWWvsTepsilon}
If $M$ is stably parallelisable, then there is an equivalence away from two
$$
\Omega^\infty\big(\Hsp(M)_{hC_2}\big)\simeq_{[\frac12]}\Omega^\infty\big(\big(S^{d\cdot(\sigma-1)-1}\wedge\Whsp(M;\epsilon)\big)_{hC_2}\big).
$$
\end{thm}

\begin{rem}
Though it probably is, we do not claim the equivalence above is one of infinite loop spaces. 
\end{rem}

 We will need the a few preliminary results for the proof of Theorem \ref{absolutepropTWWvsTepsilon}.

\begin{lem}\label{MIKlem}
For each $k\geq 0$, there is a natural $C_2$-equivariant equivalence of spectra
\begin{equation}\label{ThetaMIkeq}
e_k:\Hsp(M\times I^k)\simeq S^{k\cdot(\sigma-1)}\wedge \Hsp(M).
\end{equation}
such that the following square is homotopy commutative:
$$
\begin{tikzcd}
    \Hsp(M\times I^{k})\dar[dash, "\vsim"', "e_k"]\rar[dash, "\sim", "e_1"'] &S^{\sigma-1}\wedge \Hsp(M\times I^{k-1})\dar[dash, "\vsim"', "{S}^{\sigma-1}\wedge e_{k-1}"]\\
    S^{k\cdot(\sigma-1)}\wedge \Hsp(M)\rar[equal] &S^{k\cdot(\sigma-1)}\wedge \Hsp(M).
\end{tikzcd}
$$
\end{lem}
\begin{proof}
    Set $e_0=\Id_{\Hsp(M)}$. By inductively defining $e_k$ to fit in the commutative square above, we may assume that $k=1$. Recall $\R^{a,b}:=\R^{a}\oplus b\cdot \sigma$, and for any orthogonal functor $F(-)$ let $C_2=O(1)$ act on $F(\R^{a,b})$ by the induced action. Finally let $B(-):=B\diff^b_\partial(M\times(-))$ and $BI(-):=B\diff^b_\partial(M\times I\times (-))$. The Alexander trick-like map of \cite[Proposition 1.5]{WWI} is a $C_2$-equivariant map
    $$
\mathrm{alex}:\diff_\partial^b(M\times I\times \R^{a,b})\overset{\sim}\longrightarrow \Omega\diff_\partial^b(M\times \R^{a+1,b})
    $$
which, upon delooping, gives rise to a $C_2$-equivariant equivalence on basepoint components $B(\mathrm{alex}): BI(\R^{a,b})\simeq_0\Omega B(\R^{a+1,b})$. Writing $\Xi$ for the $C_2$-spectrum whose $n$-th space is $B^{(1)}(\R^{1,n+1})$ and with stabilisation maps $s_{0,1}: S^1\wedge B^{(1)}(\R^{1,n})\to B^{(1)}(\R^{1,n+1})$, we obtain a $C_2$-equivariant equivalence of spectra
\begin{equation}\label{Chimap}
B(\mathrm{alex}):\Hsp(M\times I):=\Theta(BI)^{(1)}\overset\sim\longrightarrow \Xi:=\{B^{(1)}(\R^{1,n+1})\}_{n\geq 0}.
\end{equation}
But now the stabilisation map $s_{1,0}: S^{\sigma}\wedge B^{(1)}(\R^{0,n})\to B^{(1)}(\R^{1,n})=\Xi_{n-1}$ induces another $C_2$-equivariant equivalence of spectra
$$
S^{\sigma-1}\wedge \Hsp(M)\overset{\sim}\longrightarrow \Xi.
$$
Composing these two equivalences gives the one in the statement.
\end{proof}

Vogell introduced in \cite[p. 298]{VogellInvolution} the \textit{lower} and \textit{upper stabilisation maps} $\Sigma_\ell, \Sigma_u: H(M)\to H(M\times I)$. Roughly, the former sends a partition $\rho=(W,F,V)$ to $(U(W),W\cup_{F}W, \overline{M\times I\times I\setminus U(W)})$, where $U(W)$ is obtained from $W$ by bending $W\times I$ into a $U$-shape, whilst $\Sigma_u$ does the same to $V$ instead of $W$ (see Figure \ref{lowerstabmap} for a pictorial representation of $\Sigma_\ell$). We will only be interested in the lower stabilisation $\Sigma_\ell$, which we will denote by $\Sigma$ for simplicity. Here's how it interacts with the $h$-cobordism involution $\iota_H$.

\begin{lem}\label{suspensionequivariancelem}
    Let $+_I$ stand for the ``stacking in the $I$-direction'' $\mathbb{E}_1$-algebra structure\footnote{The technical assumption we imposed on a partition $\rho=(W,F,V)\in H(M\times I)$ so that the intersection of $F$ with $\partial(M\times I)\times [-1,1]$ is standard and happens exactly at $\partial(M\times I)\times \{0\}$ makes $+_I$ well-defined.} in $H(M\times I)$. Then if $J$ denotes another copy of $I$:
   $$
(a)\hspace{3pt} \iota_H\Sigma+_I\Sigma\iota_H\simeq*: H(M)\to H(M\times I), \quad (b)\hspace{3pt} \iota_H\Sigma^2\simeq\Sigma^2\iota_H: H(M)\to H(M\times I\times J).
   $$
\end{lem}

\begin{proof}
    We defer the proof of ($a$) to Lemma \ref{anticommutativitySigmaprop} in Appendix \ref{AppendixB} as it is a bit technical. Note that $+_I$ and $\Sigma: H(M\times I)\to H(M\times I\times J)$ are compatible in the sense that
    $$
\Sigma(\rho+_I\rho')=\Sigma\rho+_I\Sigma\rho', \quad \rho,\rho'\in H(M\times I).
    $$
Then $(b)$ follows from
    $$
\iota_H\Sigma^2\simeq \iota_H\Sigma^2+_I\Sigma(\iota_H\Sigma+_I\Sigma\iota_H)\simeq (\iota_H\Sigma+_I\Sigma\iota_H)\Sigma+_I\Sigma^2\iota_H\simeq\Sigma^2\iota_H.
    $$
    \end{proof}

\begin{proof}[Proof of Theorem \ref{absolutepropTWWvsTepsilon}]
We may assume without loss of generality that $\dim M\geq 5$, for if not replace it by $M\times J^{2k}$ for $k\geq 3$. The effect this has on both sides of the equivalence in the statement is rather mild: as there is a homotopy $C_2$-equivariant equivalence $S^2\approx S^{2\sigma}$, it follows by Lemma \ref{MIKlem} and Corollary \ref{corinvolutionsuptohomotopy}(i) that there are equivalences of spectra
\begin{align*}
\Hsp(M\times J^{2k})_{hC_2}&\simeq_{[\frac12]} \Hsp(M)_{hC_2},\\
\big(S^{(d+2k)\cdot(\sigma-1)-1}\wedge\Whsp(M\times J^{2k};\epsilon)\big)_{hC_2}&\simeq_{[\frac12]}\big(S^{d\cdot(\sigma-1)-1}\wedge\Whsp(M;\epsilon)\big)_{hC_2}.
\end{align*}
In the second equivalence we also use that $\Whsp(-)$ and $\tau_\epsilon$ are homotopy invariants of $(-)$.

Let $k\geq0$ and let $\xi$ denote the $(d+k)$-spherical fibration corresponding to the stable tangent bundle $T(M\times I^k)\oplus\epsilon^1$. If $M$ (and hence $M\times I^k$) is stably parallelisable, the involution $(-1)^{d+k}\tau_{\epsilon}$ is homotopic to $\tau_\xi$ by (\ref{Vogellinvtrivbundle}), and by \cite[Proposition Corollary 2.10]{VogellInvolution} it agrees with $\mathscr T$ (and hence extends $\iota_H$). All in all, we obtain a zig-zag of homotopy $C_2$-equivariant maps 
\begin{equation}\label{monoidalityzigzag}
\begin{tikzcd}[row sep = 7pt]
    \Omega^\infty\Hsp(M\times I^k)\dar[phantom, "\vsimeq"', "\hspace{-26pt}\scalebox{0.78}{(\ref{ThetaMIkeq})}"]&\lar["(\ref{totalalexhcobmap})", "\mathrm{alex}"'] H(M\times I^k)\rar["(\dagger)"] &\Omega^\infty\big(S^{(d+k)\cdot(\sigma-1)-1}\wedge\Whsp(M\times I^k)\big)\dar[phantom, "\vsimeq"]\\
    \Omega^\infty(S^{k\cdot(\sigma-1)}\wedge \Hsp(M))&&\Omega^\infty\big(S^{(d+k)\cdot(\sigma-1)-1}\wedge\Whsp(M)\big)
\end{tikzcd}
\end{equation}
Every space involved in (\ref{monoidalityzigzag}) is an $\mathbb{E}_1$-group if $k\geq 1$: both $H(M\times I^k)$ and $\Omega^\infty(\Hsp(M\times I^k))$ by stacking in the first of the $I^k$-coordinates, and the others by their own infinite loop structures.

\begin{claim}\label{claim1}
    All of the maps in (\ref{monoidalityzigzag}) are $H$-maps if $k\geq 1$. 
\end{claim}
\begin{proof}
The $\mathbb{E}_1$-algebra structure $+_I$ on $\Omega^\infty\Hsp(M\times I^k)$ is equivalent to any of the other ones coming from its infinite loop structure. As (\ref{ThetaMIkeq}) and the right vertical equivalence are infinite loop maps, they are in particular $H$-maps. As for (\ref{totalalexhcobmap}), it is an $\mathbb{E}_1$-map since $(\ref{alexhcob}): H(M\times I^k)\to\widehat{\Hsp}(M\times I^k)_0$ is (by construction).

 Non equivariantly, the map $(\dagger)$ is the composition $(\dagger): H(M\times I^k)\xhookrightarrow{}\mathcal H(M\times I^k)\simeq \Omega^{\infty+1}\Whsp(M\times I^k)$, where the last equivalence is the stable parametrised $h$-cobordism theorem of Waldhausen--Jahren--Rognes \cite[Theorem 0.1]{WJR}. As communicated to us in private by Bjørn Jahren and John Rognes, such equivalence is only stated to hold in the category of spaces (and not of infinite loop spaces, though it should definitely also hold there). We now explain why this equivalence is one of $H$-groups (which is the general consensus, but we couldn't find it written down anywhere): again assume $k=1$. One reduces to the $PL$-case as in \cite[Pages 15--16]{WJR}. Then if $X$ is a simplicial set such that $|X|\simeq M$, the equivalence in the $PL$-setting is induced by a zig-zag of equivalences of simplicial sets (cf the left vertical column of \cite[Eq. (0.4)]{WJR})
 $$
 \begin{tikzcd}
\mathcal{H}(M\times I)_\bullet &\lar["\sim"']\cdot\rar["\sim"] & s\mathscr{C}^h(X\times I),
\end{tikzcd}
 $$
 where $s\mathscr{C}^h(X\times I)$ is the category (seen as a simplicial set by taking its nerve) of finite\footnote{This means that $Y$ is generated by the image of $X\times I\xhookrightarrow{}Y$ and finitely many simplices.} acyclic cofibrations $X\times I\xhookrightarrow{}Y$ together with simple maps over $X\times I$. One can verify that it makes sense to stack in the $I$-direction in each of the simplicial sets involved in the zig-zag, and that the maps between them respect this monoidal structure. Let $\mu_0: s\mathscr{C}^h(X\times I)\times s\mathscr{C}^h(X\times I)\to s\mathscr{C}^h(X\times I)$ stand for this monoidal structure, given explicitly by
 $\mu_0(Y,Z):=Y {}_{X\times {1}}\cup_{X\times 0}Z$. There is another monoidal structure $\mu_1$ induced by the pushout along $X\times I$, ie $\mu_1(Y,Z):=Y\cup_{X\times I} Z$. Sliding gives a homotopy between $\mu_0$ and $\mu_1$: intuitively, the maps
 $$
\mu_t: s\mathscr{C}^h(X\times I)\times s\mathscr{C}^h(X\times I)\longrightarrow s\mathscr{C}^h(X\times I), \quad (Y,Z)\longmapsto Y{}_{X\times [1-t,1]}\cup_{X\times [0,t]} Z, \quad t\in [0,1],
 $$
constitute the homotopy. More precisely, consider the simplicial category $s\widetilde{\mathscr{C}}^h_\bullet(X)$ \cite[Defn. 3.1.1]{WJR} whose objects in simplicial degree $q$ consist of commutative diagrams
$$
\begin{tikzcd}
    X\times \Delta^q_\bullet\ar[dr, "\mathrm{pr}"]\ar[rr,hook, "i", "\sim"'] &&Y\ar[dl, "\pi"]\\
    &\Delta^q_\bullet,&
\end{tikzcd}
$$
where $i$ is an acyclic cofibration and $\pi$ is a Serre fibration. The inclusion $s\mathscr{C}^h(X)\xhookrightarrow{} s\widetilde{\mathscr{C}}^h_\bullet(X)$ as the $0$-simplices is a homotopy equivalence by \cite[Corollary 3.5.2]{WJR}, where a simplicial category is seen as a bisimplicial set by taking its nerve, and a bisimplicial set as an ordinary simplicial set by taking its totalisation. The monoidal structures $\mu_t$ make perfectly good sense in $s\widetilde{\mathscr{C}}^h_\bullet(X\times I)$, and  one can indeed define a simplicial homotopy between $\mu_0$ and $\mu_1$ in this setting resembling the idea above. 

But by Proposition 3.1.1, and Theorems 3.1.7 and 3.3.1 of \cite{WaldhausenAtheory} (see also \cite[p. 5]{WJR}), for any simplicial set $T$, there is a zig-zag of equivalences connecting $|s\mathscr{C}^h(T)|$ and $\Omega\Wh^{PL}(T):=\Omega^{\infty+1}(\mathbf{Wh}^{PL}(T))$ which is monoidal up to homotopy with respect to $\mu_1$ in the domain and the loop structure on the looped Whitehead space. It hence follows that ($\dagger$) is indeed a zig-zag of $H$-maps. 
\end{proof}

\begin{claim}\label{claim2}
For each $k\geq 1$, the diagram
\begin{equation}\label{colimitzgizagdiagram}
\begin{tikzcd}[column sep = 16pt, row sep = 15pt]
    \Omega^\infty(S^{2k\cdot(\sigma-1)}\wedge\Hsp(M))\dar[phantom, "\vapprox"]&\lar H(M\times I^{2k})\dar["\Sigma^2"]\rar & \Omega^\infty(S^{(d+2k)\cdot(\sigma-1)-1}\wedge\Whsp(M))\dar[phantom, "\vapprox"]\\
    \Omega^\infty(S^{(2k+2)\cdot(\sigma-1)}\wedge\Hsp(M))&\lar H(M\times I^{2k+2})\rar & \Omega^\infty(S^{(d+2k+2)\cdot(\sigma-1)-1}\wedge\Whsp(M))
\end{tikzcd}
\end{equation}
is homotopy commutative, where the rows are the zig-zags (\ref{monoidalityzigzag}) and the vertical external maps are induced by the homotopy $C_2$-equivariant equivalence $\s^0\approx \s^{2\cdot(\sigma-1)}$.
\end{claim}

\begin{proof}
    The right hand square is clearly commutative, for both right horizontal maps factor through $\mathcal{H}(M)$. For the commutativity of the left one, by Lemma \ref{MIKlem}, it suffices to argue that the square
\begin{equation}\label{stabilisationsquareHsp}
\begin{tikzcd}
    H(M\times I^{2k})\dar["\Sigma"]\rar["\mathrm{alex}"] & \Omega^\infty\Hsp(M\times I^{2k})\dar[dash, "\vsim", "e_1"']\\
    H(M\times I^{2k+1})\rar["\mathrm{alex}"] & \Omega^\infty\Hsp(M\times I^{2k+1})
\end{tikzcd}
\end{equation}
   homotopy commutes (non-equivariantly\footnote{Homotopy $C_2$-equivariant maps that are homotopic (as ordinary maps) induce the same morphism in the homotopy category of $C_2$-spaces.}). Recall that the map $e_1$, non-equivariantly, is induced by the zig-zags
   $$
   \begin{tikzcd}
\Hsp(M\times I^{2k})_n\rar["s^{\vee}"]&\Omega\Hsp(M\times I^{2k})_{n+1}&\lar["\mathrm{alex}"', "\sim"] \Hsp(M\times I^{2k+1})_n,
\end{tikzcd}
   $$
where $s^\vee$ is the (adjoint to the) structure map of $\Hsp(M\times I^{2k})$. Since (\ref{stabilisationsquareHsp}) is a diagram of $\mathbb{E}_1$-groups by stacking in the first of the $I^{2k}$-coordinates, it suffices to provide a homotopy for the diagram
    $$
\begin{tikzcd}
    H^s(M\times I^{2k})\dar["\Sigma"]\ar[rr,"(\ref{alexhcob})"] & & \Hsp(M\times I^{2k})_0\dar["s^{\vee}"]\\
    H^s(M\times I^{2k+1})\rar["(\ref{alexhcob})"] & \Hsp(M\times I^{2k+1})_0\rar["\mathrm{alex}"] &\Omega\Hsp(M\times I^{2k})_{1}.
\end{tikzcd}
    $$
    As in Remark \ref{deloopingrem}, such homotopy is obtained by delooping (with respect to stacking in the second of the $I^{2k}$-coordinates) the diagram (\ref{commsquareBconc}) of Proposition \ref{alexpropBfunctor}.
\end{proof}

Clearly $\Sigma^2$ is an $H$-map and also homotopy $C_2$-equivariant (with respect to $\iota_H$) by Lemma \ref{suspensionequivariancelem}($b$). Therefore by Claim \ref{claim1}, all of the maps involved in (\ref{colimitzgizagdiagram}) are $H$-maps and homotopy $C_2$-equivariant. Taking the homotopy colimit as $k\to\infty$, we obtain a homotopy $C_2$-equivariant zig-zag
\begin{equation}\label{zigzaginvolutions}
\begin{tikzcd}
\Omega^\infty\Hsp(M)&\lar["\approx"']\underset{k}{\hocolim}\ H(M\times I^{2k})\rar["\approx"] &\Omega^\infty(S^{d\cdot(\sigma-1)-1}\wedge\Whsp(M))
    \end{tikzcd}
\end{equation}
of $H$-maps. The connectivity of, say, the upper horizontal maps in (\ref{colimitzgizagdiagram}) is $\phi(d+2k)\gtrsim (d+2k)/3$ by Igusa's theorem and, as this lower bound increases linearly with $k$, the horizontal maps in (\ref{zigzaginvolutions}) are indeed equivalences. The equivalence in the statement now follows by Corollary \ref{corinvolutionsuptohomotopy}(ii) applied to (\ref{zigzaginvolutions}), and because taking homotopy $C_2$-orbits commutes up to equivalence with $\Omega^\infty(-)$ if $2$ is inverted (as in Corollary \ref{corinvolutionsuptohomotopy}). The proof of Theorem \ref{absolutepropTWWvsTepsilon} is now complete.
\end{proof}


\begin{cor}\label{propTWWvsTepsilon}
    If $M^d$ is stably parallelisable and $P\subset M^d$ is a codimension zero submanifold with $p\leq d-3$ (in the notation of Theorem \ref{EmbWWIThm}), then there is an equivalence away from two
    $$
\Omega^\infty\big(\CEsp(P,M)_{hC_2}\big)\simeq_{[\frac{1}{2}]}\Omega^\infty\big(\big(S^{d\cdot (\sigma-1)-2}\wedge\Whsp(M,M-P;\epsilon)\big)_{hC_2}\big),
    $$
    where $\Whsp(M,M-P;\epsilon)$ stands for the homotopy cofibre of $\Whsp(M-P;\epsilon)\to\Whsp(M;\epsilon)$.
\end{cor}

\begin{proof}
    Note that $\overline{M-P}$ is stably parallelisable because $M$ is. The zig-zag (\ref{zigzaginvolutions}) is functorial with respect to codimension zero embeddings of stably parallelisable manifolds, and hence taking homotopy fibres in the map from $(\ref{zigzaginvolutions})$ with $M$ replaced by $\overline{M-P}$ to $(\ref{zigzaginvolutions})$ itself, we obtain another homotopy $C_2$-equivariant zig-zag of equivalences
    $$
\begin{tikzcd}[ column sep = 20pt]    \Omega^\infty\CEsp(P,M)&\lar["\approx"']\underset{k}{\hocolim}\ \cemb(P\times I^{2k}, M\times I^{2k})\rar["\approx"] &\Omega^\infty(S^{d\cdot(\sigma-1)-2}\wedge\Whsp(M,M-P)).
    \end{tikzcd}
    $$
The same line of reasoning as before yields the desired result.
\end{proof}

\subsection{The canonical involution in algebraic \texorpdfstring{$K$}{K}-theory}\label{canonicalinvsection} We now define the canonical involution $\tau_\epsilon$ on $A(X)$, for $X$ based, and relate it to an involution in the model of $A$-theory via ``spaces of matrices with values in the ring up to homotopy'' $Q_+\Omega X$ \cite[Section 2.2]{WaldhausenAtheory}. Throughout, let $G:=GX$ denote the topological monoid of Moore loops on $X$, and write $\s[G]$ for the $\mathbb{E}_1$-ring spectrum $\s\wedge G_+$.

We will work over the $\infty$-category $\mathsf{Mod}_{\s[G]}$ of right $\s[G]$-module spectra; we will also write $_{\s[G]}\mathsf{Mod}$ for the $\infty$-category of left $\s[G]$-modules. Then for $m\geq 1$, if $\mathrm{Aut}_G(\oplus^{m}\s[G])$ denotes the homotopy invertible components of the mapping space $\mathsf{Mod}_{\s[G]}(\oplus^m\s[G],\oplus^m\s[G])$, Waldhausen showed in \cite[Theorem 2.2.1]{WaldhausenAtheory} that for $X$ connected, there is a natural equivalence 
\begin{equation}\label{BweirdHeq}
A(X)\simeq \Z\times\underset{m}{\hocolim}\ B\mathrm{Aut}_G(\oplus^m \s[G])^+.
\end{equation}
In order to define $\tau_\epsilon$, we will introduce compatible anti-involutions on $\mathrm{Aut}_G(\oplus^{m}\s[G])$ defined in terms of Spanier--Whitehead duality. As $\s[G]$ is not commutative, this duality really arises as an instance of a \textit{duality in the symmetric closed bicategory} $\mathsf{Bimod}_\s$ of bimodule spectra, in the sense of May--Sigurdsson \cite[Section 16.4]{MaySigurdsson}. This duality coincides with the one considered by Vogell in \cite[Section 1]{VogellInvolution}.

\begin{rem}
    We can safely import the duality theory of May--Sigurdsson \cite{MaySigurdsson}: even though it is developed only $2$-categorically, and $\mathsf{Bimod}_\s$ (aka the Morita category of the sphere spectrum, cf~\cite{HaugsengMoritaCategory}) is an $(\infty,2)$-category, the arguments that rely on duality only involve the homotopy $2$-category $\mathrm{Ho}_2(\mathsf{Bimod}_\s)$, which is \emph{symmetric closed} in the sense of \cite[Defn.~16.2.1 \& 16.3.1]{MaySigurdsson}.

\end{rem}

First observe that a right $\s[G]$-module $M$ can always be regarded as a left $\s[G]$-module by
\[
\begin{tikzcd}
    \s[G]\otimes M\rar["\text{swap}"] & M\otimes \s[G]\rar["\Id_M\otimes \text{inv}"] & M\otimes \s[G^{\mathrm{op}}]=M\otimes \s[G]\rar["\mathrm{act}"] & M,
\end{tikzcd}
\]
where ``inv'' stands for inversion in the monoid $G$---write $M_\ell$ for this left $\s[G]$-module. Here $\otimes=\otimes_\s$ stands for the usual smash product of spectra. Note also that $$_{\s[G]}\mathsf{Mod}(M_\ell,M_\ell)\simeq \mathsf{Mod}_{\s[G]}(M,M)$$ as $\mathbb{E}_1$-algebras. If $\nu: \s\to \s[G]_\ell$ denotes the unit, consider the map of spectra
$$
\begin{tikzcd}
\eta_1: \s\simeq \s[G]\otimes_{\s[G]}\s\rar["1\otimes\nu"] & \s[G]\otimes_{\s[G]}\s[G]_\ell
\end{tikzcd}
$$
and the map of $(\s[G],\s[G])$-bimodules
{\[
\begin{tikzcd}
I_1: \s[G]_\ell\otimes \s[G]\rar["\text{inv}\otimes 1"]&\s[G]\otimes \s[G]\rar["\text{act}"] &\s[G].
\end{tikzcd}
\]}
Then $(\eta_1,I_1)$ exhibits $(\s[G], \s[G]_\ell)$ as a dual pair in the sense of \cite[Defn. 16.4.1]{MaySigurdsson} by Example 16.4.3\footnote{This example is really concerned with ordinary rings and modules, but the same argument applies to bimodule spectra. Moreover, it really shows that $\s[G]$, as a right $\s[G]$-module, is \emph{left} dual to $\s[G]$ as a left $\s[G]$-module. We have identified the latter with $\s[G]_\ell$ via the isomorphism of left $\s[G]$-modules $\mathrm{inv}:\s[G]_\ell\cong \s[G]$.} loc. cit. More generally, the map of spectra
$$
\begin{tikzcd}
\eta_m:\s \ar[rr,"\bigoplus_{i,j}\delta_{ij}\eta_1"]&&\bigoplus_{i,j=1}^m \s[G]\otimes_{\s[G]}\s[G]_\ell\cong \bigoplus^m_{j=1}\s[G] \otimes_{\s[G]}\big(\bigoplus^m_{i=1}\s[G]\big)_\ell
\end{tikzcd}
$$
together with the map of $(\s[G],\s[G])$-bimodules
$$
\begin{tikzcd}
I_m: \big(\bigoplus^m_{i=1}\s[G]\big)_\ell \otimes\bigoplus^m_{j=1}\s[G]\cong \bigoplus_{i,j=1}^m \s[G]_\ell\otimes \s[G]\ar[rr,"\bigoplus_{i,j}\delta_{ij}I_1"] & &\s[G],
\end{tikzcd}
$$
exhibit $(\oplus^m\s[G])_\ell$ as a \textit{right} dual to $\oplus^m\s[G]$. (This pretty much follows from $(\s[G],\s[G]_\ell)$ being a dual pair.) Therefore, by \cite[Proposition 16.4.9]{MaySigurdsson}, $I_m$ induces an equivalence of left $\s[G]$-modules $$\widetilde{I}_m: (\oplus^m\s[G])_\ell\simeq D_r(\oplus^m\s[G]):=\underline{\mathrm{Hom}}_{\s[G]}(\oplus^m\s[G],\s[G]),$$ 
where $\underline{\mathrm{Hom}}_{\s[G]}$ denotes the right $\s[G]$-linear mapping spectrum. 

With (\ref{BweirdHeq}) and Lemma \ref{antiinvolutionlem} in mind, the involution $\tau_\epsilon$ on $A(X;\epsilon)$ is then induced by the map of $\mathbb{E}_1$-algebras
\begin{equation}\label{tauepsilonantiinv}
\begin{tikzcd}
    \mathrm{Aut}_G(\oplus^m\s[G])\rar["D_r"] & _{G}\mathrm{Aut}(D_r(\oplus^m\s[G]))^{\mathrm{op}}\overset{\widetilde{I}_\#}\simeq {}_{G}\mathrm{Aut}((\oplus^m\s[G])_\ell)^{\mathrm{op}}\simeq \mathrm{Aut}_G(\oplus^m\s[G])^{\mathrm{op}},
\end{tikzcd}
\end{equation}
where $\widetilde{I}_\#$ stands for conjugation with the equivalence $\widetilde{I}_m$. It will be convenient to think of (\ref{tauepsilonantiinv}) in the following way: let $GL_m(Q_+G)$ denote the union of path components in $(Q_+G)^{m\times m}$ in the image of $\mathrm{Aut}_G(\oplus^m\s[G])$ under the natural equivalence
\begin{align*}
u:\mathsf{Mod}_{\s[G]}(\oplus^m\s[G],\oplus^m\s[G])\overset{\sim}{\longrightarrow}& \ \mathsf{Mod}_{\s[G]}(\oplus^m\s[G],\prod^m\s[G])\\
\simeq &\ \mathsf{Sp}(\s,\s[G])^{m\times m}\simeq (Q_+G)^{m\times m},
\end{align*}
where $\mathsf{Sp}\simeq\mathsf{Mod}_\s$ stands for the $\infty$-category of spectra. So $u:\mathrm{Aut}_G(\oplus^m\s[G])\simeq GL_m(Q_+G)$ and, just as in standard linear algebra, under this equivalence the anti-involution (\ref{tauepsilonantiinv}) corresponds to the rule that sends a matrix $A$ to its conjugate transpose $A^\dagger$ (conjugate with respect to inversion of $G$ in $\s[G]$). More precisely:

\begin{prop}\label{tranpositionprop} Write $\mathrm{End}_G(\oplus^m\s[G]):=\mathsf{Mod}_{\s[G]}(\oplus^m\s[G],\oplus^m\s[G])$ with the action of the cyclic group $C_m$ by conjugation with the permutation automorphisms of $\oplus^m\s[G]$. Let $C_m$ act similarly on $(Q_+G)^{m\times m}$ by conjugation. Then the following square is commutative in the homotopy category of $C_m$-spaces:
\begin{equation}\label{transposematrixsquare}
\begin{tikzcd}
    \mathrm{End}_G(\oplus^m\s[G])\dar["u", "\vsim"'] \rar["(\ref{tauepsilonantiinv})", "\sim"'] &\mathrm{End}_G(\oplus^m\s[G])\dar["u", "\vsim"'] \\
    (Q_+G)^{m\times m}\rar["\dagger", "\sim"'] & (Q_+G)^{m\times m}.
\end{tikzcd}
\end{equation}
\end{prop}

\begin{rem}
    Passing to the homotopy invertible components in (\ref{transposematrixsquare}), we obtain the following commutative square in the homotopy category of $C_m$-spaces
    $$
\begin{tikzcd}
    \mathrm{Aut}_G(\oplus^m\s[G])\dar["u", "\vsim"'] \rar["(\ref{tauepsilonantiinv})", "\sim"'] &\mathrm{Aut}_G(\oplus^m\s[G])\dar["u", "\vsim"'] \\
    GL_m(Q_+G)\rar["\dagger", "\sim"'] & GL_m(Q_+G).
\end{tikzcd}
    $$
    We have suppressed the $(-)^{\mathrm{op}}$ in the codomain of the map (\ref{tauepsilonantiinv}) in the previous squares to emphasise that such squares take place in the homotopy category of $C_m$-spaces, and not that of $\mathbb{E}_1$-spaces. 
\end{rem}

\begin{proof}[Proof of Proposition \ref{tranpositionprop}] First note that all the maps involved in (\ref{transposematrixsquare}) are indeed $C_m$-maps: the only one that is not obviously so is (\ref{tauepsilonantiinv}), but this follows from the observation that $I_m$ is $C_m$-equivariant for the diagonal action on the domain and the trivial action on the target. Note also that the $C_m$-action on $(Q_+G)^{m\times m}$ restricts to a cofree $C_m$-action on each of the right $C_m$-cosets of the diagonal subspace, and hence $(Q_+G)^{m\times m}=\prod^m\mathrm{coInd}_{e}^{C_m}Q_+G$ as a $C_m$-space. Thus, in order to show that (\ref{transposematrixsquare}) commutes in the homotopy category of $C_m$-spaces, it suffices to prove that it commutes in the homotopy category of spaces after postcomposing it with the map
$$
(Q_+G)^{m\times m}=\prod^m\mathrm{coInd}_{e}^{C_m}Q_+G\longrightarrow \prod^mQ_+G
$$
that records the first column of a matrix.

    Now given an endomorphism $h$ of $\oplus^m\s[G]$, the $(m\times m)$-matrix $u(h)=(h_{ij})\in GL_m(Q_+G)$ has components
    $$
    \begin{tikzcd}
h_{ij}: \s\rar["\nu"] & \s[G]\rar["\mathrm{inc}_j"]&\bigoplus_{k=1}^m\s[G]\rar["h"] & \bigoplus_{k=1}^m\s[G]\rar["\mathrm{pr}_i"] & \s[G].
\end{tikzcd}
    $$
    Slightly abusing the notation, we will write $\tau_\epsilon$ to mean (\ref{tauepsilonantiinv}). Then we must only check that $\tau_\epsilon(h)_{ij}$ is homotopic to $\overline{h}_{ji}$, coherently in $h$ (and for $j=1$, though it is still true for all $j$ of course). Observe now that $(-)_\ell: \mathsf{Mod}_{\s[G]}\to {}_{\s[G]}\mathsf{Mod}$ is a functor over $\mathsf{Sp}$, and hence the last equivalence in (\ref{tauepsilonantiinv}) happens over the automorphism space of $\oplus^m\s[G]$ as a regular spectrum. Consequently, $\tau_\epsilon(h)_{ij}$ is by definition the top horizontal composition in the diagram of spectra
    $$
\begin{tikzcd}[scale cd = 0.9,column sep = 13pt, row sep = 15pt]
\s\ar[dr, "\nu"]\rar["\nu"]&\s[G]_\ell\ar[ddr, phantom, "(*_1)"]\dar["\mathrm{inv}", "\vsim"']\rar["\mathrm{inc}_j"]&\ar[dd,"\widetilde{I}_m", "\vsim"']\bigoplus_{k=1}^m\s[G]_\ell\ar[ddr, phantom, "(*_2)"]\rar["\tau_\epsilon(h)"]&\ar[ddr, phantom, "(*_3)"]\ar[dd,"\widetilde{I}_m", "\vsim"']\bigoplus_{k=1}^m\s[G]_\ell\rar["\mathrm{pr}_i"]&\s[G]_\ell\\
    &\s[G]\dar[phantom, "\vsimeq"]&&&\s[G]\dar[phantom, "\vsimeq"]\uar["\mathrm{inv}", "\vsim"']\\
    &\underline{\mathrm{Hom}}_{\s[G]}(\s[G],\s[G])\rar["\mathrm{pr}_j^*"]&\underline{\mathrm{Hom}}_{\s[G]}(\oplus^m\s[G],\s[G])\rar["h^*"]&\underline{\mathrm{Hom}}_{\s[G]}(\oplus^m\s[G],\s[G])\rar["\mathrm{inc}_i^*"]&\underline{\mathrm{Hom}}_{\s[G]}(\s[G],\s[G]).
\end{tikzcd}
    $$
    On the other hand, one recognises the composite that goes through the bottom row to be $\overline{h}_{ji}$. Note that the left triangle is commutative, and that the square $(*_2)$ is too by definition of $\tau_{\epsilon}(h)$. Moreover, ($*_1$) and $(*_3)$ do not depend on $h$; we must then argue that ($*_1$) and $(*_3)$ are commutative up to homotopy. 
    
    For the commutativity of $(*_1)$, first observe that the left vertical composite equivalence of $(*_1)$ coincides up to homotopy with the equivalence of left $\s[G]$-modules $\widetilde{I}_1: \s[G]_\ell\simeq \underline{\mathrm{Hom}}_{\s[G]}(\s[G],\s[G])$. This is because, under the usual tensor-hom adjunction, both maps represent the same element in
    $$
\pi_0\left({}_{\s[G]}\mathsf{Mod}\big(\s[G]_\ell,\underline{\mathrm{Hom}}_{\s[G]}(\s[G],\s[G])\big)\right)\cong\pi_0\big({}_{\s[G]}\mathsf{Mod}_{\s[G]}\big(\s[G]_\ell\otimes \s[G],\s[G]\big)\big)
    $$
    by definition of $I_1$. But now ($*_1$), with $\widetilde{I}_1$ in place of the left vertical composite, commutes up to homotopy as both composites represent the same element in
    $$
    \begin{tikzcd}
\pi_0\left({}_{\s[G]}\mathsf{Mod}\big(\s[G]_\ell,\underline{\mathrm{Hom}}_{\s[G]}(\oplus^m\s[G],\s[G])\big)\right)\cong\pi_0\big({}_{\s[G]}\mathsf{Mod}_{\s[G]}\big(\s[G]_\ell\otimes (\oplus^m\s[G]),\s[G]\big)\big)
\end{tikzcd}
    $$
    simply because the following diagram commutes by definition of $I_1$ and $I_m$:
    $$
\begin{tikzcd}
\s[G]_\ell\otimes\bigoplus^m\s[G]\dar["1\otimes \mathrm{pr}_j"]\rar["\mathrm{inc}_j\otimes 1"] &\bigoplus^m\s[G]_\ell\otimes \bigoplus^m\s[G]\dar["I_m"]\\
    \s[G]_\ell\otimes \s[G]\rar["I_1"] & \s[G].
\end{tikzcd}
    $$
    Finally the commutativity of ($*_3$) follows by a similar reasoning using that
        $$
\begin{tikzcd}
    \bigoplus^m\s[G]_\ell\otimes\s[G]\dar["1\otimes\mathrm{inc}_i"]\rar["\mathrm{pr}_i\otimes 1"] &\s[G]_\ell\otimes \s[G]\dar["I_1"]\\
    \bigoplus^m\s[G]_\ell\otimes \bigoplus^m\s[G]\rar["I_m"] & \s[G]
\end{tikzcd}
    $$
    is also commutative by definition.
\end{proof}

\begin{warn}\label{Milnorinvolutionwarning}
    The canonical involution $\tau_{\epsilon}$ induces an involution on the Whitehead group 
    $$\Wh(X):=\pi_1^s(\Whsp(X))\cong GL(\Z[\pi_1(X)])^{\mathrm{ab}}/(\pm \pi_1(X)).$$ 
    In the foundational paper \cite{MilnorWhiteheadTorsion}, Milnor also defined an involution $\Wh(X)\ni\kappa\mapsto \overline\kappa$ induced by sending a matrix in $GL(\Z[\pi_1(X)])^{\mathrm{ab}}$ to its conjugate transpose (conjugate with respect to inversion in $\pi_1(X)$). This is an actual homomorphism because of the abelianisation present in the general linear group of $\Z[\pi_1(X)]$. It is worth being aware that these two involutions on $\Wh(X)$ are only the same after introducing a minus sign, ie 
    \begin{equation}\label{Milnorinvrelation}
\tau_\epsilon(\kappa)=-\overline{\kappa}.
    \end{equation}
This does \textit{not} contradict the commutativity of (\ref{transposematrixsquare}). On the contrary, this extra minus sign is the result of having to deloop the anti-involution (\ref{tauepsilonantiinv}) in the sense of Lemma \ref{antiinvolutionlem} in order to obtain $\tau_\epsilon$.  
\end{warn}

\subsection{\texorpdfstring{$A$}{A}-theory of a suspension}
In this section we focus our attention on the homotopy type of $A(X;\epsilon)$ when $X$ is the suspension $\Sigma Y$ of a connected based space $Y$. By a theorem of Carlsson--Cohen--Goodwillie--Hsiang\footnote{As pointed out in \cite[p. 543]{AvsTC}, the proof in \cite{Atheorysuspension} has a serious flaw around page 71. This issue was fixed in \cite[Corollary 4.15]{AvsTC}, and in particular the map $\theta$ of (\ref{Aofsuspensioneq}) constructed in \cite[Section 1]{Atheorysuspension} is still an equivalence. We are indebted to Tom Goodwillie for his help in clearing out this matter and for carefully explaining to us another more general principle for which (\ref{Aofsuspensioneq}) holds---namely, it is the observation that if $F$ is a functor (from based spaces to based spaces, say) whose $m$-th derivative spectrum is of the form $ X\mapsto X^{\wedge m}_{hC_m}$ for every $m\geq 1$, then its Taylor tower must split globally. This is indeed the case for the functor $F(-):= \Omega\circ A\circ \Sigma(-)$.} \cite[Theorem 3]{Atheorysuspension}, in such cases there is an equivalence of infinite loop spaces
\begin{equation}\label{Aofsuspensioneq}
    \theta: \prod_{m\geq 1} Q(Y^{\wedge m}_{hC_m})\overset{\sim}{\longrightarrow} \Omega \widetilde{A}(\Sigma Y),
\end{equation}
where $\widetilde{A}(-):=\hofib(A(-)\to A(*))$ and $C_m$ acts on $Y^{\wedge m}$ by cyclic permutation of the factors. In this section, we argue that (\ref{Aofsuspensioneq}) can be upgraded to be $C_2$-equivariant up to homotopy. 

\begin{prop}\label{Aofsuspensionprop}
    Let $Y$ be a connected, based $C_2$-space. There is an equivalence of spectra
    $$
\boldsymbol{\theta}: \bigvee_{m\geq 1} \Sigma^{\infty+\sigma}\big((ED_m)_+\wedge_{C_m} Y^{\wedge m}\big)\overset{\sim}\longrightarrow \widetilde{\Asp}(\Sigma^{\sigma} Y;\epsilon)
    $$
    that is $C_2$-equivariant up to homotopy, and whose underlying (non-equivariant) equivalence induces (\ref{Aofsuspensioneq}). Here $\Sigma^\sigma Y:=S^\sigma\wedge Y$ and $D_m\subset \Sigma_m$ acts on $Y^{\wedge m}$ by 
    $$
g\cdot(y_1\wedge \dots\wedge y_m):= g\cdot y_{g(1)}\wedge\dots \wedge g\cdot y_{g(m)}, \quad g\in D_m, \quad y_i\in Y,
    $$
    where $Y$ is now seen as a based $D_m$-space (on which $C_m$ acts trivially). Finally $C_2=D_m/C_m$ acts on $(ED_m)_+\wedge_{C_m} Y^{\wedge m}$ by its residual diagonal action.  
\end{prop}

\begin{rem}
    The $C_2$-space $\Sigma^{\sigma}Y$ induces an involution on $\Asp(\Sigma^\sigma Y)$ which commutes with the canonical involution $\tau_\epsilon$ described in the previous section, by naturality of its construction. Therefore its composite gives the involution on $\widetilde{\Asp}(\Sigma^\sigma Y;\epsilon)$ appearing in the statement of Proposition \ref{Aofsuspensionprop}. Alternatively, we can allow $X$ in the previous section to mean a $C_2$-space (eg $\Sigma^\sigma Y$), and agree that $\mathrm{inv}: G=GX\to G^{\mathrm{op}}$ there stands for inversion in the monoid $G$ followed by the $C_2$-action on $X$.

    In practice, we will apply Proposition \ref{Aofsuspensionprop} to the case when $Y=S^\sigma\wedge Z$ for some trivial $C_2$-space $Z$, as then $\Sigma^\sigma Y\simeq S^{2\sigma}\wedge Z\approx S^2\wedge Z$ because of the homotopy $C_2$-equivariant equivalence $S^{2\sigma}\approx S^2$. In such case, as $\Asp(-;\epsilon)$ is a homotopy functor, Proposition \ref{Aofsuspensionprop} provides a simple description of the homotopy $C_2$-equivariant homotopy type of $\widetilde{\Asp}(\Sigma^2Z;\epsilon)\approx \widetilde{\Asp}(\Sigma^{2\sigma}Z;\epsilon)$. This, together with Corollary \ref{corinvolutionsuptohomotopy}, can then be used to analyse the homotopy type of $\widetilde{\Asp}(\Sigma^2Z;\epsilon)_{hC_2}$ away from $2$.
\end{rem}

We will need the following observation for the proof of Proposition \ref{Aofsuspensionprop}.

\begin{lem}\label{semidirectproductlem}
    Let $X$ be a based, connected $C_m$-space. The equivalence
    $$
\iota: B((C_m\ltimes \Omega X)^{\mathrm{op}})\simeq B(C_m\ltimes \Omega X)
    $$
    of Lemma \ref{antiinvolutionlem} coincides up to equivalence with the delooping of the inversion map
    $$
\mathrm{inv}:(C_m\ltimes \Omega X)^{\mathrm{op}}\longrightarrow C_m\ltimes \Omega X, \quad (s^{i},\gamma)\mapsto (s^{-i}, s^{i}\cdot \overline{\gamma}),
    $$
    where $\overline{\gamma}$ stands for the loop $\gamma$ with the reversed orientation. 
\end{lem}
\begin{proof}
Given a topological monoid $M$ equipped with a $C_m$-action, it is well-known (see eg \cite[Section II, Theorem 1.12]{Adem2004}) that the classifying space of the semi-direct product $C_m\ltimes M$ is equivalent to $EC_m\times_{C_m}BM$. On the simplicial level, this equivalence is given by
    \begin{align*}
\beta: B_\bullet(C_m\ltimes M)&\overset{\sim}\longrightarrow E_\bullet C_m\times_{C_m}B_\bullet M,\\
\big((s^{i_1},m_1),\dots,(s^{i_q},m_q)\big)&\longmapsto\big[(e,s^{i_1},\dots, s^{i_q}),(s^{i_1}\cdot m_1,s^{i_1+i_2}\cdot m_2,\dots,s^{i_1+\dots+i_q}\cdot m_q)\big].
\end{align*}

Now, for simplicity, we may assume that $\Omega(-)$ stands for the Moore loop space, so that $C_m\ltimes\Omega X$ is strictly associative (a \textit{Moore loop} is a pair $(\gamma,t)$ where $t\geq 0$ and $\gamma: [0,t]\to Y$ is a map with $\gamma(0)=\gamma(t)=*$; multiplication of Moore loops is given by concatenation of loops and addition of its lengths). Also recall that there is an identification $\Delta^q\cong\{\mathbf{v}=(v_1,\dots, v_q): 0\leq v_1\leq\dots\leq v_q\leq 1\}$; under this identification, the self-isomorphism $\Phi_q: \Delta^q\cong \Delta^q$ in the proof of Lemma \ref{antiinvolutionlem} becomes the rule that sends a partition $\mathbf{v}=(0\leq v_1\leq\dots\leq v_q\leq 1)$ to $1-\mathbf{v}:=(0\leq 1-v_q\leq\dots\leq 1-v_1\leq 1)$. There is a $C_m$-equivariant map 
 $$
\xi: B\Omega X\longrightarrow X, \quad \big[(\gamma_1,t_1),\dots, (\gamma_q,t_q),\mathbf{v}\big]\longmapsto \gamma_q\cdot\gamma_{q-1}\cdot\dotso\cdot \gamma_1\left(\sum_{i=1}^qt_iv_i\right)
$$
that is an equivalence if $X$ is connected. This map satisfies the property that
$$
\xi\left(\big[(\overline{\gamma}_1,t_1),\dots, (\overline{\gamma}_q,t_q),\mathbf{v}\big]\right)=\xi\left(\big[(\gamma_q,t_q),\dots, (\gamma_1,t_1),1-\mathbf{v}\big]\right).
$$

With all of this in mind, one verifies that the following diagram commutes up to homotopy
$$
\begin{tikzcd}[column sep = 30pt]
    B((C_m\ltimes \Omega X)^{\mathrm{op}})\dar["B(\mathrm{inv})"]\rar["\iota", "\sim"'] & B(C_m\ltimes \Omega X)\rar["\beta", "\sim"'] &EC_m\times_{C_m}B\Omega X\dar["EC_m\times_{C_m}\xi", "\vsim"']\\
    B(C_m\ltimes \Omega X)\rar["\beta", "\sim"'] &EC_m\times_{C_m}B\Omega X\rar["EC_m\times_{C_m}\xi", "\sim"'] &EC_m\times_{C_m}X.    
\end{tikzcd}
$$
\end{proof}

\begin{proof}[Proof of Proposition \ref{Aofsuspensionprop}]
In the notation of the previous section, we let $X=\Sigma^\sigma Y$ now, so that $G=GX=\Omega^\sigma \Sigma^\sigma Y$ as a monoid with anti-involution (ie ``inv'' now means inversion in the monoid $G$ followed by the $C_2$-action on $X=\Sigma^\sigma Y$). For each $m\geq 1$, let us write $\varrho: GL_m(Q_+G)\simeq \mathrm{Aut}_G(\oplus^m\s[G])$ for $u^{-1}$ (meaning $u$ as a wrong way equivalence); it should be thought of as given by the rule that sends a matrix $h=(h_{ij})\in \map(\s,\s[G])^{m\times m}\cong\map_G(\s[G],\s[G])^{m\times m}$ to
$$
\begin{tikzcd}
\varrho(h): \bigoplus_{j=1}^m\s[G]\ar[rr,"\bigoplus_j\big(\bigoplus_ih_{ij}\big)"]&&\bigoplus_{i=1}^m\s[G].
\end{tikzcd}
$$
The map $\theta$ of (\ref{Aofsuspensioneq}) is constructed in several steps in \cite[Section 1]{Atheorysuspension}, each of which we now upgrade to the homotopy $C_2$-equivariant setting. As these homotopy $C_2$-actions will get mixed up with strict $C_m$-actions, it will be more convenient and clear, at least throughout the first few steps, to avoid speaking about ``homotopy equivariance'' and rather regard a homotopy involution as what it is, ie a map whose square happens to be homotopic to the identity.

\begin{step}
For each $m\geq 2$, consider the $D_m$-equivariant map of spaces
    $$
\widetilde{\theta}_{m,1}: Y^{\times m}\longrightarrow GL_m(Q_+G), \quad (y_1,\dots, y_m)\longmapsto \resizebox{5cm}{!}{$\begin{pmatrix}
    1 & y_1-1 & &&&\\
    &1& y_2-1 &&&\\
    &&1&\ddots&&\\
    &&&\ddots&&\\
    &&&&1&y_{m-1}-1\\
    y_m-1 &&&&&1
\end{pmatrix}$}
    $$
    where, if $D_m:=\langle s,r\mid s^m=r^2=rsrs=e\rangle$, the notation is as follows:
\begin{itemize}\setlength\itemsep{3pt}
    \item A point $y\in Y$ is identified in $G$ with the path $\eta(y):=(t\mapsto t\wedge y)\in G$, which is itself identified with a point in $\{1\}\times QG\subset QS^0\times QG\simeq Q_+G$. Then $y-1$ is the corresponding point in $\{0\}\times QG\subset Q_+G$. Here the $n$-th component of $QS^0$ has been fixed a basepoint $n\in QS^0$.

    \item The action of $D_m$ on $(y_1,\dots, y_m)\in Y^{\times m}$ is given by
    $$
s\cdot (y_1,\dots, y_m):=(y_m, y_1,\dots, y_{m-1}), \quad r\cdot(y_1,\dots, y_m):=(y^*_{m-1}, y^*_{m-2},\dots,  y_1^*, y_m^*),
    $$
    where $y\mapsto  y^*$ denotes the $C_2$-action on $Y$.

    \item Let $S, R\in GL_m(\Z)$ be the permutation matrices that send the $i$-th unit vector $e_i$ to $e_{i+1}$ and $e_{m+1-i}$ (with subindexes taken modulo $m$), respectively. Then $D_m$ acts on $A\in GL_m(Q_+ G)$ by
    $$
s\cdot A:=SAS^{-1}, \qquad r\cdot A:= RA^{\dagger}R.
    $$
    In other words, $r$ acts by transposition along the ``$x=y$''-axis together with conjugation on $G$. 
\end{itemize}
    We also define $\widetilde{\theta}_{1,1}: Y\to GL_1(Q_+G)=(Q_+G)^\times$ by sending $y\in Y$ to $y\in \{1\}\times QG\subset Q_+G$. We note that $\widetilde{\theta}_{m,1}(y_1,\dots, y_m)$ is homotopy invertible by choosing a path from each of the $y_i$'s to the basepoint $*\in Y$. Then for $m\geq 1$, define $\theta_{m,1}$ as the composite
    $$
\begin{tikzcd}
    \theta_{m,1}: Y^{\times m}\rar["\widetilde{\theta}_{m,1}"]&GL_m(Q_+G)\rar["\varrho", "\sim"'] & \mathrm{Aut}_G(\oplus^m\s[G]).
\end{tikzcd}
    $$
    By construction, $\theta_{m,1}$ is a $C_m$-map as $\widetilde{\theta}_{m,1}$ and $u$ are. Recall that $s\in C_m\subset D_m$ acts on $\mathrm{Aut}_G(\oplus^m\s[G])$ by conjugation with $S\in \mathrm{Aut}_G(\oplus^m\s[G])$, which we denote by $S_{\#}$.
    \end{step}

\begin{step}
    Recall that the free $\mathbb{E}_1$-algebra on a based, connected space $X$ is naturally equivalent to $\Omega\Sigma X$. Therefore, we can extend $\theta_{m,1}$ to a $C_m$-equivariant $\mathbb{E}_1$-map
    $$
\theta_{m,2}: \Omega\Sigma(Y^{\times m})\longrightarrow \mathrm{Aut}_G(\oplus^m\s[G]).
    $$
For any based space $X$, let us write $\sigma: \Omega X\to (\Omega X)^{\mathrm{op}}$ for inversion in $\Omega X$ (ie reversing the loop direction). Given a $C_m$-space $X$, we will denote $X^{\mathrm{op}_{C_m}}$ for $X$ with the opposite $C_m$-action (ie that in which $s$ acts by $s^{-1}$, which is a valid left action as $C_m$ is abelian). If $X$ additionally has an $\mathbb{E}_1$-structure, we will write $X^{\mathrm{op},\mathrm{op}_{C_m}}$ for $X$ with both the opposite $\mathbb{E}_1$-structure and the opposite $C_m$-action. Then the square of $\mathbb{E}_1$-maps
\begin{equation}\label{htpyDmsquare}
\begin{tikzcd}
    \Omega\Sigma(Y^{\times m})\dar["\sigma\circ r"]\rar["\theta_{m,2}"] & \mathrm{Aut}_G(\oplus^m\s[G])\dar["R_{\#}\circ\tau_\epsilon"]\\
    \Omega\Sigma(Y^{\times m})^{\mathrm{op},\mathrm{op}_{C_m}}\rar["\theta_{m,2}^{\mathrm{op},\mathrm{op}_{C_m}}"] & \mathrm{Aut}_G(\oplus^m\s[G])^{\mathrm{op},\mathrm{op}_{C_m}}
\end{tikzcd}
\end{equation}
commutes in the homotopy category of $\mathbb{E}_1$-spaces with a $C_m$-action. Here $r: \Omega\Sigma (Y^{\times m})\to \Omega\Sigma (Y^{\times m})^{\mathrm{op}_{C_m}}$ is induced by the action of $r\in D_m$ on $Y^{\times m}$ together with the flip of the suspension coordinate, and $\tau_\epsilon$ really stands for (\ref{tauepsilonantiinv}). To see this, consider the diagram
\begin{equation}\label{biggerDmsquare}
\begin{tikzcd}
Y^{\times m}\ar[dr, "r"]\ar[dd, "\sigma\circ r\circ\eta"']\ar[rr,"\widetilde{\theta}_{m,1}"] && GL_m(Q_+G)\dar["r=R_{\#}\circ \dagger"]\rar["\varrho", "\sim"'] & \mathrm{Aut}_G(\oplus^m\s[G])\ar[dd,"R_{\#}\circ\tau_{\epsilon}"]\\
&(Y^{\times m})^{\mathrm{op}_{C_m}}\rar["\widetilde{\theta}_{m,1}^{\hspace{2pt}\mathrm{op}_{C_m}}"]\ar[ld, "\eta^{\mathrm{op}_{C_m}}"]&GL_m(Q_+G)^{\mathrm{op}_{C_m}}\ar[dr, "\varrho^{\mathrm{op}_{C_m}}", "\sim"']&\\
\Omega\Sigma(Y^{\times m})^{\mathrm{op}_{C_m}}\ar[rrr, "\theta_{m,2}^{\mathrm{op}_{C_m}}"]&&& \mathrm{Aut}_G(\oplus^m\s[G])^{\mathrm{op}_{C_m}}
\end{tikzcd}
\end{equation}
of $C_m$-spaces. By definition, $\theta_{m,2}$ is the $\mathbb{E}_1$-map induced from the top horizontal composite in (\ref{biggerDmsquare}). Thus, in order to show that (\ref{htpyDmsquare}) homotopy commutes as $C_m$-equivariant $\mathbb{E}_1$-maps, it suffices to show that the outer square of (\ref{biggerDmsquare}) commutes in the homotopy category of $C_m$-spaces. But each of its subsquares/triangles commute in this category: indeed the lower subsquare does so by definition of $\theta_{m,2}$ (after applying $(-)^{\mathrm{op}_{C_m}}$), the left subtriangle and the upper subsquare too by an easy check, and the right subsquare by Proposition \ref{tranpositionprop} and the observation that $R_{\#}\circ u=u^{\mathrm{op}_{C_m}}\circ R_\#$.
\end{step}

\begin{step}
The $C_m$-equivariant $\mathbb{E}_1$-map $\theta_{m,2}$ gives rise to an $\mathbb{E}_1$-map
$$
\begin{tikzcd}
\theta_{m,3}: C_m\ltimes \Omega\Sigma(Y^{\times m})\rar["C_m\ltimes \theta_{m,2}"] &C_m\ltimes \mathrm{Aut}_G(\oplus^m \s[G])\rar["\mu"]&\mathrm{Aut}_G(\oplus^m \s[G]),
\end{tikzcd}
$$
where $\mu(s^{i},h):=S^{i}h$. Observe that $\mu$ is indeed an $\mathbb{E}_1$-map as
$$
\mu((s^{i},h)\cdot (s^j,h'))=\mu(s^{i+j},S^{-j}h S^{j}h')=S^{i}hS^{j}h'=\mu(s^{i},h)\mu(s^j,h').
$$
Now from the homotopy commutativity of (\ref{htpyDmsquare}), it immediately follows that the left subsquare in
\begin{equation}\label{ltimesdiagram}
\begin{tikzcd}[column sep = 40pt]
C_m\ltimes \Omega\Sigma(Y^{\times m})\dar["C_m\ltimes (\sigma\circ r)"]\ar[rr,bend left=12pt, "\theta_{m,3}"] \rar["C_m\ltimes\theta_{m,2}"]&C_m\ltimes \mathrm{Aut}_G(\oplus^m\s[G])\dar["C_m\ltimes (R_\#\circ\tau_{\epsilon})"]\rar["\mu"]& \mathrm{Aut}_G(\oplus^m\s[G])\dar["R_{\#}\circ\tau_\epsilon"]\\
C_m\ltimes \Omega\Sigma(Y^{\times m})^{\mathrm{op},\mathrm{op}_{C_m}}\rar["C_m\ltimes \theta_{m,2}^{\mathrm{op},\mathrm{op}_{C_m}}"]&C_m\ltimes \mathrm{Aut}_G(\oplus^m\s[G])^{\mathrm{op},\mathrm{op}_{C_m}}\rar["\mu^{\mathrm{op}}"]&\mathrm{Aut}_G(\oplus^m\s[G])^{\mathrm{op}}
\end{tikzcd}
\end{equation}
commutes in the homotopy category of $\mathbb{E}_1$-spaces. Here $\mu^{\mathrm{op}}(s^{i},h):=hS^{i}$, and since $\tau_\epsilon(S)=S^{\dagger}=S^{-1}$ and $RS=S^{-1}R$, it easily follows that the right subsquare also commutes as $\mathbb{E}_1$-maps. So the outer square of (\ref{ltimesdiagram}) commutes in the homotopy category of $\mathbb{E}_1$-spaces. 

But given an $\mathbb{E}_1$-space $X$ equipped with a $C_m$-action, there is an isomorphism of $\mathbb{E}_1$-spaces
$$
\alpha:C_m\ltimes X^{\mathrm{op},\mathrm{op}_{C_m}}\overset\cong\longrightarrow (C_m\ltimes X)^{\mathrm{op}}, \quad (s^{i},x)\longmapsto (s^{i},s^{-i}\cdot x).
$$
Under this identification, the lower horizontal composite of (\ref{ltimesdiagram}) becomes $\theta_{m,3}^{\mathrm{op}}$, and hence
\begin{equation}\label{improvedltimesdiagram}
\begin{tikzcd}
    C_m\ltimes\Omega\Sigma(Y^{\times m})\dar["\alpha\circ(C_m\ltimes (\sigma\circ r))"]\rar["\theta_{m,3}"] & \mathrm{Aut}_G(\oplus^m\s[G])\dar["R_\#\circ\tau_\epsilon"]\\
    (C_m\ltimes \Omega\Sigma(Y^{\times m}))^{\mathrm{op}}\rar["\theta_{m,3}^{\mathrm{op}}"] & \mathrm{Aut}_G(\oplus^m\s[G])^{\mathrm{op}}
\end{tikzcd}
\end{equation}
is commutative in the homotopy category of $\mathbb{E}_1$-spaces.
\end{step}

\begin{step}
We wish to deloop (\ref{improvedltimesdiagram}), viewing the vertical maps as anti-involutions of their respective domains, and appealing to Lemma \ref{antiinvolutionlem} to do so. But by Lemma \ref{semidirectproductlem}, the delooping of the anti-involution $\alpha\circ(C_m\ltimes(\sigma\circ r))$ is homotopic to the delooping of the involution $\mathrm{inv}\circ\alpha\circ(C_m\ltimes(\sigma\circ r))$, where $\mathrm{inv}$ stands for inversion in the $\mathbb{E}_1$-space $C_m\ltimes\Omega\Sigma(Y^{\times m})$. It is given explictly by $\mathrm{inv}(s^{i},\gamma):=(s^{-i},s^{i}\cdot\sigma(\gamma))$, and hence we see that 
$$
\mathrm{inv}\circ\alpha\circ(C_m\ltimes(\sigma\circ r)):(s^{i},\gamma)\longmapsto (s^{-i},r\cdot \gamma).
$$
We denote this map simply by $\mathrm{inv}\ltimes r$. From now on we treat $r$ as the action map on the $D_m$-space $\Sigma^{\sigma}(Y^{\times m})$, where $\sigma$ is seen as a $D_m$-representation on which $C_m$ acts trivially. Putting this together, the delooped version of (\ref{improvedltimesdiagram}) yields a homotopy commutative square of spaces
$$
\begin{tikzcd}
    B(C_m\ltimes \Omega\Sigma^{\sigma}(Y^{\times m}))\dar["B(\mathrm{inv}\ltimes r)"]\rar["B(\theta_{m,3})"] & B\mathrm{Aut}_G(\oplus^m\s[G])\dar["\overline{B}(R_\#\circ\tau_\epsilon)"]\\
    B(C_m\ltimes \Omega\Sigma^{\sigma}(Y^{\times m}))\rar["B(\theta_{m,3})"] & B\mathrm{Aut}_G(\oplus^m\s[G]),
\end{tikzcd}
$$
where the notation $\overline{B}(-)$ stands for delooping in the sense of Lemma \ref{antiinvolutionlem}. 

To simplify the terms in this last diagram, first observe that as $Y^{\times m}$ is connected, we have
$$
B(C_m\ltimes\Omega\Sigma^{\sigma}(Y^{\times m}))\simeq ED_m\times_{C_m}B\Omega\Sigma^{\sigma}(Y^{\times m})\simeq ED_m\times_{C_m}\Sigma^{\sigma}(Y^{\times m}).
$$
The inversion on $C_m$ coincides with the residual $C_2=D_m/C_m$-action on $C_m$ by conjugation, which explains why we chose to write $ED_m$ instead of $EC_m$. As for the right hand side, note that $R_\#$ is an inner automorphism of $\mathrm{Aut}_G(\oplus^m \s[G])$, and hence it induces a map homotopic to the identity on the classifying space level \cite[Section II, Theorem 1.9]{Adem2004}. But delooping is functorial, so $\overline{B}(R_\#\circ\tau_\epsilon)$ and $\overline{B}(\tau_\epsilon)=:\tau_{\epsilon}$ are homotopic involutions on $B\mathrm{Aut}_G(\oplus^m \s[G])$. All together, we obtain a homotopy $C_2$-equivariant map
$$
\scalebox{0.93}{$\theta_{m,4}: ED_m \times_{C_m} \Sigma^\sigma (Y^{\times m}) \xlongrightarrow{B(\theta_{m,3})^+} B\mathrm{Aut}_G(\oplus^m \s[G])^+ \subset \Z \times B\mathrm{Aut}_G(\oplus^m \s[G])^+ \to A(\Sigma^\sigma Y;\epsilon)$}
    $$
where the last map is the passage to the colimit as $m\to \infty$ (see (\ref{BweirdHeq})).
\end{step}

\begin{step}
The following diagram commutes up to homotopy:
    $$
    \begin{tikzcd}
        \theta_{m,4}:ED_m/C_m\simeq BC_m\rar["Bj"]\dar &B\mathrm{Aut}_G(\oplus^m\s)\rar\dar & A(*;\epsilon)\dar\\
        \theta_{m,4}:ED_m\times_{C_m}\Sigma^\sigma(Y^{\times m})\rar & B\mathrm{Aut}_G(\oplus^m\s[G])\rar & A(\Sigma^\sigma Y;\epsilon),
    \end{tikzcd}
    $$
    where $j: C_m\to \mathrm{Aut}_G(\oplus^m\s)$ is the inclusion of the permutation automorphisms. As $A(-;\epsilon)=\Omega^\infty\Asp(-;\epsilon)$, we can adjoin the $\Omega^\infty(-)$ to get a similar homotopy commutative diagram of spectra. Then passing to vertical cofibres and noting that $C_m$ acts trivially on the suspension coordinate of $\Sigma^\sigma(Y^{\times m})$, we get a homotopy $C_2$-equivariant map of spectra
    $$
\theta_{m,5}: \Sigma^{\infty+\sigma}\big((ED_m)_+\wedge_{C_m} Y^{\times m}\big)\longrightarrow \widetilde{\Asp}(\Sigma^\sigma Y;\epsilon).
    $$
\end{step}

\begin{step}
    Now by \cite[Lemma 1.4]{Atheorysuspension} (see also \cite[Lemma 2.4]{CohenCarlssonTCFLS}), the obvious projection $(ED_m)_+\wedge_{C_m} Y^{\times m}\to (ED_m)_+\wedge_{C_m} Y^{\wedge m}$ has a stable section $\Sigma^\infty\big((ED_m)_+\wedge_{C_m} Y^{\wedge m}\big)\to \Sigma^\infty\big((ED_m)_+\wedge_{C_m} Y^{\times m}\big)$ that is $D_m/C_m$-equivariant. This observation gives rise to a homotopy $C_2$-equivariant map of spectra
$$
\boldsymbol{\theta}_{m,6}: \Sigma^{\infty+\sigma}\big ((ED_m)_+\wedge_{C_m} Y^{\wedge m}\big )\longrightarrow \widetilde{\Asp}(\Sigma^\sigma Y;\epsilon).
    $$
\end{step}
     Finally set $\boldsymbol{\theta}$ to be 
    $$
    \begin{tikzcd}
\boldsymbol{\theta}: \bigvee_{m\geq 1} \Sigma^{\infty+\sigma}\big((ED_m)_+\wedge_{C_m} Y^{\wedge m}\big)\ar[rr,"\bigvee_{m\geqs 1}(\boldsymbol{\theta}_{m,6})"] && \bigvee_{m\geq 1}  \widetilde{\Asp}(\Sigma^\sigma Y;\epsilon)\rar[two heads] & \widetilde{\Asp}(\Sigma^\sigma Y;\epsilon).
\end{tikzcd}
    $$
    This map is homotopy $C_2$-equivariant by construction, and non-equivariantly yields (\ref{Aofsuspensioneq}) after applying $\Omega^{\infty+\sigma}(-)$. This latter map is an equivalence of infinite loop spaces by \cite[Theorem 1.6]{Atheorysuspension}, and as both the domain and codomain of $\boldsymbol{\theta}$ are $1$-connective, it follows that $\boldsymbol{\theta}$ is itself an equivalence of spectra. This concludes the proof of Proposition \ref{Aofsuspensionprop}. 
\end{proof}

\begin{cor}\label{Whofsuspension}
    Let $Y$ be a connected, based $C_2$-space. Denote by $\bWhsp(-)$ the homotopy fibre $\hofib(\Whsp(-)\to \Whsp(*))$. Then there is an equivalence of spectra
    $$
\boldsymbol{\theta}: \bigvee_{m\geq 2}\Sigma^{\infty+\sigma}\big((ED_m)_+\wedge_{C_m}Y^{\wedge m}\big)\overset{\sim}\longrightarrow \bWhsp(\Sigma^\sigma Y;\epsilon)
    $$
    that is $C_2$-equivariant up to homotopy. 
\end{cor}

\begin{proof}
    It is clear from the construction that the map 
    $$
\Sigma^\infty(\Sigma^\sigma Y)\simeq\Sigma^{\infty+\sigma}\big((ED_1)_+\wedge_{C_1}Y^{\wedge 1}\big)\xlongrightarrow{\boldsymbol{\theta}_{1,6}} \widetilde{\Asp}(\Sigma^\sigma Y;\epsilon)
    $$
    is the (reduced version of the) usual inclusion of the stable homotopy into $A$-theory. Thus its cofibre is $\bWhsp(\Sigma^\sigma Y;\epsilon)$, and the claim follows immediately.
\end{proof}

\begin{rem}
    As the reader may have noticed by now, the last two sections are a tiny bit technical, and one may wonder if there could be alternative approaches to deal with them. Such an approach that may come to mind is to use trace methods to analyse $\widetilde{A}(\Sigma^\sigma Y;\epsilon)$, since it coincides (non-equivariantly) with the reduced $TC$ of $\s[\Omega\Sigma Y]$ (as $Y$ is connected). In fact, recent developments have been made towards a (genuine) $C_2$-equivariant version of topological cyclic homology for ring spectra with anti-involutions, commonly known as \textit{real topological cyclic homology} (cf \cite{HøgenhavenRealTC, RealKtheoryHessMads,DottoMoiPatchkoria}). This approach has two caveats:
    \begin{itemize}
        \item A \textit{real cyclotomic trace} map does not yet exist (at the time of writing). The construction of such a map was supposed to appear in \cite{RealKtheoryHessMads}, but it never saw the light in the end. This is, nevertheless, current work in progress by Harpaz--Nikolaus--Shah \cite[Page 24]{RealCycTraceMap}.

        \item Even though much is known about the $p$-complete homotopy type of the $TC$ of spherical group rings (cf \cite{AvsTC} or \cite[Section 4.3]{NikolausScholze}), the analysis of its integral homotopy type does not seem to be present in the literature.
    \end{itemize}
    For these two reasons, we preferred to proceed as we have. 
\end{rem}

\section{The homotopy type of spaces of long knots}\label{longknotsection}

This section is devoted to Theorem \ref{LongKnotsThm}, which describes the homotopy type of $\emb_\partial(D^p,D^d)$ for $p\leq d-3$ and $d\geq 5$, localised at odd primes and up to the concordance embedding stable range $\phi_{\cemb}(d,p)$. After its proof, which will not take too much effort given the results in the preceeding sections, we will draw some conclusions on the homotopy groups of spaces of long knots. For convenience let us recall the statement of Theorem \ref{LongKnotsThm}. Recall that $\psi_m$ stands for the real $m$-dimensional permutation representation of the dihedral group $D_m$ and $\sigma$ for the sign representation, regarded as a $D_m$-representation by restricting along the determinant $D_m\xhookrightarrow{}O(2)\overset\det\to \{\pm 1\}= C_2$.

\begin{thm*}[Theorem \ref{LongKnotsThm}]
For $p\leq d-3$ and $d\geq 5$, consider the virtual $D_m$-representations
$$
    \rho_m:=(d+1)(\sigma-1)+\psi_m\otimes(d-p-3+\sigma).
$$  
    Then the homotopy fibre sequence (\ref{knotfibseq}), upon localising away from $2$ and taking $(\phi_{\cemb}(d,p)-1)$-th Postnikov sections, takes the form
$$
 \begin{tikzcd}
\prod_{m\geqs 2} \Omega^{\infty}\big(\s^{\rho_m}_{hD_m}\big)\rar&\emb_\partial(D^p, D^d)\rar&\Omega^{p}\hofib\big(G(d-p)/O(d-p)\to G/O\big).
\end{tikzcd}
$$
    The resulting sequence is split if $p\geq 2$, and splits after being looped once if $p=1$.
\end{thm*}

Recall from Remark \ref{thmBrem}($ii$) what we mean by localising the spaces $\emb_\partial(D^p,D^d)$ and $\bemb_\partial(D^p,D^d)$.

\subsection{Proof of Theorem \ref{LongKnotsThm}} \label{outlineproofsection}
Recall from (\ref{pseudoembdefn}) that $\emb_\partial^{(\sim)}(P,M)$ denotes $\hofib_\iota(\emb_\partial(P,M)\to \bemb_\partial(P,M))$ when $\iota$ is clear from the context. We saw in Corollary \ref{splittinglongknotcor} that the fibration sequence
\begin{equation}\label{knotfibseq}
\begin{tikzcd}
\pemb_\partial(D^p,D^d)\rar & \emb_\partial(D^p, D^d)\rar & \bemb_{\partial}(D^p,D^d),
\end{tikzcd}
\end{equation}
upon localising at odd primes and taking $(\phi_{\cemb}(d,p)-1)$-th Postnikov sections, is split for $2\leq p\leq d-3$, and splits for $p=1$ after looping once. So we need to describe the exterior terms of (\ref{knotfibseq}) after inverting $2$.

For the block embeddings, the graphing map 
\begin{equation}\label{blockgraphingmap}
\Gamma: \Omega^p\bemb(*,D^{d-p})\overset{\sim}\longrightarrow \bemb_\partial(D^p,D^{d-p}\times D^p)\cong \bemb_\partial(D^p,D^{d})
\end{equation}
of (\ref{graphmapemb}) is an equivalence by inspection. Then by \cite[Theorem 2.2.1]{GoodwillieKleinWeiss} and the example right after it, when $d-p\geq 3$ and $d\geq 5$ (see Remark \ref{p1d4caserem} below), it follows that
\begin{align}
    \bemb_\partial(D^p,D^d)&\simeq \Omega^p\hofib(O/O(d-p)\to G/G(d-p))\label{blocklongknotequiv}\simeq \Omega^p\hofib(G(d-p)/O(d-p)\to G/O),
    \end{align}
yielding the base of (\ref{knotfibseq}). 

\begin{rem}\label{p1d4caserem}
    The equivalence (\ref{blocklongknotequiv}) is only valid if $d-p\geq 3$ and $d\geq 5$. As pointed out right after \cite[Theorem 2.2.1]{GoodwillieKleinWeiss}, the second condition is not that important. For instance in the case $p=1$ and $d=4$, it follows directly from (\ref{blockgraphingmap}) and (\ref{blocklongknotequiv}) that
    $$
\Omega\bemb_\partial(D^1,D^4)\simeq \bemb_\partial(D^2,D^5)\simeq \Omega^2\hofib(G(3)/O(3)\to G/O).
    $$
    The codimension condition $d-p\geq 3$, however, is essential.
\end{rem}

For the fibre of (\ref{knotfibseq}), we know by Theorem \ref{EmbWWIThm} that for $N=\phi_{\cemb}(d,p)-1$, there is an equivalence
$$
\tau_{\leqs N}\pemb_\partial(D^p,D^d) \simeq \tau_{\leqs N}\ \Omega^\infty\left(\CEsp(D^p,D^d)_{hC_2}\right).
$$
We now use Corollaries \ref{propTWWvsTepsilon} and \ref{Whofsuspension} to describe the right hand side of the equivalence above.

\begin{prop}
    For $\rho_m$ as in Theorem \ref{LongKnotsThm}, there is an equivalence
    $$
\Omega^\infty\left(\CEsp(D^p,D^d)_{hC_2}\right)\simeq_{[\frac{1}{2}]} \prod_{m\geqs 2} \Omega^{\infty}\big(\s^{\rho_m}_{hD_m}\big).
    $$
\end{prop}

\begin{proof}
First observe that there is a homotopy $C_2$-equivariant equivalence $S^2\approx S^{2\sigma}$. So by Corollary \ref{Whofsuspension}, for each $n\geq 2$ there is a homotopy $C_2$-equivariant equivalence of spectra
$$
\bWhsp(S^n;\epsilon)\approx \bWhsp(\Sigma^{\sigma}S^{n-2+\sigma};\epsilon)\approx \bigvee_{m\geq 2}\Sigma^{\infty+\sigma}\big(S^{\psi_m\otimes (n-2+\sigma)}_{hC_m}\big)
$$
Using this for $n=d-p-1\geq 2$, we obtain a chain of equivalences
\begin{align*}
    \Omega^\infty\left(\CEsp(D^p,D^d)_{hC_2}\right)&\simeq_{[\frac12]} \Omega^\infty\left((S^{d\cdot(\sigma-1)-1}\wedge\bWhsp(S^{d-p-1}))_{hC_2}\right)\\
    &\simeq_{[\frac12]}\Omega^\infty\big(\big(\bigvee_{m\geq 2}\s^{(d+1)\cdot(\sigma-1)}\wedge \s^{\psi_m\otimes(n-2+\sigma)}_{hC_m}\big)_{hC_2}\big).\\
    &=\prod_{m\geqs 2} \Omega^{\infty}\big(\s^{\rho_m}_{hD_m}\big).
\end{align*}
The first equivalence follows from Corollary \ref{propTWWvsTepsilon}, together with the observation that $\bWhsp(S^{d-p-1})\simeq \Sigma^{-1}\Whsp(D^p,S^{d-p-1};\epsilon)$ as both $\tau_\epsilon$ and $\Whsp(-)$ are homotopy invariants of $(-)$. The second equivalence is a consequence of the previous argument and Corollary \ref{corinvolutionsuptohomotopy}. This establishes the desired equivalence.
\end{proof}

All together, this concludes the proof of Theorem \ref{LongKnotsThm}. \qed

\begin{rem}[Topological version of Theorem \ref{LongKnotsThm}]\label{TopLongKnotRem} The space $\emb_{\partial}^{\mathrm{Top}}(D^p,D^d)$ of topological long knots is contractible (for all $p\leq d$)  by the Alexander trick. We could still be interested in the homotopy type of the space $\emb_{\partial_0}^{\mathrm{Top}}(D^p\times D^{d-p},D^d)$ of thickened topological long knots with $p\leq d-3$, and one can get a description of it localised away from $2$ and up to the concordance embedding stable range, similar to the one in Theorem \ref{LongKnotsThm}---let us explain how. As before, we have a homotopy fibre sequence
$$
\begin{tikzcd}[column sep = 12pt]
    \emb_{\partial_0}^{\mathrm{Top}, (\sim)}(D^p\times D^{d-p},D^d)\rar & \emb_{\partial_0}^{\mathrm{Top}}(D^p\times D^{d-p},D^d)\rar & \bemb_{\partial_0}^{\mathrm{Top}}(D^p\times D^{d-p},D^d)
\end{tikzcd}
$$
which, upon localising at odd primes and taking $(\phi_{\cemb}(d,p)-1)$-th Postnikov sections, is split for $2\leq p\leq d-3$, and splits for $p=1$ after looping once. So we should describe the side terms.

For the block embeddings, consider the space\footnote{If $q\geq 3$, this definition of $\widetilde{\rmTop}(q)$ coincides with the one given in Remark \ref{codimzerotopremark}.}
$$
B\widetilde{\mathrm{Top}}(q):=\operatorname{holim}\left(\begin{tikzcd}[sep = small]
    &B\mathrm{Top}\dar\\
    BG(q)\rar & BG
\end{tikzcd}\right)
$$
which is responsible for the classification of topological block normal bundles if $q\geq 3$ (cf \cite[Section 2]{RourkeSandI}, \cite{RourkeSandIII} and \cite[Section 11]{Wall}). As $\bemb^{\mathrm{Top}}_{\partial}(D^p,D^d)$ is contractible by the Alexander trick, it then follows that
$$
\bemb^{\mathrm{Top}}_{\partial_0}(D^p\times D^{d-p},D^d)\simeq \map_\partial(D^p,\widetilde{\mathrm{Top}}(d-p))=\Omega^{p}\widetilde{\mathrm{Top}}(d-p).
$$

As for the pseudoisotopy embeddings, the topological version of Theorem \ref{EmbWWIThm} (see Remark \ref{remembWWI}) tells us that for $N=\phi_{\cemb(d,p)}-1$, there is an equivalence
\begin{equation}\label{LKtopEq}
\tau_{\leqs N}\emb_{\partial_0}^{\mathrm{Top},(\sim)}(D^p\times D^{d-p},D^d))\simeq \tau_{\leqs N}\ \Omega^\infty\left(\CEsp^{\mathrm{Top}}(D^p\times D^{d-p},D^d)_{hC_2}\right),
\end{equation}
where $\CEsp^{\mathrm{Top}}(P,M)$ stands for the first orthogonal derivative of the topological analogue of $F(-)$ (as in Corollary \ref{propTWWvsTepsilon}, this is a $C_2$-spectrum whose infinite loop space is equivalent to that of $\Sigma^{-2}\Whsptop(M,M-P;\epsilon)$). Noting that there is fibre sequence of $C_2$-spectra $\Whsp(M;\epsilon)\to \Whsptop(M;\epsilon)\to \Sigma\Whsp(*;\epsilon)\wedge M_+$, one verifies that the infinite loop space in the right hand side of (\ref{LKtopEq}) fits in a fibre sequence away from $2$
$$
\begin{tikzcd}[column sep = 15pt]
    \prod_{m\geqs 2} \Omega^{\infty}\big(\s^{\rho_m}_{hD_m}\big)\rar & \Omega^\infty\left(\CEsp^{\mathrm{Top}}(D^p\times D^{d-p},D^d)_{hC_2}\right)\rar
    &\Omega^\infty\big((S^{d\cdot\sigma-p-2}\wedge \Whsp(*;\epsilon))_{hC_2}\big).
\end{tikzcd}
$$
In particular, it is easy to check (eg rationally) that the left hand side of (\ref{LKtopEq}) is not contractible (at least if $d$ is sufficiently large). This was claimed in Remark \ref{codimzerotopremark}.
\end{rem}

\subsection{On the homotopy groups of spaces of long knots}\label{htpyLKsection} We can get plenty of information about the homotopy groups of $\emb_\partial(D^p,D^d)$ from Theorem \ref{LongKnotsThm}. First observe that by Morlet's lemma of disjunction \cite[Theorem 3.1]{BurgLashRoth} (and Proposition \ref{positivecodimprop} to reduce to the codimension zero case), the pseudoisotopy embedding space $\emb_\partial^{(\sim)}(D^p,D^d)$ is at least $(2(d-p-2)-1)$-connected. So by (\ref{blocklongknotequiv}), it follows that
\begin{equation}\label{lowerhtpygroupslongknots}
    \pi_*(\emb_\partial(D^p,D^d))\cong \pi_{*+p}(\hofib(G(d-p)/O(d-p)\to G/O)), \quad *<2(d-p-2).
\end{equation}

\begin{rem}\label{BudneyComparisonRem}
This should be compared to work of Budney \cite[Proposition 3.9]{BudneyIntegralLongKnots}. The main result there is the computation of the first non-trivial homotopy group of $\emb_\partial(D^p,D^d)$ for $d-p\geq 3$, which lies in degree $2d-3p-3$, together with a geometric interpretation of the generators. This is
\begin{equation}\label{BudneyComputation}
\pi_{2d-3p-3}(\emb_\partial(D^p,D^d))\cong\left\{
\begin{array}{cl}
    \Z, & \text{$p=1$ or $d-p$ is odd,}  \\
    \Z/2, & \text{$p\geq 2$ and $d-p$ is even.}
\end{array}
\right.
\end{equation}
From our point of view, he shows that $\hofib(G(d-p)/O(d-p)\to G/O)$ is exactly $(2d-2p-4)$-connected, which follows by work of Haefliger (see Section 3, Equation 4.11 and Corollary 6.6 of \cite{Haefliger}), and computes the group $\pi_{2d-2p-3}(\hofib(G(d-p)/O(d-p)\to G/O))$.

Regarding the Gromoll filtration on $\pi_0(\emb_{\partial}(D^p,D^d))$ for $d-p\geq3$, two things are stated:
\begin{enumerate}[label=(\arabic*), itemsep = 3pt]
    \item The Gromoll degree of the elements of $\pi_0(\emb_{\partial}(D^p,D^d))$ is at least $2d-2p-4$.\label{Budney1}

    \item When $2d-3p-3=0$, the Gromoll degree of the elements of $\pi_0(\emb_{\partial}(D^p,D^d))$ is $p-1$.  \label{Budney2}
\end{enumerate}
Claim \ref{Budney1} follows from the fact that $\emb^{(\sim)}_\partial(D^{p-j},D^{d-j})$ is $(2d-2p-5)$-connected and that the block graphing map \eqref{blockgraphingmap} is an equivalence. Given \eqref{BudneyComputation}, we can deduce \ref{Budney2} from our work \emph{only} when $\pi_{0}(\emb_\partial(D^p,D^d))\cong \Z$ (using the simple observation that a homomorphism $A\to \Z$ from a finitely generated abelian group $A$ is surjective if and only if $A[\tfrac{1}{2}]\to \Z[\tfrac{1}{2}]$ is so). 

If we allow ourselves to localise away from $2$, though, much more can be said about the Gromoll filtration on $\pi_0(\emb_\partial(D^p,D^d))[\tfrac{1}{2}]$. For example, by Corollary \ref{CorGromollFiltration}(i), it follows that
\begin{enumerate}[label=(\arabic*)]\setcounter{enumi}{2}
    \item For $d,p \geq 0$ with $d-p \geq 3$, the elements of $\pi_0(\emb_{\partial}(D^p,D^d))[\tfrac{1}{2}]$ have Gromoll degree at least
    \[
    j:=\min\left(p-1, \lfloor\tfrac{2d-p-6}{2}\rfloor\right).
    \]
\end{enumerate}
See Corollary~\ref{CorGromollFiltration} for further consequences for the Gromoll filtration of $\pi_k(\emb_{\partial}(D^p,D^d))[\tfrac{1}{2}]$ for $k \geq 0$.

\end{rem}

We now return to the explicit computation of the homotopy groups of $\emb_{\partial}(D^p,D^d)$. Recall that $\phi_{\cemb}(d,p)\geq 2d-p-5$ by work of Goodwillie--Krannich--Kupers \cite{ConcEmbStablerange}, and so the space $\emb^{(\sim)}_\partial(D^p,D^d)$ has interesting homotopy in degrees from $2d-2p-4$ up to that range that we can understand by Theorem \ref{LongKnotsThm}. For any odd prime $\ell$ there are isomorphisms in degrees $*\leq \phi_{\cemb}(d,p)-1$
$$
\pi_*(\emb_\partial(D^p,D^d))_{(\ell)}\cong \pi_{*+p}(\hofib(G(d-p)/O(d-p)\to G/O))_{(\ell)}\oplus\bigoplus_{m\geq 2}\pi_*^s(\s^{\rho_m}_{hD_m})_{(\ell)},
$$
and if $\phi=\phi_{\cemb}(d,p)$, there is also an exact sequence of abelian groups
$$
\begin{tikzcd}[column sep=14pt, scale = 0.1]
    \bigoplus_{m\geq 2}\pi_{\phi}^s(\s^{\rho_m}_{hD_m})_{(\ell)}\rar &\pi_\phi(\emb_\partial(D^p,D^d))_{(\ell)}\rar[two heads] & \pi_{\phi+p}(\hofib(G(d-p)/O(d-p)\to G/O))_{(\ell)},
\end{tikzcd}
$$
where $A_{(\ell)}$ denotes $A\otimes \Z_{(\ell)}$, for $A$ an abelian group. It remains to understand the groups $\pi_*^s(\s^{\rho_m}_{hD_m})_{(\ell)}$, which are easier to study when $\ell$ is coprime to $m$. 

\begin{prop}\label{htpygroupsEmcoprimeprop}
    Let $m\geq 2$ and $d-p\geq 3$. For $\ell\nmid 2m$ a prime,
    \begin{itemize}
    \item if $d$ is even and $p$ is even, then
    $$
\pi_*^s(\s^{\rho_m}_{hD_m})_{(\ell)}\cong\left\{
\begin{array}{cl}
    \pi_{*-m(d-p-2)}^s\otimes \Z_{(\ell)}, & m=3,5,7,\dots\\
    0, & \text{otherwise}.
\end{array}
\right.
    $$

    \item if $d$ is odd and $p$ is odd, then
    $$
\pi_*^s(\s^{\rho_m}_{hD_m})_{(\ell)}\cong\left\{
\begin{array}{cl}
    \pi_{*-m(d-p-2)}^s\otimes \Z_{(\ell)}, & m=2,4,6,\dots\\
    0, & \text{otherwise}.
\end{array}
\right.
    $$

    \item if $d$ is even and $p$ is odd, then
    $$
\pi_*^s(\s^{\rho_m}_{hD_m})_{(\ell)}\cong\left\{
\begin{array}{cl}
    \pi_{*-m(d-p-2)}^s\otimes \Z_{(\ell)}, & m=5,9,13,\dots\\
    0, & \text{otherwise}.
\end{array}
\right.
$$ 

    \item if $d$ is odd and $p$ is even, then
    $$
\pi_*^s(\s^{\rho_m}_{hD_m})_{(\ell)}\cong\left\{
\begin{array}{cl}
    \pi_{*-m(d-p-2)}^s\otimes \Z_{(\ell)}, & m=3,7,11,\dots\\
    0, & \text{otherwise}.
\end{array}
\right.
    $$
\end{itemize}
\end{prop}

\begin{rem}[Rational homotopy of spaces of long knots]\label{rationallongknotsrem} The rational homology and homotopy of $\emb_\partial(D^p, D^d)$ for $d-p\geq 3$ has been extensively studied in recent years (see eg \cite{TurchinLongKnots1, AroneTurchinHomologyLK, AroneTurchinHomotopyLK}) through the lens of embedding calculus and its relation to the little disks operads and their formality, finally culminating in the work of Fresse--Turchin--Willwacher \cite{FresseTurchinWillwacher}. There they compute the rational homotopy groups of $\overline{\emb}_\partial(D^p,D^d):=\hofib_\iota(\emb_\partial(D^p,D^d)\to \mathrm{Imm}_\partial(D^p,D^d))$ as the homology of the \textit{hairy graph complex} (shifted appropriately). Observationally, our results correspond to the $0$- and $1$-loop order parts of this graph complex up to degree $\phi_{\cemb}(d,p)\geq 2d-p-5$, where higher loop orders are still not seen. More precisely, the $0$-loop part corresponds to the rational homotopy of $G/G(d-p)$, the lowest summand (ie $m=1$ when $d-p$ is even) of the $1$-loop part appears as that of $O/O(d-p)$, and the higher summands of the $1$-loop part come from the rational homotopy of the spectra $\s^{\rho_m}_{hD_m}$ for $m\geq 2$, which we just computed. It is worth noting that:
\begin{itemize}
    \item The first non-trivial rational homotopy group of $\emb_\partial(D^p,D^d)$ coming from the $2$-loop part of the hairy graph complex lies in degree $2d-p-4$ when both $d$ and $p$ are odd (cf \cite[Eq. 3]{FresseTurchinWillwacher}). Therefore, the lower bound $\phi_{\cemb}(d,p)\geq 2d-p-5$ on the concordance stable range of Goodwillie--Krannich--Kupers is quite sharp.
\vspace{2pt}
    \item The $1$-loop part of the hairy graph complex seems to be completely generated by the spectra $\s^{\rho_m}_{hD_m}$ for $m\geq 2$ in all degrees outside of the concordance embedding stable range. In other words, the computations in \cite{FresseTurchinWillwacher} give evidence for the existence of a rational left splitting of the Weiss--Williams map
    $$
\Phi^{\emb}:\emb^{(\sim)}_\partial(D^p,D^d)\longrightarrow \prod_{m\geqs 2} \Omega^{\infty}(\s^{\rho_m}_{hD_m}).
    $$
    To investigate this, one should first understand the attachment of the second orthogonal derivative $\Theta F^{(2)}$ of the orthogonal functor $F(U):=\emb^b_\partial(D^p\times U, D^d\times U)$ to its Taylor tower (\ref{OrthTower}). It would also be interesting to understand the integral picture. 
\end{itemize}
    
\end{rem}

\begin{proof}[Proof of Proposition \ref{htpygroupsEmcoprimeprop}]
When $\ell$ is coprime to $2m$, we have that
$$
\pi_*^s(\s^{\rho_m}_{hD_m})_{(\ell)}\cong H_0(D_m; \pi_*^s(\s^{\rho_m}))_{(\ell)}
$$
by the homotopy fixed point spectral sequence, because the higher group homology of $D_m$ is $2m$-torsion. Let $t$ and $r$ denote the generators of $D_m$ with $t^m=e$ and $rtr=t^{-1}$ such that
\begin{align*}
\psi_m(t):&\ \R^m\ni(a_1,\dots,a_m)\longmapsto (a_m,a_1,\dots,a_{m-1}),\\
\psi_m(r):&\ \R^m\ni(a_1,\dots,a_m)\longmapsto (a_m,a_{m-1},\dots,a_{1}).
\end{align*}
Then $t$ and $r$ act on the group $\pi^s_*(\s^{(d+1)(\sigma-1)+\psi_m\otimes(d-p-3+\sigma)})$ by $(-1)^{\epsilon_t}$ and $(-1)^{\epsilon_{r}}$, respectively, where
\begin{align*}
    \epsilon_t=(m-1)(d-p-2), \qquad \epsilon_r&=\underbrace{d+1}_{(1)}+\underbrace{\frac{1}{2}m(m-1)(d-p-3)}_{(2)}+\underbrace{m+\frac{1}{2}m(m-1)}_{(3)}\\
    &\equiv d+1+m+\frac{1}{2}m(m-1)(d-p-2)\mod 2.
\end{align*}
The terms ($1$), ($2$) and ($3$) are the contributions coming, respectively, from the summands $(d+1)(\sigma-1)$, $\psi_m\otimes(d-p-3)$ and $\psi_m\otimes\sigma$ of $\rho_m$. One then readily verifies that the groups $H_0(D_m; \pi^s_*(\s^{\rho_m}))$ are given by the fomulae in the statement.
\end{proof}

A bit more interesting are the homotopy groups $\pi^s_*(\s^{\rho_m}_{hD_m})_{(\ell)}$ when $\ell$ is odd but divides $m$. We treat the case when $\ell=m=3$, which hopefully serves as a sample computation for other cases. 

\begin{prop}\label{htpyat3prop}
    The first few homotopy groups $\pi_*^s(\s^{\rho_3}_{hD_3})\otimes \Z_{(3)}$ of the spectrum $\s^{\rho_3}_{hD_3}$, localised at $3$ and when $d-p=3$, are given in Table \ref{tablehtpygroupsd-p3}. Equally coloured groups in this table correspond to the same case depending on whether certain differentials in Figure \ref{ASSpictured-p3} vanish or not. Entries containing ``$?$'' correspond to potentially more complicated answers that do not conveniently fit in the table.
\end{prop}
    \begin{table}[h]
\centering
\caption{$\pi_*^s(\s^{\rho_3}_{hD_3})\otimes\mathbb{Z}_{(3)}$ for $d-p=3$ for low values of $*\geq 3$.}
\label{tablehtpygroupsd-p3}

    \scalebox{0.88}{
\begin{tabular}{|c||*{11}{c}|}\hline
$*$
&3&4&5&6&7&8&9&10&11&12&13 \\\hline\hline
$p$ even & $\Z_{(3)}$& $0$& $0$& $\Z/9$& $0$ & $0$ & $0$ & $\Z/9$ & $0$ & $0$ & $\Z/3$\\\hline
$p$ odd & $0$& $\Z/3$& $0$& $0$& $0$ & $\Z/3$ & $0$ & $0$ & $\Z/3$ & $\Z/9$ & $0$\\\hline

\multicolumn{12}{c}

\\
\hline
$*$
&14&15&16&17&18&19&20&21&22&23&24 \\\hline\hline
$p$ even & $\Z/27$& $0$& $0$& $\begin{array}{c}\color{teal}{\Z/3}\\[3pt] \color{orange}\Z/3\oplus \Z/3\end{array}$& $\begin{array}{c} \color{teal}\Z/3^4\\[2pt] \color{orange}\Z/3^5\end{array}$ & $0$ & $?$ & $?$ & $\Z/9$ & $\Z/3$ & $0$\\\hline
$p$ odd & $\Z/3$& $\Z/3$& $\Z/3$& $0$& $0$& $0$ & $\Z/3$ & $\Z/3$ & $0$ & $0$ & $\Z/9\oplus \Z/3$ \\\hline
\end{tabular}}
\end{table}

\begin{rem}
    Since the time of writing this article, more extended computations of the groups $\pi_*^s(\s^{\rho_3}_{hD_3})\otimes\mathbb{Z}_{(3)}$ for $d-p=3$ have become available in the Master's thesis of Andrés Morán Lamas \cite{MoranLamasThesis}. For instance, he computes that in the ``$p$ even'' case, the groups in degrees $*=20$ and $21$ are both $\mathbb{Z}/3$. He also provides an extended version of Table \ref{tablehtpygroupsd-p3} in \cite[Tables 3.3 \& 3.4]{MoranLamasThesis}, computing most of the groups up to degree $*\leq 39$. We are grateful to him for his enthusiasm in this particular computation.
\end{rem}

To prove Proposition \ref{htpyat3prop}, we will need to understand the cohomology of $\s^{\rho_3}_{hC_3}$ as a module over the Steenrod algebra $\mathcal A_3$, which we recall is generated by the Steenrod powers $P^k$ and the Bockstein operation $\beta$. 
\begin{lem}\label{E3moduleoverA3lem}
    The spectrum cohomology of $\s^{\rho_3}_{hC_3}$ is given by
    $$
H^*(\s^{\rho_3}_{hC_3}; \mathbb{F}_3)\cong \mathbb{F}_3\langle u\rangle\otimes_{\mathbb{F}_3}\mathbb{F}_3[\alpha,s]/(\alpha^2), \qquad |\alpha|=1, \quad |s|=2, \quad |u|=3(d-p-2),
    $$
with
    $$
P^k(u\alpha^{i}s^j)=\left(\sum_{r=0}^k \binom{d-p-2}{r} \binom{j}{k-r} \right)u\alpha^{i}s^{j+2k}, \qquad \beta(u\alpha^{i}s^j)=\left\{\begin{array}{cl}
    0, & i=0, \\
    -us^{j+1}, & i=1.
\end{array}\right.
    $$
Moreover $C_2=D_3/C_3$ acts on $H^*(\s^{\rho_3}_{hC_3};\mathbb{F}_3)$ by
$
u\alpha^{i}s^j\mapsto (-1)^{p+i+j}u\alpha^{i}s^j
$.
\end{lem}

\begin{proof}
    The key observation to carry out this calculation is that the $C_3$-representation $\psi_3\mid_{C_3}$ decomposes as $1+\theta$, where $\theta$ is the $2$-dimensional representation pulled back from the standard complex $U(1)\cong SO(2)$-representation on $\C\cong\R^2$. In particular, the associated vector bundle of the representation $\psi_m\otimes (d-p-3+\sigma)\mid_{C_m}$ is orientable; write $u\in H^*(\s^{\rho_3}_{hC_3};\mathbb{F}_3)$ for the corresponding Thom class. The $\mathbb{F}_3$-cohomology of $BC_3$ is $\mathbb{F}_3[\alpha,s]/(\alpha^2)$ with $|\alpha|=1$, $|s|=2$ and
    $$
\beta(\alpha)=s, \qquad P^k(\alpha^{i}s^j)=\binom{j}{k}\alpha^{i}s^{j+2k}.
    $$
So it remains to understand the action of $\mathcal A_3$ on $u$. Clearly $\beta(u)=0$, as if $\beta(u)=nu\alpha$ for some $n\in\mathbb{F}_3$, then  $0=\beta^2(u)=n^2u\alpha^2-nus$ and hence $n=0$. For $u_{\theta}$ the Thom class of $\theta$, 
$
P^1(u_{\theta})=u_{\theta}^3=u_{\theta}c_1(\theta)^2=u_{\theta}s^2
$. It then easily follows that $P^k(u)=\binom{d-p-2}{k}us^{2k}$.

The residual $D_3/C_3$-action on the $\mathbb{F}_3$-cohomology of $BC_3$ sends $s$ to $-s$, and hence $\alpha$ to $-\alpha$. This action also switches the orientation of the vector bundle associated to $\theta$, and hence that of $\psi_3=1+\theta$. So that of $\psi_3\otimes\sigma$ does not change then, as $\psi_3$ is odd-dimensional. All in all, this means that $u$ is sent by the $C_2$-action to $(-1)^\epsilon u$ with
$
\epsilon=(d+1)+(d-p-3)\equiv p\mod 2
$, where the first term is the contribution of $\s^{(d+1)(\sigma-1)}$. This establishes the claim. 
\end{proof}

\begin{proof}[Proof of Proposition \ref{htpyat3prop}]
Consider the $\mathcal A_3$-submodules of $H^*(\s^{\rho_3}_{hC_3};\mathbb{F}_3)$ given by
$$
J_0:=\langle u\alpha^{i}s^j: i+j\equiv 0\mod 2\rangle, \qquad J_1:=\langle u\alpha^{i}s^j: i+j\equiv 1\mod 2\rangle. 
$$
Then $H^*(\s^{\rho_3}_{hC_3};\mathbb{F}_3)=J_0\oplus J_1$ as $\mathcal A_3$-modules, and $J_p$, where $p$ here is taken mod $2$, is the $(+1)$-eigenspace of the residual $C_2=D_3/C_3$-action. Therefore $H^*(\s^{\rho_3}_{hD_3};\mathbb{F}_3)=J_p$ as an $\mathcal{A}_3$-module, and the Adams spectral sequence of $J_p$ then converges to the stable homotopy of $\s^{\rho_3}_{hD_3}=(\s^{\rho_3}_{hC_3})_{hC_2}$. The software \cite{SSprogramme} computes the $E_2$-page of this spectral sequence, from which we determine some of the first homotopy groups of $\s^{\rho_3}_{hD_3}$. We illustrate the case $d-p=3$ in Figure \ref{ASSpictured-p3}.

\begin{figure}[h]
    \centering
    \includegraphics[scale = 0.91]{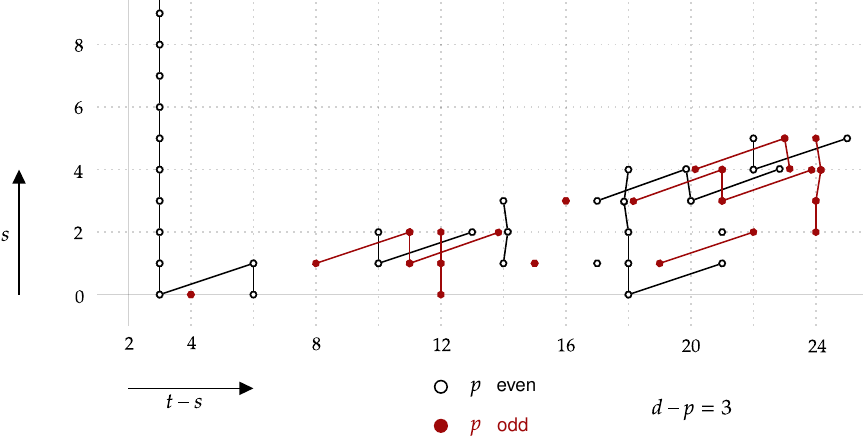}
    \caption{$E_2$-page of the Adams spectral sequence at the prime $3$ for $\s^{\rho_3}_{hD_3}$ when $d-p=3$. Here $s$ denotes the degree in the Adams filtration and $t-s$ is the total degree.}
    \label{ASSpictured-p3}
\end{figure}

Apart from standard arguments exploiting the Leibniz rule and multiplicative structures of the differentials in Figure \ref{ASSpictured-p3}, we can appeal to some more refined tricks that allow us to solve the first few possible non-zero differentials in the case when $p$ is even. Denote by $\widehat{\theta}$ the $D_3$-representation pulled back from the standard $O(2)$-representation on $\R^2$. Observe that $\widehat{\theta}\mid_{C_3}\equiv \theta$ in the notation of Lemma \ref{E3moduleoverA3lem}. Write $\underline{S}(\widehat{\theta}\otimes\sigma\mid_{C_3})$ for the unit sphere bundle of the associated vector bundle $ED_3\times_{C_3}(\widehat{\theta}\otimes\sigma)$. As it is an $S^1$-bundle over $BC_3$, $\underline{S}(\widehat\theta\otimes \sigma\mid_{C_3})$ must itself be a $K(\pi,1)$, and in fact it must be homotopy equivalent to $S^1$ because $q: \underline{S}(\widehat\theta\otimes \sigma\mid_{C_3})\to BC_3$ does not admit a section (its euler class is $s\in H^2(BC_3;\mathbb{F}_3)=\mathbb{F}[\alpha,s]/(\alpha^2)$). Moreover the homology class represented by $q$ is the dual of $\alpha$, and hence the residual $C_2=D_3/C_3$-action on $\underline{S}(\widehat\theta\otimes \sigma\mid_{C_3})\simeq S^1$ must have degree $-1$. So there is an equivalence of unbased spaces $\underline S(\widehat\theta\otimes\sigma\mid_{C_3})\simeq S^\sigma$ which is $C_2$-equivariant up to homotopy. We thus get a cofibration
$$
\begin{tikzcd}
S^\sigma_+\simeq \underline{S}(\widehat{\theta}\otimes\sigma\mid_{C_3})_+\rar["q"] &(BC_3)_+\rar[two heads] & \mathrm{Th}(\widehat\theta\otimes\sigma\mid_{C_3})\simeq S^{-\sigma}\wedge \mathrm{Th}(\psi_3\otimes\sigma\mid_{C_3})
\end{tikzcd}
$$
which is $C_2$-equivariant up to homotopy (see Notation \ref{equivariantnotn}($i$)). By equipping both $S^\sigma$ and $BC_3$ with distinguished basepoints which are fixed under the respective involutions and which match under $q$, we can get rid of the added basepoints and yield a homotopy cofibre sequence of $C_2$-spectra
$$
\begin{tikzcd}
\s^{(d+1)(\sigma-1)+2\sigma}\rar["q"] &S^{(d+1)(\sigma-1)+\sigma}\wedge\Sigma^{\infty}BC_3\rar[two heads] & \s^{\rho_3}_{hC_3}.
\end{tikzcd}
$$
Then, upon inverting $2$ and taking homotopy $C_2$-orbits in the sequence above, we obtain equivalences of spectra
\begin{equation}\label{C2cofibreE3spectrum}
    \s^{\rho_3}_{hD_3}\simeq_{[\frac{1}{2}]}\left\{\begin{array}{cl}
(\Sigma^{\infty+1} BC_3)_{hC_2}\simeq_{[\frac{1}{2}]} \Sigma^{\infty+1} BD_3, & \text{$d$ even (so $p$ odd),}\\[0.4cm]
\mathrm{hocofib}(q_{hC_2}: \s^2\to (S^{\sigma}\wedge \Sigma^{\infty}BC_3)_{hC_2}), & \text{$d$ odd (so $p$ even).}
\end{array}
\right.
\end{equation}
Note that the second equivalence in the ``$d$ even'' case really only holds after inverting $2$: by definition, there is an equivalence $(\Sigma^\infty_+BC_3)_{hC_2}\simeq \Sigma^\infty_{+}BD_3$---we need to get rid of the ``+''s. Now $\Sigma^\infty_+BC_3\simeq \s\oplus \Sigma^\infty BC_3$ is a $C_2$-equivariant equivalence, and $\s_{hC_2}= \Sigma^\infty_+BC_2\simeq \s\oplus\Sigma^\infty BC_2$; the first summand $\s$ cancels with that of $\Sigma^\infty_{+}BD_3\simeq \s\oplus \Sigma^\infty BD_3$, so a copy of $\Sigma^\infty BC_2$ remains, which is contractible only upon inverting $2$.

Now by the Kahn--Priddy theorem \cite{KahnPriddy} at the prime $3$, the transfer-like map $\Sigma^{\infty+1} BD_3\to \tau_{>1}\s^1$ is split surjective on homotopy groups localised at $3$, and hence by (\ref{C2cofibreE3spectrum}), the group $\pi_*^s(\s^{\rho_3}_{hD_3})_{(3)}$ split surjects onto $\pi^s_{*-1}\otimes \Z_{(3)}$ for $*> 1$ when $p$ is odd. We will use this fact together with knowledge of $\pi_*^s$ to determine the differentials of the red spectral sequence in Figure \ref{ASSpictured-p3}. For convenience, let us reillustrate a different portion of it in Figure \ref{ASSred}.

\begin{figure}[h]
    \centering
    \includegraphics[scale = 0.91]{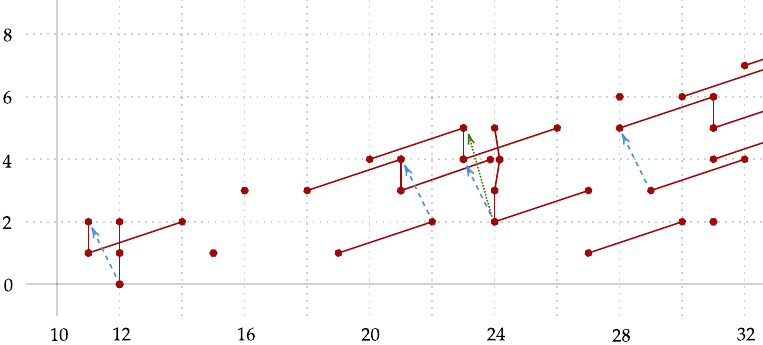}
    \caption{$E_2$-page of Adams spectral sequence at the prime $3$ for $\s^{\rho_3}_{hD_3}$ when $d-p=3$ and $p$ is odd. Some of the $d_2$ (dashed blue) and $d_3$ (dotted green) differentials that will be analysed are depicted.}
    \label{ASSred}
\end{figure}

The first possible non-zero differential goes from $(12,0)$ to $(11,2)$. Depending on its (non-)vanishing, we must have that either
$$
(a)\ \ \pi_*^s(\s^{\rho_3}_{hD_3})_{(3)}\cong \left\{\begin{array}{cc}
   \Z/9,  & *=11, \\
    \Z/27, & *=12,
\end{array}\right.
\quad \text{or}\qquad (b)\ \ \pi_*^s(\s^{\rho_3}_{hD_3})_{(3)}\cong \left\{\begin{array}{cc}
   \Z/3,  & *=11, \\
    \Z/9, & *=12.
\end{array}\right.
$$
Since $\pi_{10}^s\otimes \Z_{(3)}\cong \Z/3$, we must rule out possibility $(a)$ as $\Z/9$ does not split surject onto $\Z/3$. In other words, the $d_2$ differential in Figure \ref{ASSpictured-p3} from $(12,0)$ to $(11,2)$ is non-zero and $(b)$ holds.

The $d_2$ differential from $(22,2)$ to $(21, 4)$ must be non-zero, and hence so will that from $(19,1)$ to $(18,3)$ by the multiplicative structure. Indeed if such differential was zero, it would follow that $\pi_{21}^s(\s^{\rho_3}_{hD_3})_{(3)}\cong \Z/9$, which does not split surject onto $\pi_{20}^s\otimes \Z_{(3)}\cong \Z/3$.  

The $d_2$ differential from $(24, 2)$ to $(23,4)$ must also be non-zero (and hence that from $(24,3)$ to $(23,5)$) by a similar reason; indeed if it were trivial, then $\pi_{23}^s(\s^{\rho_3}_{hC_3})_{(3)}$ would be isomorphic to $\Z/3\oplus \Z/81$ or $\Z/3\oplus \Z/27$ (depending on whether the $d_3$ differential from $(24,2)$ to $(24, 5)$ vanishes or not), either of which do not split surject onto $\pi_{23}^s\otimes \Z_{(3)}\cong\Z/3\oplus \Z/9$. 

The arguments we just made above, together with standard ones exploiting the Leibniz rule and the multiplicative structure of the differentials in the spectral sequence in Figure \ref{ASSpictured-p3}, establish the homotopy groups appearing in Table \ref{tablehtpygroupsd-p3}.
\end{proof}

One can keep up using the Kahn--Priddy theorem and go quite far up determining all possible non-zero differentials when $p$ is odd; we leave it as a fun exercise to the eager reader. One would also hope that (\ref{C2cofibreE3spectrum}) could be used in the case when $p$ is even, but this approach has somehow been inconclusive for us (at least for the first non-zero differentials that cannot be ruled out by elementary means).

\appendix

\section{The first orthogonal derivative}\label{orthcalcappendix}
Throughout, let $F: \mathcal{J}_0\to \mathsf{Top}_*$ be an orthogonal functor. In this section, we present an explicit model for the structure maps (\ref{orthstructuremap}) of the $O(1)$-spectrum $\Theta F^{(1)}$, and compare our convention for its $O(1)$-action with that of \cite[Proposition 3.1]{WeissOrthCalc}.

\subsection{An explicit model of the (pre-)spectrum \texorpdfstring{$\Theta F^{(1)}$}{ThetaF1}}\label{ExplicitStructureMapSection}

For $V$, we now describe an explicit model for the structure map (\ref{orthstructuremap})
$$
s_V: S^1\wedge F^{(1)}(V)\longrightarrow F^{(1)}(V\oplus \R),
$$
where $F^{(1)}(V):=\hofib(F(V)\to F(V\oplus \underline{\R}))$; here $\underline{\R}$ simply stands for a copy of $\R$ that we underline to distinguish it from the one appearing in the codomain of $s_V$. Our model for the homotopy fibre of a map $X\to Y$ of pointed spaces is the standard one, ie, the subspace $\{(x,\gamma)\in X\times Y^{[0,1]}:\gamma(0)=x, \ \gamma(1)=*\}$.

The evident commutative diagram in $\mathcal{J}_0$
\[
\begin{tikzcd}[row sep = 40pt, column sep = 80pt]
    V\dar["\shortstack{\(v\)\\\rotatebox{-90}{\(\mapsto\)}\\\((v,0)\)}"']\rar["{v\mapsto(v,0)}"] &V\oplus \R\dar["\shortstack{\((v,t)\)\\\rotatebox{-90}{\(\mapsto\)}\\\((v,t,0)\)}"]\\
    V\oplus \underline{\R}\rar["{(v,\underline{t})\mapsto(v,0,\underline{t})}"'] & V\oplus\R\oplus\underline{\R}
\end{tikzcd}
\]
induces a commutative diagram of based spaces
$$
\begin{tikzcd}
    F^{(1)}(V)\dar["i"]\rar[dashed]\ar[dr, "\alpha"] & F^{(1)}(V\oplus \R)\dar\\
    F(V)\dar["\beta"]\rar["h"] & F(V\oplus\R)\dar["\beta'"]\\
    F(V\oplus \underline{\R})\rar["h'"] & F(V\oplus\R\oplus \underline{\R}).
\end{tikzcd}
$$
Our task is to define a loop of dashed arrows (based at the constant map)---we will try to do so by providing a loop of null-homotopies $\beta'\alpha\sim *$ (we will fail, but just slightly). 

Note that, since $\beta i$ is null-homotopic via $\widetilde{H}^{(t)}(x,\gamma)=\gamma(t)$, we get a null-homotopy $H_0:= h'\widetilde{H}:\beta'\alpha \sim *$. For each $\theta \in [-\tfrac{\pi}{2},+\tfrac{\pi}{2}]$, let $\theta_*$ denote the automorphism of $F(V\oplus\R\oplus\underline{\R})$ induced by the rotation of the plane $0\oplus \R\oplus\underline{\R}$ with angle $\theta$. Then note that by functoriality of $F$, we have that $\theta_* \beta'h=\beta'h$. Thus, the maps $H_\theta:=\theta_* H_0: F^{(1)}(V)\times I\to F(V\oplus\R\oplus\underline{\R})$ provide a path of null-homotopies $\beta'\alpha\sim *$.

As foreshadowed, the null-homotopies $H_{\tfrac{-\pi}{2}}:=(\tfrac{-\pi}{2})_*h'\widetilde{H}$ and $H_{\tfrac{\pi}{2}}:=(\tfrac{\pi}{2})_*h'\widetilde{H}$ are distinct (as one can check). The crucial point then is that both are of the form $\beta' G$ for some null-homotopy $G: \alpha\sim *$: indeed, if $\phi:  F(V\oplus\underline{\R})\cong F(V\oplus \R)$ denotes the map induced by identifying $\underline{\R}$ with $\R$, then
$$
(\tfrac{-\pi}{2})_*h'=\beta'(-1_\R)_*\phi, \qquad (\tfrac{\pi}{2})_*h'=\beta'\phi,
$$
where $(-1_\R)_*$ is the automorphism of $F(V\oplus \R)$ induced by $(v,t)\mapsto (v,-t)$. By noting that $\phi\beta=h$ and that $(-1_\R)_* h=h$, one easily verifies that $G_{\tfrac{-\pi}{2}}:=(-1_\R)\phi \widetilde{H}$ and $G_{\tfrac{\pi}{2}}:=\phi \widetilde{H}$ indeed provide the desired null-homotopies $\alpha\sim *$.

Finally, we note that the null-homotopies $G_{\tfrac{\pm\pi}{2}}$ give rise to canonical null-homotopies of the maps $\sigma_{\tfrac{\pm\pi}{2}}: F^{(1)}(V)\to F^{(1)}(V\oplus \R)$ induced by $H_{\tfrac{\pm\pi}{2}}$. All together, we obtain a loop of maps $F^{(1)}(V)\to F^{(1)}(V\oplus \R)$ that is adjoint to the structure map $\sigma_V$ of (\ref{orthstructuremap}).

\subsection{The $O(1)$-action}\label{ThetaO1appendix}
We have defined the first derivative spectrum $\Theta F^{(1)}$ of $F$ to be the (naïve) $O(1)$-spectrum with $n$-th $O(1)$-space $\Theta F_n^{(1)}= F^{(1)}(\R^n)$ and with structure maps as defined just above. Now
$$
F^{(1)}(V)=\hofib(F(V)\longrightarrow F(V\oplus \underline{\R}))
$$
is an $O(1)$-space by declaring $-1 \in O(1)$ to act on $V$ and $V\oplus \underline{\R}$ by $-1$ on \textit{all} coordinates. A straightforward verification shows that, under this convention, the map
$$
s_V: S^1\wedge F^{(1)}(V)\longrightarrow F^{(1)}(V\oplus \R)
$$
is indeed $O(1)$-equivariant, where $O(1)$ acts trivially on $S^1$. The key point in this verification is that if $R$ is a $(2\times 2)$-matrix (eg a rotation of the plane), then $\big(\begin{smallmatrix}
  -1 & 0\\
  0 & -1
\end{smallmatrix}\big)R$ and $R\big(\begin{smallmatrix}
  -1 & 0\\
  0 & 1
\end{smallmatrix}\big)$ agree on the subspace $\R\oplus 0\subset \R\oplus\underline{\R}$ even though maybe not equal.

In \cite[Proposition 3.1]{WeissOrthCalc}, $O(1)$ instead acts on $F^{(1)}(V)=\hofib(F(V)\to F(V\oplus \R))$ by declaring the action of $-1\in O(1)$ on $V$ to be trivial and by $-1$ on the $\R$-summand of $V\oplus \R$. If we write $\underline{F}^{(1)}(V)$ for this $O(1)$-space, then the maps 
\begin{equation}\label{underlineFstabmap}
s_V: S^{\sigma}\wedge \underline{F}^{(1)}(V)\longrightarrow \underline{F}^{(1)}(V\oplus \R) \quad\text{and}\quad \underline{F}^{(1)}(V)\to F(V)
\end{equation}
are $O(1)$-equivariant, where $\sigma$ stands for the ($1$-dimensional) sign $O(1)$-representation and $S^\sigma$ for its associated representation sphere. The corresponding (sequential) spectrum, call it $\underline{\Theta}F^{(1)}$, is \textit{not} a naïve $O(1)$-spectrum in the usual sense anymore, as $O(1)$ acts non-trivially on the suspension coordinates. To solve this issue, Weiss introduces in \cite[p. 17]{WeissOrthCalc} the $O(1)$-spectrum $\Theta^\#F^{(1)}$ with $n$-th $O(1)$-space
\begin{equation}\label{Thetasharpdefn}
\Theta^\#F^{(1)}_n:=\Omega^{\infty\cdot \sigma}(S^n\wedge \underline\Theta F^{(1)})=\underset{k}{\hocolim}\ \Omega^{k\cdot\sigma}(S^n\wedge \underline{F}^{(1)}(\R^k)),
\end{equation}
where $\Omega^{k\cdot\sigma}(-):=\mathrm{Map}_*(S^{k\cdot\sigma},-)$ and $O(1)$ acts by conjugation on this mapping space. 

In order to relate the $O(1)$-spectra $\Theta F^{(1)}$ and $\Theta^\#F^{(1)}$,  we observe that $F$ can be naturally upgraded to a functor $\underline{F}: \mathcal J_0^{O(1)}\to \mathsf{Top}_*^{O(1)}$ enriched over $\Top_*$, where $\mathcal J_0^{O(1)}:=\mathrm{Fun}(O(1),\mathcal J_0)$ is regarded as the pointed topological category of inner product finite-dimensional $O(1)$-representations. We likewise define for $V\in \mathcal J_0^{O(1)}$ the $O(1)$-space $\underline{F}^{(1)}(V):=\hofib(\underline{F}(V)\to \underline{F}(V\oplus\sigma))$. Now tensoring such an $O(1)$-representation with the sign representation $\sigma$ gives a self-isomorphism of $\mathcal J_0^{O(1)}$ denoted by $-\cdot\sigma$. One could stabilise $\underline{F}^{(1)}(-)$ with respect to $\R$ as in (\ref{underlineFstabmap}), or with the sign representation $\sigma$, giving rise to maps 
$$
s_{a,b}: S^{a\cdot\sigma+b}\wedge \underline{F}^{(1)}(V)\longrightarrow \underline{F}^{(1)}(V\oplus \R^{a,b}), \quad a,b\geq 0,\quad V\in \mathcal J_0^{O(1)},
$$
where $S^{a\cdot\sigma+b}:=S^{a\cdot\sigma}\wedge S^b$ and $\R^{a,b}:=\R^{a}\oplus b\cdot\sigma$. We then obtain a zig-zag of maps of $O(1)$-spectra
\begin{equation}\label{Thetaheartzigzag}
\begin{tikzcd}[row sep = 20pt, column sep = -45pt]
\Theta^{\#}F^{(1)}:=\underset{a\geqs0}{\mathrm{hocolim}}\ \s^{-a\cdot\sigma} \wedge \underline{F}^{(1)}(\R^{a})\ar[hook, rd,"b=0", "\sim"'] & &\ar[hook', ld,"\sim","a=0"'] \underset{b\geqs 0}{\mathrm{hocolim}}\ \s^{-b} \wedge \underline{F}^{(1)}(b\cdot \sigma)=:\Theta F^{(1)}\\
&\underset{a,b\geqs0}{\mathrm{hocolim}}\ \s^{-a\cdot\sigma-b} \wedge \underline{F}^{(1)}(\R^{a,b}),&
\end{tikzcd}
\end{equation}
where the maps in the colimit of the middle spectrum are induced by $s_{1,0}$ and $s_{0,1}$. Non-equivariantly, both of the maps in the zig-zag are equivalences by Fubini's theorem. This establishes the desired natural $O(1)$-equivariant equivalence\footnote{In the sense of Borel (see Notation \ref{equivariantnotn}($i$)).} $\Theta F^{(1)}\simeq \Theta^\#F^{(1)}$.

\begin{conv}
    In the body of the paper, $F^{(1)}(V)$ stands for $\underline{F}^{(1)}(V\cdot\sigma):=\hofib(\underline{F}(V\cdot\sigma)\to \underline{F}((V\oplus\R)\cdot\sigma))$ in the notation of this section, unless we explictly say otherwise. This way (\ref{orthstructuremap}) is $O(1)$-equivariant.
\end{conv}

\section{Bounded geometry}\label{boundedappendix}
Throughout, $M^d$ denotes a smooth compact $d$-manifold (possibly with boundary).

\subsection{Models for bounded diffeomorphisms and embeddings}\label{boundeddifftopmodelsection} Let $N$ be a (possibly non-compact) smooth manifold and fix some smooth embedding $\iota: M\xhookrightarrow{}N$. For $V\in \mathcal J_0$, we will write $\emb_\partial(M\times V,N\times V)$ for the space of smooth embeddings of $M$ into $N$ that agree with $\iota\times \Id_V$ on some neighbourhood of the boundary $\partial M\times V$, endowed with the Whitney weak $C^{\infty}$-topology. Following Definition \ref{topboundeddefn}, the \textit{space of bounded embeddings} of $M\times V$ into $N\times V$ relative to $\partial M\times V$ is the subspace of $[0,+\infty)\times \emb_\partial(M\times V,N\times V)$ given by
$$
\emb_\partial^b(M\times V,N\times V):=\left\{
(t,\varphi)\in [0,+\infty)\times \emb_\partial(M\times V,N\times V): \text{$\varphi$ is $t$-bounded}
\right\}.
$$
Define similarly its simplicial version $\emb_\partial^b(M\times V,N\times V)_\bullet$ as in Definition \ref{simplicialboundeddefn}. In this section we prove

\begin{proposition}
\label{myAmazingTheorem}
There is a zig-zag of weak equivalences of semi-simplicial group-like monoids
    $$
    \begin{tikzcd}
       \diff^b_\partial(M\times V)_\bullet&\lar["\sim"']\cdot \rar["\sim"]&\Sing_\bullet(\diff_\partial^b(M\times V)).
    \end{tikzcd}
    $$
    Similarly, there is a zig-zag of weak equivalences of semi-simplicial sets
    $$
    \begin{tikzcd}
       \emb^b_\partial(M\times V, N\times V)_\bullet&\lar["\sim"']\cdot \rar["\sim"]&\Sing_\bullet(\emb_\partial^b(M\times V, N\times V)).
    \end{tikzcd}
    $$
\end{proposition}

We will only deal with the first part of the statement, as the proof for the embedding case is completely analogous. Let us first introduce some notation. Given a topological space $X$, let $\Sing^{\mathsf{col}}_\bullet(X)$ be the sub-simplicial set of $\Sing_\bullet(X)$ consisting of those singular simplices that satisfy the $\epsilon$-collaring condition of Section \ref{boundedsection} for some $0<\epsilon<1/2$. Denote by $\Sing^{\mathsf{col}, b}_\bullet(\diff_\partial(M\times V))$ the sub-simplicial group of $\Sing_\bullet^{\mathsf{col}}(\diff_\partial(M\times V))$ consisting of those $j$-simplices which are adjoint to a bounded map $\Delta^j\times M\times V\to M\times V$. Then, there is a zig-zag of maps of simplicial group-like monoids 
\begin{equation}\label{zigzagequation}
\begin{tikzcd}[scale cd = 0.99, column sep = 20pt]
    \diff^b_\partial(M\times V)_\bullet\rar[hook, "\circled{1}"] & \Sing^{\mathsf{col}, b}_\bullet(\diff_\partial(M\times V)) &\lar["\circled{2}"']\Sing^{\mathsf{col}}_\bullet(\diff^b_\partial(M\times V))\rar[hook, "\circled{3}"] 
    &\Sing_\bullet(\diff^b_\partial(M\times V)),
\end{tikzcd}
\end{equation}
where the map $\circled{2}$ forgets the explicit bounding constant of a simplex. We will show that all the maps in (\ref{zigzagequation}) are weak equivalences. We start with $\circled{3}$.

\begin{lem}
    The inclusion $i:\Sing^{\mathsf{col}}_\bullet(X)\xhookrightarrow{}\Sing_\bullet(X)$ is a weak equivalence for every topological space $X$.
\end{lem}
\begin{proof}
We show that the relative homotopy groups $\pi_j(\Sing_\bullet(X),\Sing^{\mathsf{col}}_\bullet(X))$ vanish for all $j\geq 0$. Indeed, a homotopy class $x\in \pi_j(\Sing_\bullet(X),\Sing^{\mathsf{col}}_\bullet(X))$ corresponds, by the Yoneda Lemma, to a singular $j$-simplex $g: \Delta^j\to X$ which satisfies the $\epsilon$-collaring condition for all faces $\sigma\subset \partial\Delta^j$ and some $\epsilon>0$. Now fix some identification $\Delta^j\cong\Delta^j\cup_{\partial\Delta^j}(\partial\Delta^j\times[0,\epsilon])$, and consider the singular $j$-simplex
$$
\overline g=g\cup(g\mid_{\partial_{\Delta^j}}\circ\operatorname{proj}_{\partial\Delta^j}):\Delta^j\cong\Delta^j\cup_{\partial\Delta^j}(\partial\Delta^j\times[0,\epsilon])\longrightarrow X.
$$
By construction $\overline g$ now satisfies the $\delta$-collaring conditions for all faces $\sigma\subset\Delta^j$ and some $0<\delta\leq\epsilon$, so the corresponding relative homotopy class $\overline{x}$ is trivial. But clearly $g$ and $\overline{g}$ are homotopic relative to the boundary by shrinking the added collar, and hence $x=\overline{x}=0$ in $\pi_j(\Sing_\bullet(X),\Sing^{\mathsf{col}}_\bullet(X))$, as claimed.
\end{proof}
\begin{rem}
    The inclusion $i: \Sing_\bullet^{\mathsf{col}}(X)\xhookrightarrow{}\Sing_\bullet(X)$ is in fact a simplicial homotopy equivalence; a homotopy inverse is constructed by induction on the skeleta of $\Delta^j$. We will not need this though.
\end{rem}

\begin{lem}
    The map $\circled{$2$}:\Sing^{\mathsf{col}}_\bullet(\diff^b_\partial(M\times V))\to \Sing^{\mathsf{col}, b}_\bullet(\diff_\partial(M\times V))$ of (\ref{zigzagequation}) is a weak equivalence.
\end{lem}
\begin{proof}
    We again show that the relative homotopy groups $\pi_j(\circled{2})$ vanish for all $j\geq 0$. Such a homotopy $x$ class can be represented, for some $\epsilon>0$, by an $\epsilon$-collared singular $j$-simplex $g:\Delta^j\to \diff_\partial(M\times V)$, adjoint to a map $g^\vee: \Delta^j\times M\times V\to M\times V$ bounded by some $K\geq 0$, together with a continuous $\epsilon$-collared map $r:\partial\Delta^j\to [0,\infty)$ such that $g(s)$ is $r(s)$-bounded for all $s\in \partial\Delta^j$. To show that $x$ is trivial, we need to extend $r$ to a continuous $\delta$-collared map $R: \Delta^j\to [0,\infty)$, for some $0<\delta\leq \epsilon$, such that $g(s)$ is $R(s)$-bounded for all $s\in \Delta^j$. Fix some identification $\Delta^j\cong (\partial\Delta^j\times[0,\epsilon])\cup_{\partial\Delta^j\times\{\epsilon\}}\Delta^j$; then $R\mid_{\partial\Delta^j\times[0,\epsilon/2]}\equiv r\circ\operatorname{proj}_{\partial\Delta^j}$ whilst $R\mid_{\partial\Delta^j\times[\epsilon/2,\epsilon]}$ is a linear interpolation along $[\epsilon/2,\epsilon]$ between $r$ and the constant map $c_{K}:\partial\Delta^j\times\{\epsilon\}\to [0,\infty)$ with value $K\geq 0$. Finally set $R$ to be constant of value $K$ in the inner $\Delta^j\subset (\partial\Delta^j\times[0,\epsilon])\cup_{\partial\Delta^j\times\{\epsilon\}}\Delta^j$. Then $R$ is as required, and hence the relative homotopy class $x\in \pi_j(\circled{2})$ is trivial.
\end{proof}

\begin{rem}\label{TOPsimplicialrem}
    For $CAT=\mathrm{Top}$, the map $\circled{1}$ of (\ref{zigzagequation}) is an equality and thus, at this point, Proposition \ref{myAmazingTheorem} is established in the topological case. 
\end{rem}

\begin{proof}[Proof of Proposition \ref{myAmazingTheorem}]
It remains to show that $\circled{1}$ is a weak equivalence, ie, that the relative homotopy groups $\pi_k(\circled{1})$ vanish for all $k\geq 0$. This is clear for $k=0$ by definition. Such a homotopy class in $\pi_k(\circled{1})$ is represented by a bounded homeomorphism
$$
\overline{g}=(\mathrm{proj}_{\Delta^k},g): \Delta^k\times M\times V\longrightarrow \Delta^k\times M\times V
$$
which is collared in the simplex direction and such that $g\mid_{\partial \Delta^k\times M\times V}$ is smooth. Therefore $g$ is smooth on (a neighbourhood of) $\partial \Delta^k\times M\times V$. We need to smooth $g$ outside of such neighbourhood in the $\Delta^k$-direction and preserving boundedness. For $r\in \Delta^k$, we will write $g_r\in \diff_\partial(M\times V)$ for $g\mid_{\{r\}\times M\times V}$.

Standard smoothing techniques \cite[Section 4]{Munkres}  (see also \cite[Proposition 6.4.2]{Kupers} or \cite[Proposition 1]{LurieSimplicialDiffGroups}) can be used to prove the following: given nested compact subsets $L\subset K\subset \Delta^k\times M\times V$ with $L\subset \mathrm{int}\ K$ and any arbitrarily small $\epsilon>0$, there exists a homotopy $H:I\times \Delta^k\times M\times V\to M\times V$ from $g$ to some map $g': \Delta^k\times M\times V\to M\times V$ satisfying that:
\begin{itemize}
    \item[$(i)$] $H$ remains fixed on $\Delta^k\times M\times V-\mathrm{int}\ K$. In particular $g$ and $g'$ agree there.

    \item[$(ii)$] $g'$ is smooth on $L$. Moreover if $g$ was already smooth on some (open neighbourhood of a) closed subset $\partial \Delta^k\times M\times V\subset F\subset \Delta^k\times M\times V$, the homotopy $H$ remains fixed on $F$. 

    \item[$(iii)$] For each $t\in I$ and $r\in \Delta^k$, the map $H^t_r=H\mid_{\{t\}\times\{r\}\times M\times V}: M\times V\to M\times V$ is smooth.

    \item[$(iv)$] $H^t$ remains arbitrarily close to $g$ for all $t\in I$. Consequently, if $g$ is bounded by some $C\geq 0$, then for every $(r,t)\in I\times \Delta^k$ the map $H_r^t: M\times V\to M\times V$ is bounded by $C+\epsilon$, and is a diffeomorphism (as diffeomorphisms of compact manifolds are open in the space of smooth self-maps).
\end{itemize}
With this in mind, we construct a homotopy in $\Sing^{\mathsf{col}, b}_\bullet(\diff_\partial(M\times V))$ from the $k$-simplex $\overline{g}=(\mathrm{proj}_{\Delta^k},g)$ to some $\overline{h}\in \diff_\partial^b(M\times V)_k$ (relative to $\partial\Delta^k\times M\times V$), as follows: without loss of generality assume $V=\R^n$. Also for $v\in \R^n$ and $\delta>0$, let $C_\delta(v)\subset  \R^n$ denote the cube of side length $2\delta$ and centered at $v$ (ie $C_\delta(v):=v+[-\delta,\delta]^{n}$). Fix an $\epsilon>0$ (eg $\epsilon=1$). Then for each $v\in 3\Z^n\subset \R^n$, choose a homotopy as above starting from $g$ with $(K,L)=(C_1(v), C_{2/3}(v))$, and perform all of these at the same time\footnote{This can be done as, by condition $(i)$, the supports of such homotopies are disjoint by construction.} to obtain some $g': \Delta^k\times M\times V\to M\times V$. Now apply the same process to $g'$ on $(K,L)=(C_1(v), C_{2/3}(v))$ for each $v=(v_1,\dots, v_n)\in \Z^n$ with $v_1\equiv 1\mod 3$ and $v_i\equiv 0\mod 3$ for $2\leq i\leq n$, keeping in mind that, by condition $(ii)$ above, the homotopies keep fixed the parts that have been smoothed in the previous step. Continue this process in a similar fashion. After $3^n$ steps, we will obtain a smooth $(C+3^n\cdot\epsilon)$-bounded map $h: \Delta^k\times M\times V\to M\times V$ such that $\overline h:=(\mathrm{proj}_{\Delta^k},h)$ represents the required $k$-simplex of $\diff_\partial^b(M\times V)_\bullet$. This means that the relative homotopy class $[\overline{g}]\in\pi_k(\circled{1})$ is trivial, as was to be shown. 
\end{proof}

\subsection{A moduli space model for classifying spaces of bounded diffeomorphism groups}\label{bdiffbmodelsection}

Fix an embedding $\iota: M\xhookrightarrow{}\R^m\subset \R^\infty$. Recall that the classifying space $B\diff_\partial(M)$ of the diffeomorphism group of $M$ admits a model as the moduli space of all $d$-manifolds $N^d\subset \R^\infty$ with $\partial N=\partial M$ which are diffeomorphic to $M$ relative to the boundary. In this section we give an analogous description of the classifying space $B\diff^b_\partial(M\times V)$, for any real finite-dimensional inner product vector space $V\in \mathcal J_0$.

\begin{prop}\label{Bdiffbmodelprop}
    Set $\emb_\partial^b(M\times V, \R^\infty\times V):=\operatorname{colim}_{n}\ \emb_\partial^b(M\times V, \R^n\times V)$, and let $\diff^b_\partial(M\times V)$ act on it  by precomposition. Then there is an equivalence
    $$
B\diff_\partial^b(M\times V)\simeq \emb^b_\partial(M\times V,\R^\infty\times V)/\diff^b_\partial(M\times V).
    $$
\end{prop}
In other words, $B\diff^b_\partial(M\times V)$ is (equivalent to) the moduli space of all submanifolds in $\R^\infty\times V$ with boundary $\partial M\times V$ which are diffeomorphic to $M\times V$ boundedly in $V$ and relative to $\partial M\times V$.

\begin{proof}
    By Proposition \ref{myAmazingTheorem}, we have that $B\diff_\partial^b(M\times V)\simeq B|\diff_\partial^b(M\times V)_\bullet|\simeq |B\diff_\partial^b(M\times V)_\bullet|$ and 
    $$
\frac{\emb^b_\partial(M\times V,\R^\infty\times V)}{\diff^b_\partial(M\times V)}\simeq \frac{|\emb^b_\partial(M\times V,\R^\infty\times V)_\bullet|}{|\diff^b_\partial(M\times V)_\bullet|}\simeq \left|\frac{\emb^b_\partial(M\times V,\R^\infty\times V)_\bullet}{\diff^b_\partial(M\times V)_\bullet}\right|.
    $$
    As the simplicial action of $\diff^b_\partial(M\times V)_\bullet$ on $\emb^b_\partial(M\times V,\R^\infty\times V)_\bullet$ is visibly free, we only need to show that $\emb^b_\partial(M\times V,\R^\infty\times V)_\bullet$ is weakly contractible by \cite[Corollary 2.6]{GoerssJardine}. To that end, let 
    $$
    \varphi=(\mathrm{proj}_{\Delta^k}, \varphi_n,\varphi_V): \Delta^k\times M\times V\xhookrightarrow{}\Delta^k\times \R^n\times V, \quad n\geq m,$$
    represent some homotopy class in $\pi_k(\emb^b_\partial(M\times V,\R^\infty\times V)_\bullet)$ for some $k\geq 0$. We will show that $[\varphi]=[\Id_{\Delta^k}\times \iota\times \Id_V]$ by constructing a simplicial map $H: \Delta^1_\bullet \to \emb^b_\partial(M\times V,\R^\infty\times V)_\bullet$ such that, under the Yoneda isomorphism, $\partial_0H=\varphi$ and $\partial_1H=\Id_{\Delta^k}\times \iota\times \Id_V$. The map $H$ will be given by (a modification of) the usual straight-line homotopy between $\varphi$ and $\Id_{\Delta^k}\times \iota\times\Id_V$.

    Let us fix some notation. Pick some open collar $c: [0,1)\times \partial M\xhookrightarrow{} M$ of the boundary of $M$. We can arrange the embedding $\iota: M\xhookrightarrow{}\R^m$ to be such that
    \begin{itemize}
        \item[($i$)] $\iota\equiv (\Id_{[0,1)]}\times i)\circ c^{-1}\mid_{c([0,1)\times \partial M)}$ for some embedding $i: \partial M\xhookrightarrow{}\R^{m-1}$, and

        \item[($ii$)] $\iota(M\setminus c((0,1]\times \partial M))\subset [1,+\infty)\times \R^{m-1}$.
    \end{itemize}
    From now on we will supress $\iota$ and $c$ from the notation, ie, we canonically identify $M$ (resp. $[0,1)\times \partial M$) with its image under $\iota$ (resp. $c$). Choose some increasing smooth function $\alpha: [0,1]\to[0,1]$ for which there exists some $0<\delta$ with $\alpha\mid_{[0,\delta]}\equiv 0$, $\alpha\mid_{[1-\delta,1]}\equiv 1$ and $0<\alpha(t)<1$ for $\delta<t<1-\delta$ (this $\delta$ is required for the collaring condition right before Definition \ref{simplicialboundeddefn}). Now by the collaring condition,
    \begin{equation}\label{sentenceequation}
        \text{there exists some $0<\epsilon<1$ such that $\varphi\equiv \Id_{\Delta^k}\times\iota\times \Id_V$ on $\Delta^k\times[0,\epsilon)\times\partial M\times V$.}
    \end{equation}
Finally, fix some smooth function $\rho: M\to [0,1]$ such that
    $$
\rho\mid_{[0,\epsilon/2]\times \partial M}\equiv 0\quad \text{and}\quad \rho\mid_{M\setminus[0,\epsilon)\times\partial M}\equiv 1.
    $$
Then for $t\in [0,1]$, consider the map 
\begin{align}
H_t: \Delta^k\times M\times V&\longrightarrow \Delta^k\times \R^n\times \R^m\times \R^{|V|}\times V\subset \Delta^k\times \R^\infty\times V,\nonumber\\
\label{Htmapcoordinates}(r,x,v)&\longmapsto \begin{pmatrix}
    r\\
    \alpha(t)\cdot x+(1-\alpha(t))\cdot\varphi_n(r,x,v)\\
    \rho(x)\alpha(t)(1-\alpha(t))\cdot x\\
    \rho(x)\alpha(t)(1-\alpha(t))\cdot v\\
    \alpha(t)\cdot v+(1-\alpha(t))\cdot \varphi_V(t,x,v)
\end{pmatrix}
\end{align}
Here $x\in M\subset \R^m\subset \R^n$ and $\R^{|V|}$ is a Euclidean space of the same dimension as $V$, treated as a copy of $V$.
\begin{cl}
    Let $C\geq 0$ be the bound of $\varphi$ on the $V$-coordinate. Then the map $H_t$ is a $C$-bounded embedding for $t\in [0,1]$. Moreover, $H_t$ agrees with $\Id_{\Delta^k}\times\iota\times\Id_V$ on $\Delta^k\times [0,\epsilon/2]\times \partial M\times V$.
\end{cl}
\begin{proof}[Proof of Claim]
    Indeed $H_t$ is bounded by $C\geq0$ (in the $V$-coordinate) as
    \begin{align*}
\norm{(\alpha(t)\cdot v+(1-\alpha(t))\cdot \varphi_V(t,x,v))-v}_V&=(1-\alpha(t))\cdot\norm{\varphi_V(t,x,v)-v}_V\\
&\leq(1-\alpha(t))\cdot C\leq C. 
    \end{align*}
To see that $H_t$ is an embedding, suppose that $H_t(r,x,v)=H_t(r',x',v')$. Clearly then $r=r'$ by the first coordinate in (\ref{Htmapcoordinates}). Note that $H_t=\varphi$ if $t\leq \delta$ and $H_t=\Id_{\Delta^k}\times\iota\times \Id_{V}$ if $t\geq 1-\delta$. As both are embeddings, we may assume that $\delta<t<1-\delta$ so that $\alpha(t)(1-\alpha(t))\neq 0$. To show that $x=x'$ we consider three cases:
\begin{itemize}
    \item If $x,x'\in[0,\epsilon]\times \partial M$, then by (\ref{sentenceequation}), the equation on the second coordinate of (\ref{Htmapcoordinates}) yields $x=x'$.

    \item If $x,x'\notin [0,\epsilon]\times \partial M$, then $\rho(x)=\rho(x')=1$ and thus the third coordinate equation yields $x=x'$.

    \item  If $x\in [0,\epsilon]\times \partial M$ but $x'\notin [0,\epsilon]\times \partial M$, then the third coordinate equation becomes $\rho(x)\cdot x=x'\in \R^n$. On the first coordinate of $[0,+\infty)\times \R^{n-1}\subset \R^n$, this implies, by items ($i$) and ($ii$) above, that $\rho(x)>1$ which is a contradiction. 
\end{itemize}
    In all cases $x=x'$. Then the equation on the fourth coordinate of (\ref{Htmapcoordinates}) implies that $v=v'$, as required. 

    Finally, the last part of the claim again follows from (\ref{sentenceequation}) and the nature of $\rho$. 
\end{proof}
The family of $C$-bounded embeddings $\{H_t\}_{t\in [0,1]}$ gives rise to the required simplicial map $H$. This finishes the proof of the proposition.
\end{proof}

\section{The \texorpdfstring{$h$}{h}-cobordism stabilisation map}\label{AppendixB}
The (lower) stabilisation map $\Sigma: H(M)\to H(M\times I)$ of \cite[p. 298]{VogellInvolution} is depicted in Figure \ref{lowerstabmap} below.

\begin{figure}[h]
    \centering
    \includegraphics[scale=0.65]{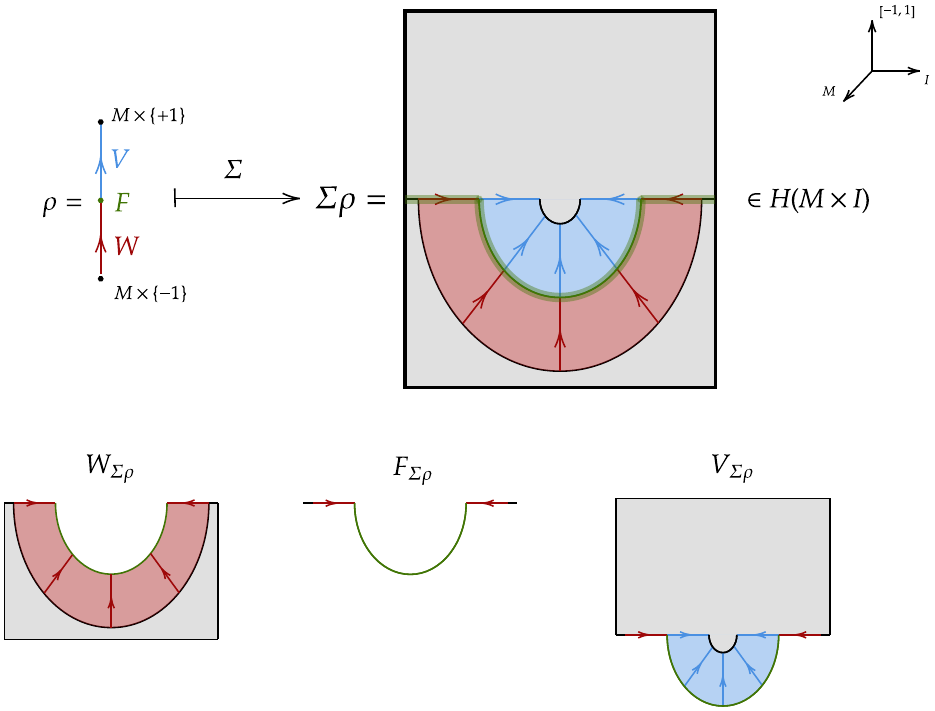}
    \caption{The $h$-cobordism (lower) stabilisation map $\Sigma_\ell=\Sigma: H(M)\to H(M\times I)$, sending $\rho=(W,F,V)$ to $\Sigma\rho:=(W_{\Sigma\rho}, F_{\Sigma\rho}, V_{\Sigma\rho})$. A grey shaded region of shape $S$ represents a manifold of the form $M\times S$.}
    \label{lowerstabmap}
\end{figure}

Recall that the $h$-cobordism space $H(M\times I)$ is an $\mathbb{E}_1$-space under stacking in the $I$-direction, denoted $+_I$. In this section we argue that $\Sigma$ anticommutes with the $h$-cobordism involution $\iota_H$ in the following sense:

\begin{lem}[Lemma \ref{suspensionequivariancelem}($a$)]\label{anticommutativitySigmaprop}
    The map $\iota_H\Sigma+_I\Sigma\iota_H: H(M)\to H(M\times I)$ is null-homotopic, ie it is homotopic to the constant map at the trivial partition $*\in H(M\times I)$.
\end{lem}

\begin{proof}
We describe the null-homotopy in the topological setting; the smooth case is very similar, but one has to be slightly careful with issues regarding corners (which can be overcome by working with the collared version $H_{\mathsf{col}}(M\times I)$ of the $h$-cobordism space). It will be convenient to work with yet another (upgraded) version of the $h$-cobordism space: let $\overline{H}_{\mathsf{col}}(M)_\bullet$ denote the simplicial set in which a $q$-simplex consists of a pair $(\rho,\phi)$ with $\rho:=(W,F,V)\in H_{\mathsf{col}}(M)_q$ and a diffeomorphism $\phi: V\cup_{M\times \Delta^q} W\cong F\times [-1,1]\times \Delta^q$ over $\Delta^q$ which fixes pointwise (a neighbourhood of) $\partial(F\times [-1,1])\times \Delta^q$. There is a Kan fibration
$$
\begin{tikzcd}
\diff_\partial(M\times[-1,1])_\bullet\rar["j"] & \overline{H}_{\mathsf{col}}(M)_\bullet\rar["p"] & H_{\mathsf{col}}(M)_\bullet,
\end{tikzcd}
$$
where $p(\rho,\phi):=\rho$. The inclusion $j$ admits a (left) section up to homotopy $s: \overline{H}_{\mathsf{col}}(M)_\bullet\to \diff_\partial(M\times[-1,1])_\bullet$ given roughly by applying $\phi^{-1}$ on the collar of $F$ in $M\times[-1,1]=W\cup_{F}V$ and then canonically identifying $W\cup_{F}V\cup_M W\cup_F V$ with $M\times[-1,1]$. This yields an equivalence $$(\hspace{1pt}p,s): \overline{H}_{\mathsf{col}}(M)_\bullet\overset{\sim}\longrightarrow H_{\mathsf{col}}(M)_\bullet\times \diff_\partial(M\times[-1,1])_\bullet.$$ 
But now the following diagram commutes up to homotopy:
$$
\begin{tikzcd}[row sep = 14pt]
    \overline{H}_{\mathsf{col}}(M)_\bullet \dar["\vsim"]\ar[rrrrd, "f_1"]&&&&\\
    H_{\mathsf{col}}(M)_\bullet\times \diff_\partial(M\times[-1,1])_\bullet\ar[rrrr, "f_2"]\dar[shift right = 4pt, two heads, "\mathrm{pr}_{1}"'] &&&&H(M\times I)_\bullet\\
    H_{\mathsf{col}}(M)_\bullet\dar["u", "\vsim"']\uar[shift right = 4pt, hook, "i"'] \ar[rrrru, "f_3"]&&&&\\
    H(M)_\bullet\ar[rrrruu, "f_4=\iota_H\Sigma+_I\Sigma\iota_H"'] &&&&
\end{tikzcd}
$$
where $u$ is the map that forgets the collaring data, and all the horizontal maps (strictly) factor through the bottom horizontal map $f_4:=\iota_H\Sigma+_I\Sigma\iota_H$. Therefore, in order to show that $f_4$ is null-homotopic, it suffices to show that $f_1$ is so: indeed this would imply that $f_2$ is null-homotopic. But $f_2=f_3\circ \mathrm{pr}_1$, so $f_3\simeq f_3\circ \mathrm{pr}_1\circ i\simeq *$ too. This in turn would imply that $f_4$ is null-homotopic, as we aim to prove. 

We therefore need to describe a null-homotopy of $f_1$. We will just describe a path (or rather, a $1$-simplex) between $f_4(\rho,\phi)=(\iota_H\Sigma+_I\Sigma\iota_H)(\rho)$, for a fixed $0$-simplex $(\rho=(W,F,V),\phi)\in \overline{H}_{\mathsf{col}}(M)_0$, and the trivial partition $*\in H(M\times I)_0$---an exactly analogous argument yields an actual simplicial null-homotopy. This path is depicted in Figures \ref{nullhomotopy1}, \ref{nullhomotopy2}, \ref{nullhomotopy3} and \ref{nullhomotopy4}. The green shaded regions in each picture represent the $F$-part of a partion, ie, the intersection of the two $h$-cobordisms making up the partition, which is a $(d+1)$-manifold embedded in $M\times I\times[-1,1]$. The partition $\rho=(W,F,V)\in H(M)_0$ is as depicted in Figure \ref{lowerstabmap}.

 Firstly, the path between the partition $P_0=\iota_H\Sigma(\rho)+_I\Sigma\iota_H(\rho)$ and $P_1$ is obtained from rescaling (and slightly shifting inwards). But as $W\cup_F V$ is canonically\footnote{\emph{Canonically} in the sense that it does not depend in any other choice than the one of $\rho$.} identified with $M\times[-1,1]$ as part of the data of $\rho$, the outer bent regions of the form $(W\cup_F V)\times I$ added to $P_1$ in order to obtain $P_2$ are canonically identified with $M\times [-1,1]\times I$, and therefore $P_1=P_2$ in $H(M\times I)$. See Figure \ref{nullhomotopy1}.
 
    \begin{figure}[h]
        \centering
        \includegraphics[scale = 0.75]{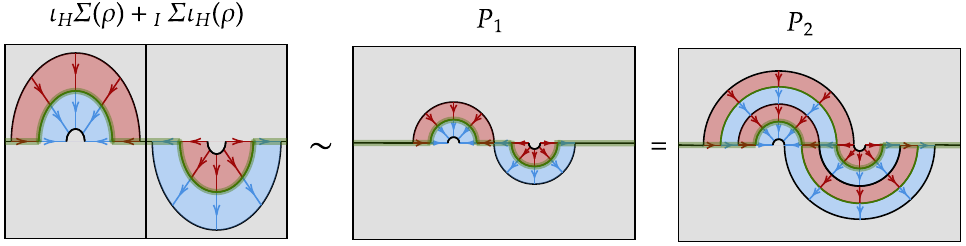}
        \caption{Path in $H(M\times I)$ between $\iota_H\Sigma(\rho)+\Sigma\iota_H(\rho)$ and $P_2$.}
        \label{nullhomotopy1}
    \end{figure}

Unbending and straightening the region of the form $(W\cup_F V\cup_M W\cup_F V)\times I\equiv M\times [-1,1]\times I$ in $P_2$, we get the path to the partition $P_3$ of Figure \ref{nullhomotopy2} (this step is not strictly necessary, but convenient for depiction).

     \begin{figure}[h]
        \centering
        \includegraphics[scale = 0.75]{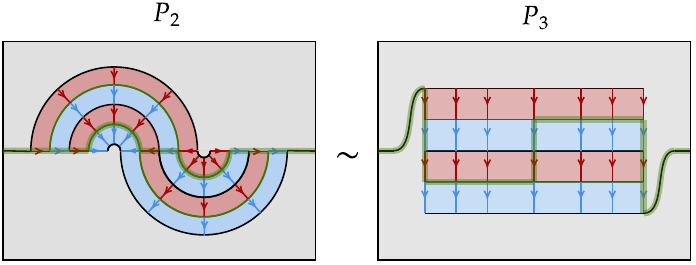}
        \caption{Path in $H(M\times I)$ between $P_2$ and $P_3$.}
        \label{nullhomotopy2}
    \end{figure}

    We now use the diffeomorphism $\phi: V\cup_M\cup W\cong F\times[-1,1]$ (rel. $\partial(F\times[-1,1])$) to carry out the path depicted in the lower part of Figure \ref{nullhomotopy3} locally in the circled sub-rectangle of $P_3$. This yields the path of Figure \ref{nullhomotopy3} between $P_3$ and $P_4$.

     \begin{figure}[h]
        \centering
        \includegraphics[scale = 0.75]{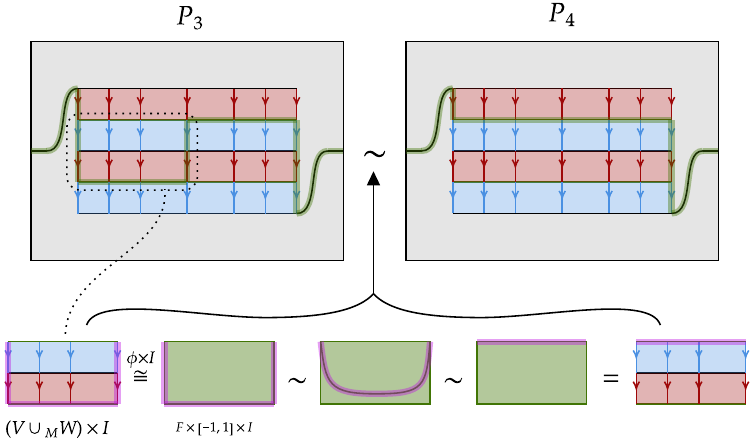}
        \caption{Path in $H(M\times I)$ between $P_3$ and $P_4$. The green shaded region in the lower part of the figure represents $F\times [-1,1]\times I$. The purple shaded region there represents the $F$-part of the partition (which used to be green, but is purple momentarily).}
        \label{nullhomotopy3}
    \end{figure}

    Retracting the region of the form $(W\cup_{F}V\cup_M W\cup_F V)\times I$ in $P_4$ to its midpoint $(W\cup_{F}V\cup_M W\cup_F V)\times \{1/2\}$ yields the path of Figure \ref{nullhomotopy4} between $P_4$ and $P_6$ (passing through $P_5$). But now as $W\cup_{F}V\cup_M W\cup_F V\equiv M\times [-1,3]$, the $F$-part of the partition $P_6$ is of the form $M\times \gamma$, for the $1$-dimensional submanifold $\gamma\subset [-1,1]\times I$ depicted in $P_6$. By straightening $\gamma$ to the submanifold $\{0\}\times I\subset[-1,1]\times I$, we finally get the path of Figure \ref{nullhomotopy4} between $P_6$ and the trivial partition $P_7=*\in H(M\times I)$.

     \begin{figure}[h]
        \centering
        \includegraphics[scale = 0.75]{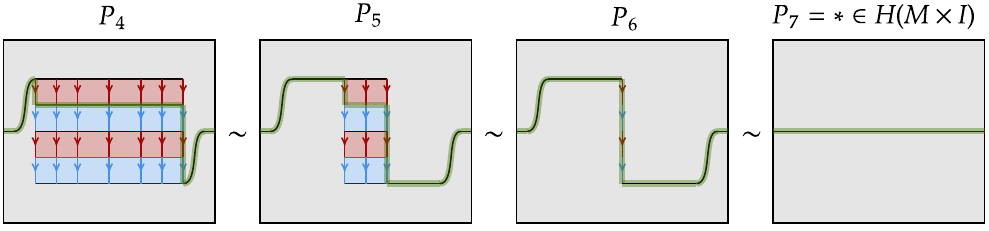}
        \caption{Path in $H(M\times I)$ between $P_4$ and $P_7$.}
        \label{nullhomotopy4}
    \end{figure}
    
    This process depends continuously on $(\rho,\phi)\in \overline{H}_{\mathsf{col}}(M)$ (ie can be set up as a homotopy of simplicial maps $\overline{H}_{\mathsf{col}}(M)_\bullet\to H(M\times I)_\bullet$), and therefore gives the required null-homotopy.
\end{proof}

\begin{rem}
    This result should be compared with \cite[Appendix I, Lemma]{HatcherSSeq} (or \cite[Corollary A7]{BurgLash}), where it is shown that the concordance stabilisation map $\Sigma: C(M)\to C(M\times I)$ anti-commutes (in the same sense of Lemma \ref{suspensionequivariancelem}) with the concordance involution $\iota_C$ of Warning \ref{vogellwarning}.
\end{rem}

\bibliography{bibliographymcg.bib} 

@article{WWI,
  doi = {10.1007/bf00533787},
  url = {https://doi.org/10.1007/bf00533787},
  year = {1988},
  month = nov,
  publisher = {Portico},
  volume = {1},
  number = {6},
  pages = {575--626},
  author = {M. Weiss and B. Williams},
  title = {Automorphisms of manifolds and algebraic {$K$}-theory: {I}},
  journal = {{$K$}-Theory}
}

@incollection {WaldManifoldApproach,
    AUTHOR = {Waldhausen, F.},
     TITLE = {Algebraic {$K$}-theory of spaces, a manifold approach},
 BOOKTITLE = {Current trends in algebraic topology, {P}art 1 ({L}ondon,
              {O}nt., 1981)},
    SERIES = {CMS Conf. Proc.},
    VOLUME = {2},
     PAGES = {141--184},
 PUBLISHER = {Amer. Math. Soc., Providence, RI},
      YEAR = {1982},
      ISBN = {0-8218-6001-1},
   MRCLASS = {18F25 (57R52)},
  MRNUMBER = {686115},
}

@article{CohenCarlssonTCFLS,
  author = {R. L. Cohen and G. E. Carlsson},
  journal = {Comment. Math. Helv.},
  pages = {423--449},
  title = {The cyclic groups and the free loop space},
  url = {http://eudml.org/doc/140092},
  volume = {62},
  year = {1987}
}

@article{AvsTC,
  doi = {10.1215/s0012-7094-96-08417-3},
  url = {https://doi.org/10.1215/s0012-7094-96-08417-3},
  year = {1996},
  month = sep,
  publisher = {Duke Univ. Press},
  pages = {541--566},
  volume = {84},
  number = {3},
  author = {M. Bökstedt and G. Carlsson and R. Cohen and T. Goodwillie, W. C. Hsiang and I. Madsen},
  title = {On the algebraic {$K$}-theory of simply connected spaces},
  journal = {Duke Math. J.}
}

@article{ConcEmbStablerange,
  doi = {10.1017/prm.2023.17},
  url = {https://doi.org/10.1017/prm.2023.17},
  year = {2023},
  month = mar,
  publisher = {Cambridge University Press ({CUP})},
  pages = {1--36},
  author = {T. Goodwillie and M. Krannich and A. Kupers},
  title = {Stability of concordance embeddings},
  journal = {Proc. R. Soc. Edinb., Sect. A: Math.}
}

@misc{LurieSimplicialDiffGroups,
  doi = {},
  url = {https://www.math.ias.edu/~lurie/937notes/937Lecture6.pdf},
  author = {Lurie,  J.},
  title = {Diffeomorphisms and {$PL$}-homeomorphisms},
  publisher = {},
  year = {2009},
  copyright = {},
note = {Lecture notes of ``Topics in Geometric Topology'', Lecture 6. Available at: \href{https://www.math.ias.edu/~lurie/937notes/937Lecture6.pdf}{https://www.math.ias.edu/~lurie/937notes/937Lecture6.pdf}}
}

@article{InvolutionAtheoryBustamanteFarrell,
  doi = {10.1090/tran/8135},
  url = {https://doi.org/10.1090/tran/8135},
  year = {2020},
  month = jul,
  publisher = {American Mathematical Society ({AMS})},
  volume = {373},
  number = {10},
  pages = {7225--7252},
  author = {Bustamante, M. and Farrell, F. T. and Jiang, Yi.},
  title = {Involution on pseudoisotopy spaces and the space of nonnegatively curved metrics},
  journal = {Trans. Am. Math. Soc.}
}

@article{GoodwillierelativeAtheory,
  doi = {10.2307/1971283},
  url = {https://doi.org/10.2307/1971283},
  year = {1986},
  month = sep,
  publisher = {{JSTOR}},
  volume = {124},
  number = {2},
  pages = {347--402},
  author = {Goodwillie, T. G.},
  title = {Relative Algebraic {$K$}-Theory and Cyclic Homology},
  journal = {Ann. of Math.}
}

@article{Blumberg2019,
  doi = {10.2140/gt.2019.23.101},
  url = {https://doi.org/10.2140/gt.2019.23.101},
  year = {2019},
  month = mar,
  publisher = {Mathematical Sciences Publishers},
  volume = {23},
  number = {1},
  pages = {101--134},
  author = {A. Blumberg and M. Mandell},
  title = {The homotopy groups of the algebraic {$K$}-theory
		of the sphere spectrum},
  journal = {Geom. Topol.}
}

@article{Palais,
  doi = {10.1007/bf02565942},
  url = {https://doi.org/10.1007/bf02565942},
  year = {1960},
  month = dec,
  publisher = {European Mathematical Society - {EMS} - Publishing House {GmbH}},
  volume = {34},
  number = {1},
  pages = {305--312},
  author = {Palais, R. S.},
  title = {Local triviality of the restriction map for embeddings},
  journal = {Comment. Math. Helv.}
}

@article {Lima1963/64,
    AUTHOR = {Lima, E. L.},
     TITLE = {On the local triviality of the restriction map for embeddings},
   JOURNAL = {Comment. Math. Helv.},
  FJOURNAL = {Commentarii Mathematici Helvetici},
    VOLUME = {38},
      YEAR = {1964},
     PAGES = {163--164},
      ISSN = {0010-2571,1420-8946},
   MRCLASS = {57.20},
  MRNUMBER = {161343},
MRREVIEWER = {R.\ S.\ Palais},
       DOI = {10.1007/BF02566913},
       URL = {https://doi.org/10.1007/BF02566913},
}

@book {BurgLashRoth,
    AUTHOR = {Burghelea, D. and Lashof, R. and Rothenberg, M.},
     TITLE = {Groups of automorphisms of manifolds},
    SERIES = {Lecture Notes in Mathematics},
    VOLUME = {Vol. 473},
      NOTE = {With an appendix (``The topological category'') by E.
              Pedersen},
 PUBLISHER = {Springer-Verlag, Berlin-New York},
      YEAR = {1975},
     PAGES = {vii+156},
   MRCLASS = {57E05 (58D10)},
  MRNUMBER = {380841},
MRREVIEWER = {E.\ C.\ Turner},
}

@article {BurgLash,
    AUTHOR = {Burghelea, D. and Lashof, R.},
     TITLE = {Geometric transfer and the homotopy type of the automorphism
              groups of a manifold},
   JOURNAL = {Trans. Amer. Math. Soc.},
  FJOURNAL = {Transactions of the American Mathematical Society},
    VOLUME = {269},
      YEAR = {1982},
    NUMBER = {1},
     PAGES = {1--38},
      ISSN = {0002-9947,1088-6850},
   MRCLASS = {57R65 (20F38 55R10 58D05)},
  MRNUMBER = {637027},
MRREVIEWER = {J.\ M.\ Boardman},
       DOI = {10.2307/1998592},
       URL = {https://doi.org/10.2307/1998592},
}

@misc{Kupers,
    author = {Kupers, A.},
    title = {Lectures on Diffeomorphism Groups of Manifolds},
    year = {2019},
    note = {Lecture notes, available at \href{https://people.math.harvard.edu/~kupers/teaching/272x/book.pdf}{https://people.math.harvard.edu/~kupers/teaching/272x/book.pdf}}
}

@book {Ranicki,
    AUTHOR = {Ranicki, A.},
     TITLE = {Algebraic and geometric surgery},
    SERIES = {Oxford Mathematical Monographs},
      NOTE = {Oxford Science Publications},
 PUBLISHER = {The Clarendon Press, Oxford University Press, Oxford},
      YEAR = {2002},
     PAGES = {xii+373},
      ISBN = {0-19-850924-3},
   MRCLASS = {57R65 (19J25 57R67)},
  MRNUMBER = {2061749},
MRREVIEWER = {Daniel\ Ruberman},
       DOI = {10.1093/acprof:oso/9780198509240.001.0001},
       URL = {https://doi.org/10.1093/acprof:oso/9780198509240.001.0001},
}

@article {HaugsengMoritaCategory,
    AUTHOR = {Haugseng, R.},
     TITLE = {The higher {M}orita category of {$\Bbb{E}_n$}-algebras},
   JOURNAL = {Geom. Topol.},
  FJOURNAL = {Geometry \& Topology},
    VOLUME = {21},
      YEAR = {2017},
    NUMBER = {3},
     PAGES = {1631--1730},
      ISSN = {1465-3060,1364-0380},
   MRCLASS = {18D50 (16D20 18D10 55U35)},
  MRNUMBER = {3650080},
MRREVIEWER = {Julia\ Bergner},
       DOI = {10.2140/gt.2017.21.1631},
       URL = {https://doi.org/10.2140/gt.2017.21.1631},
}

@book{WJR,
    author = {Waldhausen, F. and Jahren, B. and Rognes, J.},
    title = {Spaces of {PL} Manifolds and Categories of Simple Maps},
    series = {Annals of Mathematics Studies},
    volume = {186},
    publisher = {Princeton University Press},
    address = {Princeton, NJ},
    year = {2013},
    pages = {vi+184},
    isbn = {978-0-691-15776-4},
    doi = {10.1515/9781400846528},
    url = {https://doi.org/10.1515/9781400846528}
}

@book {GoerssJardine,
    AUTHOR = {Goerss, P. G. and Jardine, J. F.},
     TITLE = {Simplicial homotopy theory},
    SERIES = {Progress in Mathematics},
    VOLUME = {174},
 PUBLISHER = {Birkh\"auser Verlag, Basel},
      YEAR = {1999},
     PAGES = {xvi+510},
      ISBN = {3-7643-6064-X},
   MRCLASS = {55U10 (18G55 55-01 55Pxx)},
  MRNUMBER = {1711612},
MRREVIEWER = {R.\ M.\ Vogt},
       DOI = {10.1007/978-3-0348-8707-6},
       URL = {https://doi.org/10.1007/978-3-0348-8707-6},
}

@article{Atheorysuspension,
  doi = {10.1007/bf00533987},
  url = {https://doi.org/10.1007/bf00533987},
  year = {1987},
  month = jan,
  publisher = {Portico},
  volume = {1},
  number = {1},
  pages = {53--82},
  author = {G. E. Carlsson and R. L. Cohen and T. Goodwillie and W. C. Hsiang},
  title = {The free loop space and the algebraic {$K$}-theory of spaces},
  journal = {{$K$}-Theory}
}

@article{RourkeSandIII,
  doi = {10.2307/1970714},
  url = {https://doi.org/10.2307/1970714},
  year = {1968},
  month = may,
  publisher = {{JSTOR}},
  volume = {87},
  number = {3},
  pages = {431--483},
  author = {C. P. Rourke and B. J. Sanderson},
  title = {Block Bundles: {III}. Homotopy Theory},
  journal = {Ann. of Math.}
}

@article{RourkeSandI,
  doi = {10.2307/1970591},
  url = {https://doi.org/10.2307/1970591},
  year = {1968},
  month = jan,
  publisher = {{JSTOR}},
  volume = {87},
  number = {1},
  pages = {1--28},
  author = {C. P. Rourke and B. J. Sanderson},
  title = {Block Bundles: I},
  journal = {Ann. of Math.}
}

@article{BritoHorelLongKnots, title={Galois symmetries of knot spaces}, volume={157}, DOI={10.1112/S0010437X21007041}, number={5}, journal={Compositio Mathematica}, author={Boavida de Brito, P. and Horel, G.}, year={2021}, pages={997--1021}}

@BOOK{Munkres,
  title     = "Elementary differential topology",
  author    = "Munkres, J. R.",
  publisher = "Princeton University Press",
  volume = {54},
  series    = "Ann. of Math. Stud.",
  month     =  dec,
  year      =  1966,
  address   = "Princeton, NJ"
}

@book {Adem2004,
    AUTHOR = {Adem, A. and Milgram, R. J.},
     TITLE = {Cohomology of finite groups},
    SERIES = {Grundlehren der mathematischen Wissenschaften [Fundamental
              Principles of Mathematical Sciences]},
    VOLUME = {309},
   EDITION = {Second},
 PUBLISHER = {Springer-Verlag, Berlin},
      YEAR = {2004},
     PAGES = {viii+324},
      ISBN = {3-540-20283-8},
   MRCLASS = {20J06 (18G99 55R35 55R40)},
  MRNUMBER = {2035696},
       DOI = {10.1007/978-3-662-06280-7},
       URL = {https://doi.org/10.1007/978-3-662-06280-7},
}

@BOOK{Wall,
  title     = "Surgery on Compact Manifolds",
  author    = "Wall, C. T. C.",
  publisher = "American Mathematical Society",
  series    = "Mathematical Surveys and Monographs",
  edition   =  2,
  month     =  mar,
  year      =  1999,
  address   = "Providence, RI"
}

@misc{DottoMoiPatchkoria,
  doi = {10.48550/ARXIV.2106.04891},
  url = {https://arxiv.org/abs/2106.04891},
  author = {Dotto,  Emanuele and Moi,  Kristian and Patchkoria,  Irakli},
  title = {On the geometric fixed-points of real topological cyclic homology},
  publisher = {arXiv},
  year = {2021},
  copyright = {arXiv.org perpetual,  non-exclusive license}
}

@book{MaySigurdsson,
  doi = {10.1090/surv/132},
  url = {https://doi.org/10.1090/surv/132},
  year = {2006},
  month = dec,
  publisher = {American Mathematical Society},
  author = {J. May and J. Sigurdsson},
  title = {Parametrized Homotopy Theory}
}

@article{ScannellSinha,
  title = {A one-dimensional embedding complex},
  volume = {170},
  ISSN = {0022-4049},
  url = {http://dx.doi.org/10.1016/S0022-4049(01)00078-0},
  DOI = {10.1016/s0022-4049(01)00078-0},
  number = {1},
  journal = {Journal of Pure and Applied Algebra},
  publisher = {Elsevier BV},
  author = {Scannell,  K. P. and Sinha,  D. P.},
  year = {2002},
  month = may,
  pages = {93--107}
}

@article{GromollFiltration,
author = {Gromoll, D.},
journal = {Mathematische Annalen},
language = {ger},
pages = {353--371},
title = {Differenzierbare Strukturen und Metriken positiver Krümmung auf Sphären.},
url = {http://eudml.org/doc/161412},
volume = {164},
year = {1966},
}

@article{PLHudsonIET,
  title = {Extending Piecewise-Linear Isotopies},
  volume = {s3-16},
  ISSN = {0024-6115},
  url = {http://dx.doi.org/10.1112/plms/s3-16.1.651},
  DOI = {10.1112/plms/s3-16.1.651},
  number = {1},
  journal = {Proceedings of the London Mathematical Society},
  publisher = {Wiley},
  author = {Hudson,  J. F. P.},
  year = {1966},
  pages = {651--668}
}

@article{Pedersen1976,
author = {Pedersen, E.K.},
journal = {Inventiones mathematicae},
pages = {255-268},
title = {Topological Concordances},
url = {http://eudml.org/doc/142451},
volume = {38},
year = {1976},
}

@article{Budney2007,
  title = {Little cubes and long knots},
  volume = {46},
  ISSN = {0040-9383},
  url = {http://dx.doi.org/10.1016/j.top.2006.09.001},
  DOI = {10.1016/j.top.2006.09.001},
  number = {1},
  journal = {Topology},
  publisher = {Elsevier BV},
  author = {Budney,  R.},
  year = {2007},
  month = jan,
  pages = {1--27}
}

@book{VogellInvolution,
    author="Vogell, W.",
    title="The involution in the algebraic {$K$}-theory of spaces",
    publisher="Springer",
    volume="1126",
    number="",
    series="Algebraic {$\&$} Geometric Topology: Lecture Notes in Math.",
    edition="",
    year="1985",
    pages="277-317",
    month="",
    note=""
}

@article{HatcherWag,
    author = "Hatcher, A. and Wagoner, J. B.",
    title = "Pseudo-isotopies of compact manifolds",
    journal = "Société Mathématique de France",
    pages = "1--275",
    volume = "",
    number = "6",
    year = "1973"
}

@article{WeissOrthCalc,
    author = "Weiss, M.",
    title = "Orthogonal Calculus",
    journal = "Trans. Amer. Math. Soc.",
    volume = "",
    number = "347",
    pages = "3743-3796",
    year = "1995"
}

@article{Rognesprime2,
  doi = {10.1016/s0040-9383(01)00005-2},
  url = {https://doi.org/10.1016/s0040-9383(01)00005-2},
  year = {2002},
  month = sep,
  publisher = {Elsevier {BV}},
  volume = {41},
  number = {5},
  pages = {873--926},
  author = {J. Rognes},
  title = {Two-primary algebraic {$K$}-theory of pointed spaces},
  journal = {Topology}
}

@article{RognesOddprimes,
author = {J. Rognes},
title = {{The smooth Whitehead spectrum of a point at odd regular primes}},
volume = {7},
journal = {Geometry {$\&$} Topology},
number = {1},
publisher = {MSP},
pages = {155 -- 184},
year = {2003},
doi = {10.2140/gt.2003.7.155},
URL = {https://doi.org/10.2140/gt.2003.7.155}
}

@article{Dundas1997,
  doi = {10.1007/bf02392744},
  url = {https://doi.org/10.1007/bf02392744},
  year = {1997},
  publisher = {International Press of Boston},
  volume = {179},
  number = {2},
  pages = {223--242},
  author = {Dundas, B. I.},
  title = {Relative {$K$}-theory and topological cyclic homology},
  journal = {Acta Mathematica}
}

@article{WatanabeLongKnots,
  doi = {10.2140/agt.2007.7.47},
  url = {https://doi.org/10.2140/agt.2007.7.47},
  year = {2007},
  month = feb,
  publisher = {Mathematical Sciences Publishers},
  volume = {7},
  number = {1},
  pages = {47--92},
  author = {T. Watanabe},
  title = {Configuration space integral for long {$n$}-knots and the {A}lexander polynomial},
  journal = {Algebraic {$\&$} Geometric Topology}
}

@article{BudneyCohen,
  doi = {10.2140/gt.2009.13.99},
  url = {https://doi.org/10.2140/gt.2009.13.99},
  year = {2009},
  month = jan,
  publisher = {Mathematical Sciences Publishers},
  volume = {13},
  number = {1},
  pages = {99--139},
  author = {R. Budney and F. Cohen},
  title = {On the homology of the space of knots},
  journal = {Geometry {$\&$} Topology}
}

@article{DevPSinha2009,
  doi = {10.1353/ajm.0.0061},
  url = {https://doi.org/10.1353/ajm.0.0061},
  year = {2009},
  publisher = {Project Muse},
  volume = {131},
  number = {4},
  pages = {945--980},
  author = {Sinha, D. P.},
  title = {The topology of spaces of knots: cosimplicial models},
  journal = {American Journal of Mathematics}
}

@article{VolicLongKnots,
  doi = {10.1112/s0010437x05001648},
  url = {https://doi.org/10.1112/s0010437x05001648},
  year = {2006},
  month = jan,
  publisher = {Wiley},
  volume = {142},
  number = {01},
  pages = {222--250},
  author = {I. Volić},
  title = {Finite type knot invariants and the calculus of functors},
  journal = {Compositio Mathematica}
}

@article{AroneTurchinHomologyLK,
  doi = {10.2140/gt.2014.18.1261},
  url = {https://doi.org/10.2140/gt.2014.18.1261},
  year = {2014},
  month = jul,
  publisher = {Mathematical Sciences Publishers},
  volume = {18},
  number = {3},
  pages = {1261--1322},
  author = {G. Arone and V. Turchin},
  title = {On the rational homology of high-dimensional analogues of spaces of long knots},
  journal = {Geometry {$\&$} Topology}
}

@article{AroneTurchinHomotopyLK,
     author = {Arone, G. and Turchin, V.},
     title = {Graph-complexes computing the rational homotopy of high dimensional analogues of spaces of long knots},
     journal = {Annales de l'Institut Fourier},
     pages = {1--62},
     publisher = {Association des Annales de l{\textquoteright}institut Fourier},
     volume = {65},
     number = {1},
     year = {2015},
     doi = {10.5802/aif.2924},
     zbl = {1329.57035},
     url = {http://www.numdam.org/articles/10.5802/aif.2924/}
}

@book{DundasGoodwillieMcCarthy,
  doi = {10.1007/978-1-4471-4393-2},
  url = {https://doi.org/10.1007/978-1-4471-4393-2},
  year = {2012},
  publisher = {Springer London},
  author = {Dundas, B. I. and Goodwillie, T. G. and McCarthy, R.},
  title = {The Local Structure of Algebraic {$K$}-Theory}
}

@Inbook{HesselholtWhiteheadS1,
author="Hesselholt, L.",
title="On the Whitehead Spectrum of the Circle",
bookTitle="Algebraic Topology: The Abel Symposium 2007",
year="2009",
publisher="Springer Berlin Heidelberg",
address="Berlin, Heidelberg",
pages="131--184",
url="https://doi.org/10.1007/978-3-642-01200-6_7"
}

@incollection{WaldhausenAtheory,
  doi = {10.1007/bfb0074449},
  url = {https://doi.org/10.1007/bfb0074449},
  year = {1985},
  publisher = {Springer Berlin Heidelberg},
  pages = {318--419},
  author = {Waldhausen, F.},
  title = {Algebraic {$K$}-theory of spaces},
  booktitle = {Lecture Notes in Mathematics}
}

@article{Hudson,
  doi = {10.2307/1970632},
  url = {https://doi.org/10.2307/1970632},
  year = {1970},
  month = may,
  publisher = {{JSTOR}},
  volume = {91},
  number = {3},
  pages = {425},
  author = {Hudson, J. F. P.},
  title = {Concordance,  Isotopy,  and Diffeotopy},
  journal = {The Annals of Mathematics}
}

@article{MilnorWhiteheadTorsion,
    author = "Milnor, J.",
    title = "Whitehead torsion",
    journal = "Bull. Amer. Math. Soc.",
    volume = "72",
    number = "3",
    pages = "358-426",
    year = "1966"
}

@article{IgusaConcordanceStability,
  doi = {10.1007/bf00533643},
  url = {https://doi.org/10.1007/bf00533643},
  year = {1988},
  month = jan,
  publisher = {Portico},
  volume = {2},
  number = {1-2},
  pages = {1--355},
  author = {K. Igusa},
  title = {The stability theorem for smooth pseudoisotopies},
  journal = {{$K$}-Theory}
}

@inbook{GoodwillieKleinWeiss,
title = "Spaces of smooth embeddings, disjunction and surgery",
author = "Goodwillie, T. G. and Klein, J. R. and Weiss, M.",
year = "2001",
publisher = {Princeton University Press},
isbn = "0-6910-8815-2",
pages = "221--284",
booktitle = "In: Surveys on surgery theory, Vol 2 - Ann of Math. Stud, 149, Princeton University Press, Priceton, NJ",
}

@article{BurgheleaFiedII,
  doi = {10.1016/0040-9383(86)90046-7},
  url = {https://doi.org/10.1016/0040-9383(86)90046-7},
  year = {1986},
  publisher = {Elsevier {BV}},
  volume = {25},
  number = {3},
  pages = {303--317},
  author = {D. Burghelea and Z. Fiedorowicz},
  title = {Cyclic homology and algebraic {$K$}-theory of spaces{\textemdash}{II}},
  journal = {Topology}
}

@article{DundasMcCarthy,
  doi = {10.2307/2118621},
  url = {https://doi.org/10.2307/2118621},
  year = {1994},
  month = nov,
  publisher = {{JSTOR}},
  volume = {140},
  number = {3},
  pages = {685},
  author = {Dundas, B. I. and McCarthy, R.},
  title = {Stable {$K$}-Theory and Topological Hochschild Homology},
  journal = {The Annals of Mathematics}
}

@article{BokHsiMadTC,
  doi = {10.1007/bf01231296},
  url = {https://doi.org/10.1007/bf01231296},
  year = {1993},
  month = dec,
  publisher = {Springer Science and Business Media {LLC}},
  volume = {111},
  number = {1},
  pages = {465--539},
  author = {M. Bökstedt and W. C. Hsiang and I. Madsen},
  title = {The cyclotomic trace and algebraic {$K$}-theory of spaces},
  journal = {Inventiones Mathematicae}
}

@misc{RealKtheoryHessMads,
    author = {Hesselholt, L. and Madsen, I.},
    title = {Real Algebraic {$K$}-Theory},
    year = {2016},
    note = {Lecture notes, available at \href{https://web.math.ku.dk/~larsh/papers/s05/book.pdf}{https://web.math.ku.dk/~larsh/papers/s05/book.pdf}}
}

@phdthesis{HøgenhavenRealTC,
  author  = {Høgenhaven, A.},
  title   = {Real topological cyclic homology of spherical group rings},
  school  = "University of Copenhagen",
  year    = "2016"
}

@misc{RealCycTraceMap,
  author = {Y. Harpaz and T. Nikolaus and J. Shah},
  title = {Real topological cyclic homology of spherical group rings},
  year = {2021},
  note = {talk at the conference ``Manifolds and {$K$}-theory: the legacy of Andrew Ranicki''. Slides available at \href{https://www.icms.org.uk/sites/default/files/documents/events/Manifolds\%20and\%20K-theory\%20Y\%20Harpaz\%2022.06.21.pdf}{www.icms.org.uk/sites}}
}

@inproceedings{BudneyIntegralLongKnots,
  doi = {10.2140/gtm.2008.13.41},
  url = {https://doi.org/10.2140/gtm.2008.13.41},
  year = {2008},
  month = feb,
  publisher = {Mathematical Sciences Publishers},
  author = {R. Budney},
  title = {A family of embedding spaces},
  booktitle = {Geometry and Topology Monographs}
}

@misc{SSprogramme,
author =   {D. Chua and H. Chatham and J. Beauvais-Feisthauer and M. Young},
title =    {The Spectral Sequences Project (v0.0.4)},
howpublished = {\url{https://spectralsequences.github.io/sseq/}},
year = {2022},
month = apr
}

@article{TurchinLongKnots1,
author = {Turchin, V.},
title = {Hodge-type decomposition in the homology of long knots},
journal = {Journal of Topology},
volume = {3},
number = {3},
pages = {487-534},
doi = {https://doi.org/10.1112/jtopol/jtq015},
url = {https://londmathsoc.onlinelibrary.wiley.com/doi/abs/10.1112/jtopol/jtq015},
eprint = {https://londmathsoc.onlinelibrary.wiley.com/doi/pdf/10.1112/jtopol/jtq015},
year = {2010}
}

@article{DwyerHess,
author = {W. Dwyer and K. Hess},
title = {{Long knots and maps between operads}},
volume = {16},
journal = {Geometry {$\&$} Topology},
number = {2},
publisher = {MSP},
pages = {919 -- 955},
year = {2012},
doi = {10.2140/gt.2012.16.919},
URL = {https://doi.org/10.2140/gt.2012.16.919}
}

@article{RourkeSanderson,
  doi = {10.1007/bf01402954},
  url = {https://doi.org/10.1007/bf01402954},
  year = {1967},
  month = dec,
  publisher = {Springer Science and Business Media {LLC}},
  volume = {3},
  number = {4},
  pages = {293--299},
  author = {C. P. Rourke and B. J. Sanderson},
  title = {An embedding without a normal microbundle},
  journal = {Inventiones Mathematicae}
}

@article{Haefliger,
  doi = {10.2307/1970475},
  url = {https://doi.org/10.2307/1970475},
  year = {1966},
  month = may,
  publisher = {{JSTOR}},
  volume = {83},
  number = {3},
  pages = {402},
  author = {A. Haefliger},
  title = {Differentiable Embeddings of {$S^n$} in {$S^{n+q}$} for {$q>2$}},
  journal = {The Annals of Mathematics}
}

@article{PedersenAnderson,
  doi = {10.1007/bf01456934},
  url = {https://doi.org/10.1007/bf01456934},
  year = {1983},
  month = mar,
  publisher = {Springer Science and Business Media {LLC}},
  volume = {265},
  number = {1},
  pages = {23--44},
  author = {D. R. Anderson and E. K. Pedersen},
  title = {Semifree topological actions of finite groups on spheres},
  journal = {Mathematische Annalen}
}

@article{TopIsotopyExtensionTheorem,
 ISSN = {0003486X},
 URL = {http://www.jstor.org/stable/1970753},
 author = {R. D. Edwards and R. C. Kirby},
 journal = {Annals of Mathematics},
 number = {1},
 pages = {63--88},
 publisher = {Annals of Mathematics},
 title = {Deformations of Spaces of Imbeddings},
 urldate = {2023-06-10},
 volume = {93},
 year = {1971}
}

@article{Hirschmicrobundle,
  doi = {10.1016/0040-9383(66)90007-3},
  url = {https://doi.org/10.1016/0040-9383(66)90007-3},
  year = {1966},
  month = sep,
  publisher = {Elsevier {BV}},
  volume = {5},
  number = {3},
  pages = {229--240},
  author = {M. W. Hirsch},
  title = {On normal microbundles},
  journal = {Topology}
}

@article{BoavidadeBritoWeissLK,
  doi = {10.1112/topo.12048},
  url = {https://doi.org/10.1112/topo.12048},
  year = {2018},
  month = feb,
  publisher = {Wiley},
  volume = {11},
  number = {1},
  pages = {65--143},
  author = {Boavida de Brito, P. and Weiss, M.},
  title = {Spaces of smooth embeddings and configuration categories},
  journal = {Journal of Topology}
}

@misc{FresseTurchinWillwacher,
  doi = {10.48550/ARXIV.1703.06123},
  url = {https://arxiv.org/abs/1703.06123},
  author = {Fresse,  B. and Turchin,  V. and Willwacher,  T.},
  title = {The rational homotopy of mapping spaces of {$E_n$} operads},
  publisher = {arXiv},
  year = {2017},
  copyright = {arXiv.org perpetual,  non-exclusive license},
note = {arXiv:\href{https://arxiv.org/abs/1703.06123}{1703.06123}}
}

@article{KahnPriddy, title={The transfer and stable homotopy theory}, volume={83}, DOI={10.1017/S0305004100054335}, number={1}, journal={Mathematical Proceedings of the Cambridge Philosophical Society}, publisher={Cambridge University Press}, author={Kahn, D. S. and Priddy, S. B.}, year={1978}, pages={103--111}}

@book{Hirsch1976,
  doi = {10.1007/978-1-4684-9449-5},
  url = {https://doi.org/10.1007/978-1-4684-9449-5},
  year = {1976},
  publisher = {Springer New York},
  author = {Morris W. Hirsch},
  title = {Differential Topology}
}

@article{NikolausScholze,
  doi = {10.4310/acta.2018.v221.n2.a1},
  url = {https://doi.org/10.4310/acta.2018.v221.n2.a1},
  year = {2018},
  publisher = {International Press of Boston},
  volume = {221},
  number = {2},
  pages = {203--409},
  author = {T. Nikolaus and P. Scholze},
  title = {On topological cyclic homology},
  journal = {Acta Mathematica}
}

@book{HatcherSSeq,
  doi = {10.1090/pspum/032.1},
  url = {https://doi.org/10.1090/pspum/032.1},
  year = {1978},
  publisher = {American Mathematical Society},
  author = {Hatcher, A.},
  title = {Algebraic and Geometric Topology: Concordance spaces, higher simple-homotopy theory, and applications}
}

@mastersthesis{MoranLamasThesis,
  author    = {A. Morán Lamas},
  title     = {A guide for the {A}dams spectral sequence: a computational approach and applications
},
  school    = {Pontificia Universidad Católica de Chile},
  year      = {2024},
  address   = {Santiago, Chile},
  note      = {Master's thesis},
}
\bibliographystyle{amsalpha}

\end{document}